\documentclass[11pt]{article} % use larger type; default would be 10pt
\usepackage{amsmath,amsfonts,amssymb}
\usepackage{amsthm}
\usepackage{xfrac}
\usepackage[pdftitle={The dynamics and geometry of free group endomorphisms}, pdfauthor={Jean Pierre Mutanguha}]{hyperref}

\usepackage{color}

  \theoremstyle{plain}
\newtheorem{thm}{Theorem}[subsection]
\newtheorem*{thm*}{Theorem}
\newtheorem*{mthm*}{Main Theorem}
\newtheorem{prop}[thm]{Proposition}
\newtheorem{cor}[thm]{Corollary}
\newtheorem{lem}[thm]{Lemma}
\newtheorem{athm}{Theorem}[section]
\newtheorem{alem}[athm]{Lemma}
\newtheorem*{obs*}{Observation}
\newtheorem*{sum*}{Summary}
\newtheorem*{claim}{Claim}
  \theoremstyle{definition}

\newtheorem{prob}{Problem}
\newtheorem*{qn*}{Question}
  \theoremstyle{remark}
\newtheorem*{rmk}{Remark}

\newcommand{\paper}{paper}
\newcommand{\nonumsec}[1]{\section*{#1}
\addcontentsline{toc}{section}{#1}}
\newcommand{\fakeenv}{} %%% prints the emptystring

{
  \renewcommand{\fakeenv}{#2} %%% So now \fakeenv prints #2
  \theoremstyle{plain}
  \newtheorem*{\fakeenv}{#1~\ref{#2}} %%% so now #2 is the name of a
  \begin{\fakeenv}
}
{
  \end{\fakeenv}
}

\usepackage[utf8]{inputenc} % set input encoding (not needed with XeLaTeX)

\usepackage{geometry} % to change the page dimensions
\usepackage{graphicx} % support the \includegraphics command and options

\title{The dynamics and geometry of free group endomorphisms}
\author{Jean Pierre Mutanguha\thanks{{\it Email:} {\tt \href{mailto:jpmutang@uark.edu}{jpmutang@uark.edu}}, {\it Web address:} {\tt \url{https://mutanguha.com}} \newline Department of Mathematical Sciences, University of Arkansas, Fayetteville, AR}}

\date{}  % Activate to display a given date or no date (if empty),
\begin{document}
\maketitle

\begin{abstract} 
We prove that ascending HNN extensions of free groups are word-hyperbolic if and only if they have no Baumslag-Solitar subgroups. This extends the theorem of Brinkmann that free-by-cyclic groups are word-hyperbolic if and only if they have no free abelian subgroups of rank 2. The paper is split into two independent parts: 

1) We study the dynamics of injective nonsurjective endomorphisms of free groups. We prove a canonical structure theorem that initializes the development of improved relative train tracks for endomorphisms; this structure theorem is of independent interest since it makes many open questions about injective endomorphisms tractable. 

2) As an application of the structure theorem, we are able to (relatively) combine Brinkmann's theorem with our previous work and obtain the main result stated above. In the final section, we further extend the result to HNN extensions of free groups over free factors.
\end{abstract}

\section*{Overview}

Word-hyperbolic groups are groups that act geometrically (properly and cocompactly) on proper $\delta$-hyperbolic spaces and these spaces are the coarse equivalents of negatively curved geometric spaces. Word-hyperbolic groups are an important class of groups introduced by Misha Gromov \cite{Gro87} and one of the fundamental problems in geometric group theory is determining necessary and sufficient algebraic conditions for a group to be word-hyperbolic.
\begin{qn*}[Coarse hyperbolization {\cite[Question~1.1]{BesQns}}] Let $G$ be a group of finite type. If $G$ contains no Baumslag-Solitar subgroups, then must it be word-hyperbolic?
\end{qn*}

Finite type can be thought of as a strengthening of finite presentation for torsion-free groups. Baumslag-Solitar groups are two-generator one-relator groups with the presentation $BS(m,n) = \langle a, t~|~ t^{-1}a^m t = a^n \rangle$ for $m, n \neq 0$. It is now a classical fact that a word-hyperbolic group cannot have subgroups isomorphic to Baumslag-Solitar groups. The question, historically attributed to Gromov, asks if Baumslag-Solitar subgroups are the only essential obstruction to word-hyperbolicity for finite type groups. It has been answered in the affirmative for certain classes of groups: the fundamental groups of closed $3$-manifolds (Thurston~\cite{Thu82, Thu86}, Perelman) and free-by-cyclic groups (Brinkmann~\cite{Bri00}). The main result of this paper extends the latter to a larger class of HNN extensions.

\begin{mthm*}Let $A \le F$ be a free factor of a finite rank free group $F$ and $\phi:A \to F$ be an injective homomorphism. The HNN extension $F*_A$ of $F$ over $A$ and $\phi$  is word-hyperbolic if and only if it contains no $BS(1,n)$ subgroups for $n \ge 1$.
\end{mthm*}

The HNN extension has the presentation $F*_A = \langle F, t ~|~ t^{-1}a t=\phi(a),~\forall a \in A \rangle$. If $A = F = \phi(A)$, then $F*_A = F\rtimes_\phi \mathbb Z$ is a {\bf free-by-cyclic group} and this case was proven by Peter Brinkmann. Our proof of the main theorem above uses Brinkmann's result. If $A=F  \neq \phi(A)$, then $F*_A = F*_\phi$ is a {\bf strictly ascending HNN extension} and coarse hyperbolization for this class of groups was our original motivation. Finally, when $A \neq F$, we more or less reduce this to the case $A = F$.

The case $A = F$ of the main theorem (Theorem~\ref{main}) proved to be difficult since it required an understanding of the dynamics of injective nonsurjective endomorphisms of free groups and no such complete study had been carried out yet. Patrick Reynolds studied the dynamics of irreducible nonsurjective endomorphisms of free groups \cite{Rey11} and we previously used this work to prove an instance of the main theorem where $\phi: F \to F$ is irreducible but not surjective \cite{JPMa, JPMb}. In this case, we showed that the strictly ascending HNN extension $F*_\phi$ is always word-hyperbolic.
Unfortunately, it remained unclear how Reynolds' work could be generalized to all injective nonsurjective endomorphisms and no further progress on the problem had been made.
\medskip

The first part of this paper (Sections~\ref{secStallBCC}--\ref{secRelImm}) is a self-contained systematic study of all injective nonsurjective endomorphisms of free groups and is the main novel contribution of the paper. We hereby present a summary of the important results from these sections. Note that the statements given here are not the complete statements of the cited results.

\begin{sum*}If $\phi:F \to F$ is injective but not surjective, then there is:
\begin{enumerate}
\item a unique maximal proper $[\phi]$-fixed free factor system $\mathcal A$; (Proposition~\ref{maxfixed})
\item a unique minimal proper $[\phi]$-invariant free factor system $\mathcal A^*$ that carries $\mathcal A$ and is fixed under backward iteration, i.e., $\phi^{-1}\cdot \mathcal A^* = \mathcal A^*$; (Proposition~\ref{canonical})
\item a unique expanding $\mathcal A^*$-relative immersion for $\phi$. (Corollary~\ref{uniqueexp})
\end{enumerate}
The free factor systems $\mathcal A$ and $\mathcal A^*$ could be trivial.
\end{sum*}
This means any injective nonsurjective endomorphism $\phi:F \to F$ is induced by a graph map $f:\Gamma \to \Gamma$ such that: 1) some possibly empty or disconnected proper subgraph $\Gamma' \subset \Gamma$ is $f$-invariant and the restriction of $f$ to $\Gamma'$ is a homotopy equivalence; 2) some $f$-invariant proper subgraph $\Gamma'' \supset \Gamma'$ has an $f^n$-image contained in $\Gamma'$ for some $n \ge 1$; and 3) collapsing the lifts of the subgraph $\Gamma''$ in the universal cover $\tilde \Gamma$ induces an {\it expanding immersion} $\bar f:T \to T$ on a simplicial $F$-tree. In particular, if $\Gamma'$ is empty, then $f$ is an expanding graph immersion.
\medskip

In our previous work \cite{JPMa}, we extended a result of Ilya Kapovich \cite{Kap00} and proved the special case of the main theorem with $A = F$ and $\phi:F \to F$ induced by an expanding graph immersion. The expanding immersion was crucial to the proof and so we developed expanding relative immersions with the intent of, in a sense, relativizing our previous proof and extending it to all injective nonsurjective endomorphisms. This application of the structure theorem is carried out in the second part of the paper (Sections~\ref{secPullbacks}--\ref{secHNNExts}).
\medskip

Expanding relative immersions are interesting in their own right as they can potentially be applied to many currently open problems about nonsurjective endomorphisms. To name a few: relative immersions may be needed to extend Hagen-Wise's cubulation of word-hyperbolic free-by-cyclic groups \cite{HW16,HW15} to the groups $F*_A$; we suspect that they could be used to extend the main theorem to a characterization of the possible Dehn/isoperimetric functions of the groups $F*_A$ following work by Bridson-Groves \cite{BG10} (See also Problem~\ref{isoProb}); they provide a framework for generalizing Bestvina-Feighn-Handel's construction of improved relative train tracks for automorphisms \cite{BFH00}, which were used in Brinkmann's, Hagen-Wise's, and Bridson-Groves' results; for a different direction of generalization, the main theorem may be extended to HNN extensions $G*_A$ of torsion-free word-hyperbolic groups $G$ over free factors $A \le G$; and finally, the {\it virtual fibering} question is a particularly interesting problem that was first posed to us by Dawid Kielak and our structure theorem along with the cubulation problem could be significant steps towards its resolution:

\begin{qn*}[Virtual fibering] Suppose $\phi:F \to F$ is injective but not surjective. If $F*_\phi$ is word-hyperbolic, then does it have a free-by-cyclic finite index subgroup? What if it has a quadratic Dehn function?
\end{qn*}

Besides Section~\ref{secHNNExts}, the results in this paper were the author's doctoral thesis. With future applications of the structure theorem in mind, we have chosen to emphasize in this paper the independence of the ``theory'' (first part) and ``application'' (second part). The two parts can be read independently. If you choose to skip the first part, then we encourage you to read the prologue and interlude; these optional sections contain several examples that capture the key points. The main results of the first part are summarized again in the interlude with a bit more detail since we are assuming you will, by then, be familiar with the terms defined in the definitions and conventions section. The summary is all you will need to use expanding relative immersions in your own work. The epilogue is a brief discussion of some questions about endomorphisms in the context of expanding relative immersions.

The main tools used in the first part are Bass-Serre theory, Stallings graphs, bounded cancellation, and train track theory. The second part uses some Bass-Serre theory and the Bestvina-Feighn combination theorem.
\medskip

\noindent \textbf{Acknowledgments:} I am greatly indebted to my advisor Matt Clay. I encountered all the tools used in this paper in our first directed readings, before I even knew what geometric group theory was.

\tableofcontents

%%%%%%%%%%%%%%%%%%%%%%%%%%%%%%%%%%%%%%%%%%%%%%%%%%%%%%%%%%
\nonumsec{Prologue}

We will start with a few motivating examples to illustrate the main construction of the first part of this \paper.
Let $F = F(a,b)$ be the free group on two generators and $\phi, \varphi, \psi:F \to F$ be injective nonsurjective endomorphisms given by 
\[ \phi:(a,b) \mapsto (ab, ba), \quad~
\varphi:(a, b) \mapsto (a, bab^{-1}), \quad \text{and} \quad
\psi:(a,b) \mapsto (a, abab).\]
The {\it standard rose} is a rose $R$ with two petals and an identification $F = \pi_1(R)$ such that the basis $\{a, b\}$ corresponds to the petals; it will be graphically represented by a rose with two oriented petals labelled by $a$ and $b$ respectively. 

\begin{figure}[ht]
 \centering 
 \includegraphics[scale=.9]{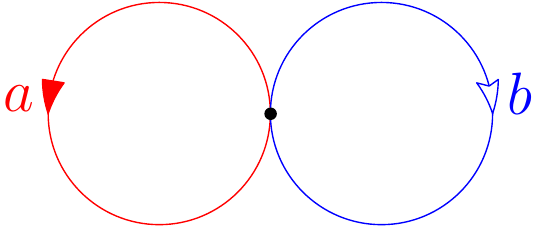}
 \caption{The standard rose.}
 \label{figstd}
\end{figure}

For each integer $k \ge 1$, let $\hat R_k$ be the cover of $R$ corresponding to the subgroup $\phi^k(F)$, i.e., it is the quotient of the universal cover $\tilde R$ by the action of the subgroup $\phi^k(F)$. The Stallings graph $S(\phi^k(F))$ is the core of $\hat R_k$, i.e., the smallest connected subgraph of $\hat R_k$ with rank 2. Alternatively, we can define $\tilde R(\phi^k(F)) \subset \tilde R$ to be the smallest subtree invariant under the $\phi^k(F)$-action and $S(\phi^k(F))$ to be the quotient $\phi^k(F)\backslash \tilde R(\phi^k(F))$. Evidently, there is a natural isomorphism $\phi^k(F) \cong \pi_1(S(\phi^k(F)))$. A Stallings graph $S$ is graphically represented as an $R$-digraph, i.e., $S$ will be an oriented graph whose oriented edges are labelled by $a$ or $b$ as shown in Figure~\ref{figphi}.

\begin{figure}[ht]
 \centering 
 \includegraphics[scale=.9]{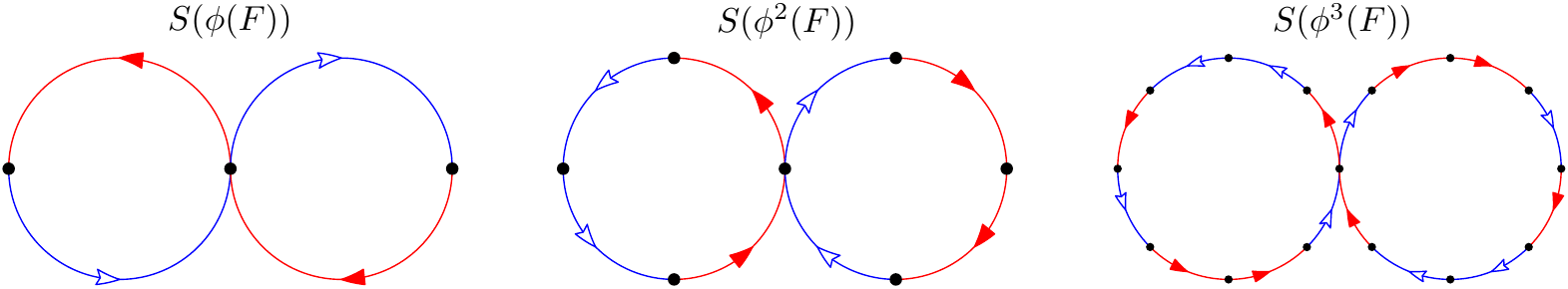}
 \caption{Stallings graphs $S(\phi^k(F))$ with respect to the standard rose for $k = 1,2,3$.}
 \label{figphi}
\end{figure}

The Stallings graphs $S(\phi^k(F))$ are roses for all $k \ge 1$ and each petal doubles in size with each iteration. Since $\phi$ is injective, it is an isomorphism onto its image. In particular, we may use $\phi$ to get an isomorphism $F \cong \phi^k(F) \cong \pi_1(S(\phi^k(F)))$ for $k \ge 1$. 
Under this isomorphism, the basis $\{a, b\}$ corresponds to the petals of $S(\phi^k(F))$ and we recover the standard rose for all $k \ge 1$. This is equivalent to the existence of an immersion (locally injective map) on the rose $f:R \to R$ such that $\pi_1(f) = \phi$ \cite[Lemma~3.2]{JPMb}. Indeed, the obvious map $f:R \to R$ that maps the $a$-petal to the path $ab$ and the $b$-petal to the path $ba$ is an immersion.
The petals of $S(\phi^k(F))$ doubling in size with each iteration implies $f$ is in fact an expanding immersion, i.e., all edges expand under $f$-iteration.
\medskip

\begin{figure}[ht]
 \centering 
 \includegraphics[scale=.9]{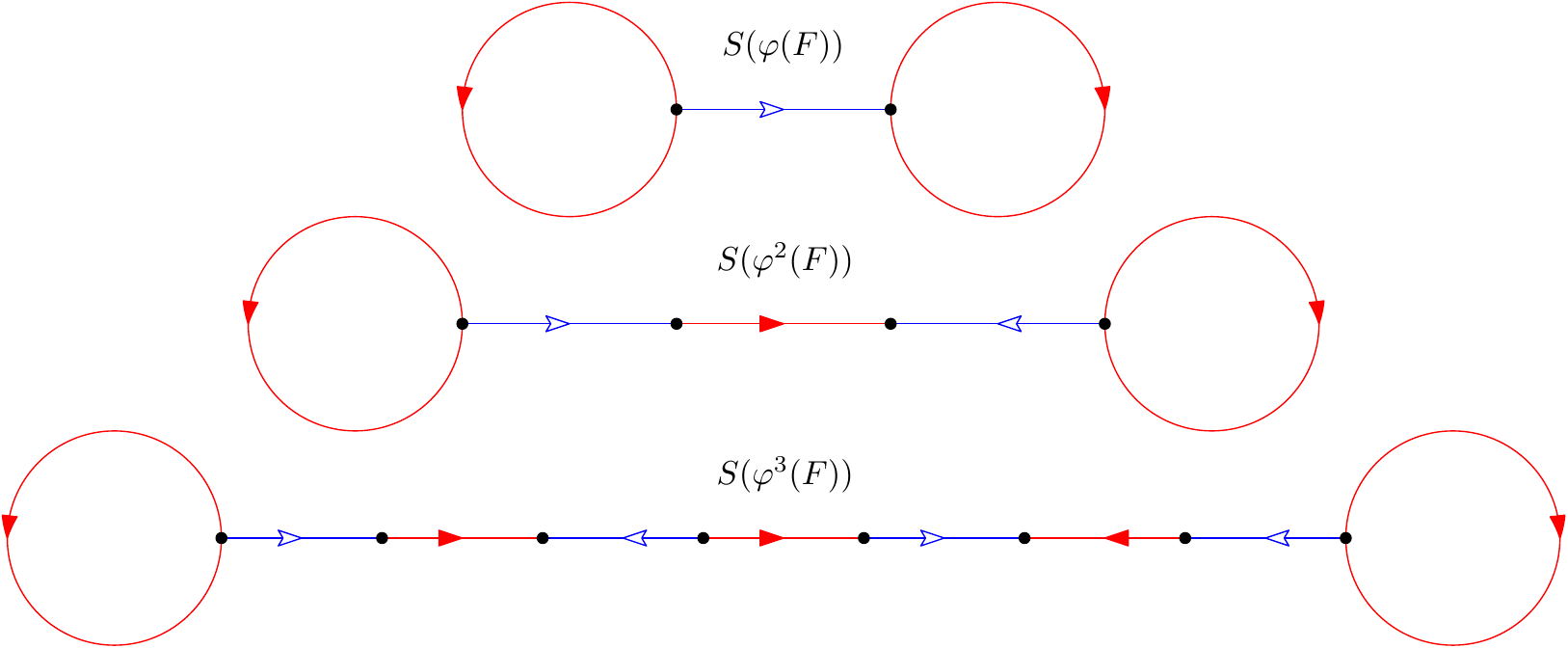}
 \caption{Stallings graphs $S(\varphi^k(F))$ for $k = 1,2,3$.}
 \label{figvarphi}
\end{figure}

The Stallings graphs $S(\varphi^k(F))$ are barbells for all $k \ge 1$, the middle bars roughly double in size with each iteration, and the {\it plates} are labelled by $a$ (See Figure~\ref{figvarphi}). Under the isomorphism $F \cong \varphi^k(F) \cong \pi_1(S(\varphi^k(F)))$, the conjugacy classes $\{[a], [b] \}$ correspond to the two plates of the barbell $S(\varphi^k(F))$ and we recover the so called {\it standard barbell} for all $k \ge 1$. Just as in the first example, this is equivalent to the existence of an immersion on the barbell $g: B \to B$ such that $\pi_1(g) = \varphi$ under some identification $F = \pi_1(B)$. However, the immersion is not expanding since the plates are not growing in size with each iteration.

What we are now about to do will be the crux of the first part of the paper. Since the (nongrowing) plates of $S(\varphi^k(F))$ are all labelled $a$, we deduce that $\langle a \rangle$ is a free factor fixed by $\varphi$ and $\langle b \rangle$ is mapped to a conjugate of $\langle a \rangle$. Thus, $\{ \langle a \rangle, \langle b \rangle \}$ forms an {\it invariant free factor system} that contains a {\it fixed free factor}. Now consider $\tilde g: \tilde B \to \tilde B$ to be the lift of $g$ to the universal cover $\tilde B$. Collapsing all translates of {axes} of $a$ and $b$ in $\tilde B$, i.e., collapsing all edges of $\tilde B$ labelled by the plates of $B$, will produce a so-called $(F, \{\langle a \rangle, \langle b \rangle \})$-tree $T$ with a nonfree $F$-action and $\tilde g$ induces an expanding immersion $\bar g:T \to T$ where each edge doubles in size under $\bar g$-iteration. The map $\bar g$ will be referred to as a {\it relative immersion}.
\medskip

\begin{figure}[ht]
 \centering 
 \includegraphics[scale=1]{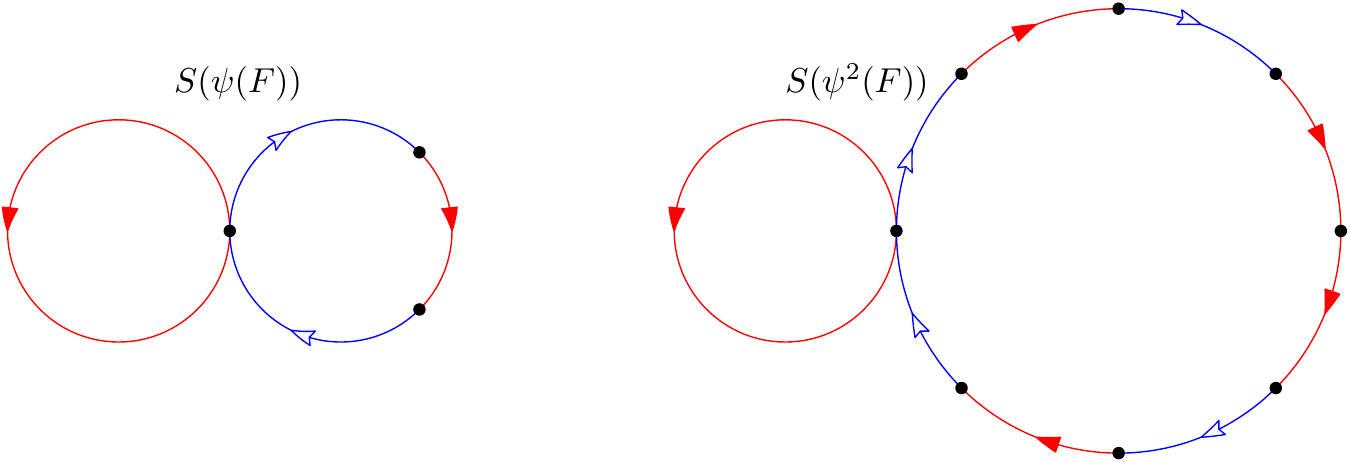}
 \caption{Stallings graphs $S(\psi^k(F))$ for $k = 1,2$.}
 \label{figpsi}
\end{figure}

For the final example, the Stallings graphs $S(\psi^k(F))$ are roses again where one petal roughly doubles in size with each iteration and another is labelled by $a$ (See Figure~\ref{figpsi}). Unlike the first example, the roses with the isomorphisms $F \cong \psi^k(F) \cong \pi_1(S(\psi^k(F)))$ are not standard roses since $b$ does not correspond to a petal for $k \ge 1$. In fact, applying $\psi$ to the labels of $S(\psi^k(F))$ forces a full fold with one petal to get $S(\psi^{k+1}(F))$ which means the Stallings graphs are all distinct {\it marked roses}. In particular, $\psi$ cannot be induced by a graph immersion! But the observation that the nongrowing petals of $S(\psi^k(F))$ are all labelled by $a$ implies $\langle a \rangle$ is a fixed free factor. Despite $\psi$ not being induced by an immersion, the obvious map on the standard rose $h:R \to R$ will induce an expanding immersion $\bar h: Y \to Y$ on an $(F, \{ \langle a \rangle\})$-tree $Y$ where every edge doubles in size. 
\medskip

The main result of the first part of this {\paper} is that this construction always produces expanding relative immersions, Theorem~\ref{expimmersion}: any injective nonsurjective endomorphism of a free group is induced by a graph map $f: \Gamma \to \Gamma$  such that collapsing the translates of axes in $\tilde \Gamma$ of some canonical invariant free factor system containing a fixed free factor system will induce an expanding relative immersion $\bar f: T \to T$. This can be summarized into two steps: first establishing the existence of a unique (possibly trivial) maximal fixed free factor system (Section~\ref{secElliptic}); then showing that collapsing this free factor system and its preimages in the universal cover of an appropriately chosen graph $\Gamma$ will produce a tree $T$ on which we can define an expanding relative immersion (Section~\ref{secRelImm}). The two guiding principles will be {\it bounded cancellation} (Lemmas~\ref{bcl}~{\&}~\ref{relbcl}) and the fact that Stallings graphs for iterated injective nonsurjective endomorphisms have unbounded size (Lemma~\ref{expand}~{\&}~Proposition~\ref{relexpand}).

%%%%%%%%%%%%%%%%%%%%%%%%%%%%%%%%%%%%%%%%%%%%%%%%%%%%%%%%%%
\nonumsec{Definitions and conventions}
%%%%%%%%%%%%%%%%%%%%%%%%%%%%%%%%%%%%%%%%%%%%%%%%%%%%%%%%%%

$F$ will always be a free group with finite rank at least $2$. A {\bf nontrivial subgroup system} of $F$ is a nonempty finite collection of nontrivial subgroups $\mathcal A = \{ A_1, \ldots, A_l \}$ of $F$. %that are in distinct conjugacy classes of subgroups, i.e.,  $[A_i] \neq [A_j]$ for $i \neq j$. 
A {\bf nontrivial free factor system} of $F$ is a nonempty collection of nontrivial free factors $\mathcal A = \{A_1, \ldots, A_l \}$ of $F$ such that $A_i \cap A_j$ is trivial if $i \neq j$ and $\langle A_1, \ldots, A_l \rangle$ is free factor of $F$. %One can check that nontrivial free factor systems are nontrivial subgroup systems. 
We define the {\bf trivial system} to be the collection consisting of the trivial subgroup; this will allow us to treat {\it absolute} and {\it relative} cases simultaneously in our proofs. For any subgroup system $\mathcal A$, the subgroups $A_i \in \mathcal A$ are called the {\bf components of $\boldsymbol{\mathcal A}$} or {\bf $\boldsymbol{\mathcal A}$-components}. A subgroup system $\mathcal A$ is finitely generated if its components are finitely generated. A free factor system $\mathcal A$ is {\bf proper} if some component is proper, i.e., $\mathcal A \neq \{ F \}$. Let $\mathcal A$ and $\mathcal B$ be subgroup systems of $F$. We shall say {\bf $\boldsymbol{\mathcal A}$  carries $\boldsymbol{\mathcal B}$} if for every component $B \in \mathcal B$, there is a component $A \in \mathcal A$ and element $x \in F$ such that $B \le x A x^{-1}$ and, when $\mathcal A$ and $\mathcal B$ are free factor systems, we shall denoted it by $\mathcal B \preceq \mathcal A$. For a conjugacy class of elements $[g]$ in $F$, we shall also say {\bf $\boldsymbol{\mathcal A}$ carries $\boldsymbol{[g]}$} if there is a component $A \in \mathcal A$  and element $x \in F$ such that $g \in x A x^{-1}$. If $\mathcal A=\{ A \}$ (or $\mathcal B = \{ B \}$) is a singleton, then we will write, ``$A$ carries $[g]\,/\,\mathcal B$ (or $B$).'' One can easily verify that the $\preceq$-relation on free factor systems is a preorder, i.e., it is reflexive and transitive. So it determines an equivalence relation on the set of free factor systems and a partial order on the set of equivalence classes. Free factor systems will always be considered up to this equivalence relation. In particular, we can replace components in a system with conjugates whenever convenient.

\begin{rmk} Suppose $F=\langle a, b \rangle$ be the free group of rank 2. With our definition, $\{ \langle a \rangle, \langle b \rangle \}$ and $\{ \langle\, bab^{-1} \,\rangle, \langle\, (ba)b(ba)^{-1} \,\rangle \}$ are equivalent free factor systems but $\{ \langle a \rangle, \langle\, (ba)b(ba)^{-1} \,\rangle \}$ is not a free factor system. Thus, when we replace components in a free factor system with conjugates, we need to ensure the resulting subgroup system is still a free factor system. \end{rmk}

Two group homomorphisms $h_1, h_2:A \to B$ are equivalent if there is an inner automorphism $i_b: B \to B$ such that $h_2 = i_b \circ h_1$ and this is denoted by $[h_1] = [h_2]$. Outer endomorphisms of $F$ are the equivalence classes on the set of endomorphisms of $F$. For instance, properties such as being irreducible are not just properties of an endomorphism $\phi$ but also its outer class $[\phi]$. 

Let $\phi: F \to F$ be an endomorphism. We say a subgroup system $\mathcal A = \{A_1, \ldots, A_l \}$ is {\bf $\boldsymbol{[\phi]}$-invariant} if $\mathcal A$ carries the subgroup system $\phi(\mathcal A) = \{ \phi(A_1), \ldots, \phi(A_l) \}$, i.e, there exists a set of elements $\{x_1, \ldots, x_l\} \subset F$ and a function $\sigma: \{1, \ldots, l\} \to \{1, \ldots, l\}$ such that $\phi(A_i) \le x_i A_{\sigma(i)} x_i^{-1}$ for all $i$. A $[\phi]$-invariant subgroup system $\mathcal A$ is {\bf $\boldsymbol{[\phi]}$-fixed} if its subgroups are permuted up to conjugacy, i.e., $\sigma$ is a permutation and $[\phi(A_i)] = [A_{\sigma(i)}]$ for all $i$. 
When $\phi$ is an automorphism, all $[\phi]$-invariant free factor systems are $[\phi]$-fixed. When $\phi$ is injective, then, for any $k \ge 1$  and free factor system $\mathcal A$ of $F$, $\phi^k(\mathcal A)$ is a free factor system of $\phi^k(F)$; conversely, any free factor system of $\phi^k(F)$ gives us a free factor system of $F$ via the isomorphism $\phi^k:F \to \phi^k(F)$. Finally, note that when $\mathcal A$ is a $[\phi]$-fixed free factor system but $\phi$ is not surjective, $\phi^k(\mathcal A)$ will typically not be a free factor system of $F$ (See the remark above). A subgroup $A \le F$ is {\bf eventually $\boldsymbol{[\phi]}$-periodic} if %there is an element $x \in F$ and integers $m > n \ge 1$ such that 
$[\phi^m(A)] = [\phi^n(A)]$ for some $m > n \ge 1$, and it is {\bf $\boldsymbol{[\phi]}$-periodic} if $[\phi^m(A)] = [A]$ for some $m \ge 1$.

The endomorphism $\phi$ is reducible if it has a nontrivial proper invariant free factor system and {\bf irreducible} otherwise. 
%It is {\bf fully irreducible} if all its iterates are irreducible; equivalently, $\phi$ is fully irreducible if there is no proper free factor $A \le F$, element $x \in F$, and integer $k \ge 1$ such that $\phi^k(A) \le x A x^{-1}$. 
One immediate consequence of Stallings folds~\cite{St83} is the injectivity of irreducible endomorphisms (Observation below). So we can drop the injectivity hypothesis when specializing results to the irreducible case.
%{\color{red}An (outer) endomorphism $\phi:F \to F$ has {\bf finite order} if there is an integer $n \ge 1$ such that $[\phi^n] =[id_F]$, where $id_F$ is the identity automorphism on $F$.}
\medskip

For the topological perspective, graphs are 1-dimensional CW-complexes and a graph map $f:X \to X'$ will be a continuous map of graphs that sends vertices to vertices and any edge to a vertex or immersed path. 
%For simplicity, we shall require that the vertices of $\Gamma$ be at least trivalent. 
%Since a graph map $f: \Gamma \to \Gamma'$ is allowed to send edges to vertices, a
An edge $e$ of $X$ is {\bf pretrivial} if $f(e)$ is a vertex. %If $\Gamma' = \Gamma$, then an edge $e$ is pretrivial if $f^n(e)$ is a vertex for some $n \ge 1$.
%Typically, we will reserve {\it graphs} for (finite) noncontractible graphs and {\it trees} for (infinite) simply-connected graphs.
For a graph map $f:\Gamma \to \Gamma'$ of finite graphs, let $K$ be the maximum of the (combinatorial) length of the edge-path $f(e)$ as $e$ varies over all the edges of $\Gamma$. Then $f$ is $K$-Lipschitz, a fact that will be used throughout the \paper. Generally, $K(f)$ will denote some convenient Lipschitz constant for $f$ rather than the infimum. For the moment, $X$ is used for arbitrary graphs but, for most of the paper, $\Gamma$ will be used for finite connected noncontractible graphs and $T$ for infinite simply-connected graphs (simplicial trees). A {\bf core graph} is a graph with no proper deformation retract.

A {\bf direction at a vertex} $v \in X$ is a {\it half-edge} attached to the vertex. A vertex is {\bf bivalent} if it has exactly two directions. The set of directions at vertex $v$ is denoted by $T_vX$. {\bf branch points} are vertices with at least three directions and {\bf natural edges} are maximal edge-paths whose interior vertices are bivalent. We will say a graph map is {\bf natural} if it maps branch points to branch points and any natural edge to a branch point or immersed path. If the graph map $f:X \to X'$ has no pretrivial edges, then the restriction to initial segments induces the {\it derivative map} at $v$, $df_v: T_vX \to T_{f(v)}X'$. %We can define equivalence on $T_v\Gamma$ by $e \sim e'$ if $df_v(e) = df_v(e')$. The equivalence classes will be known as {\bf gates}. 
The graph map $f$ is an {\bf immersion} if it is locally injective, i.e., it has no pretrivial edges and the derivative maps $df_v$ are injective for all vertices $v$; note that immersions are natural. An immersion $f$ is {\bf expanding} if the length of $f^n(e)$ is unbounded as $n\to \infty$ for every edge $e$ of $X$. 

With immersions defined, we preface the observation with a summary of Stallings' {\it folding theorem}. Let $R$ be a rose whose edges are indexed by a basis $\{a_1, \ldots, a_n\}$ of $F$ and let $f:R  \to R$ be the map where $f(a_i)$ is the immersed edge-path in $R$ labelled by $\phi(a_i)$. Stallings \cite{St83} showed that $f$ factors as $\iota \circ f_l \circ \cdots \circ f_1$ where each $f_i$ is a {\it fold} and $\iota$ is an immersion. We will use this factorization again to prove bounded cancellation in Lemma~\ref{bcl} and Proposition~\ref{relbcl}.

\begin{obs*} If $\phi: F \to F$ is irreducible, then it is injective.\end{obs*}
\begin{proof}If $\phi$ is not injective, then at least one of the folds in Stallings' factorization of $\phi$ {\it collapses} a subgraph of the domain. In particular, the kernel of $\phi$ contains a proper free factor $A\le F$; therefore, $\phi(A) = \{ 1 \} \le A$ and $[\phi]$ is reducible.
\end{proof}

Let $\mathcal A = \{A_1, \ldots, A_l\}$ be a nontrivial subgroup system of $F$, then an {\bf $\boldsymbol{\mathcal A}$-marked graph} $(\Gamma_{*}, \alpha_{*})$ is a collection of graphs $\Gamma_{*} = \{ \Gamma_1, \ldots, \Gamma_l \}$ where the finite connected core graphs $\Gamma_i$ are indexed by {\bf markings}, i.e., isomorphisms $\alpha_i:A_i\to \pi_1(\Gamma_i)$. 
When $\mathcal A = \{ A \}$ is a singleton, we may write $A$-marked graph in place of $\mathcal A$-marked graph. A {\bf marked graph} $(\Gamma, \alpha)$ is an $F$-marked graph.
For any marked graph $(\Gamma, \alpha)$ and any conjugacy class of a nontrivial element $g \in F$, denoted by $[g]$, we define its {\bf length with respect to $\boldsymbol{\alpha}$}, $\lVert g \rVert_\alpha$, to be the (combinatorial) length of the immersed loop in $\Gamma$ representing $[g]$. 

More generally, we want to consider pairs $\mathcal A \preceq \mathcal B$ of free factor systems of $F$. For each component $B_i \in \mathcal B$, let $\mathcal A_i$ be either the nonempty maximal subset of $\mathcal A$ carried by $B_i$ or the trivial system if no such subset exists. Typically, we shall replace components of $\mathcal A$ with conjugates so that $\mathcal A$ is also a free factor system of $\mathcal B$, i.e., each $\mathcal A_i$ is a free factor system of $B_i$ (next remark below). A {\bf $\boldsymbol{(\mathcal B, \mathcal A)}$-forest} $T_{*}$ is a simplicial forest of $(B_i, \mathcal A_i)$-trees $T_i$, i.e., a collection of simplicial trees $T_* = \{ T_1, \ldots, T_k \}$ where each tree $T_i$ is equipped with a minimal simplicial $B_i$-action whose edge stabilizers are trivial and point stabilizers are trivial or conjugates (in $B_i$) of $\mathcal A_i$-components. We note that a $(B, \{B\})$-tree is a point also known as a {\bf degenerate tree}. 

%An $(F, \mathcal A)$-tree is a cellular tree $T$ with no bivalent vertices along with a minimal $F$-action by isometries whose edge stabilizers are trivial and point stabilizers are trivial or conjugates of the free factors in $\mathcal A$. 
When $\mathcal A$ is the trivial system, an $(F, \mathcal A)$-tree is a free minimal $F$-tree $T$. In that case, the quotient $F\backslash T$ is a marked graph. When $\mathcal A$ is a nontrivial proper free factor system, the quotient of an $(F, \mathcal A)$-tree is a {\it graph of groups} decomposition of $F$ with trivial edge groups and $\mathcal A$ as the nontrivial vertex groups \cite{Ser77}; such decompositions are sometimes known as {\bf free splittings of $\boldsymbol F$} and they will be our relative analogues for marked graphs. 
Any given $(F, \mathcal A)$-tree endowed with the combinatorial metric has an associated length function $l_T:F \to \mathbb R$. Precisely, any isometry of a simplicial tree is either {\bf elliptic} (fixes a point) or {\bf loxodromic} (preserves an {\it axis} of least displacement). If $g \in F$ is elliptic, then $l_T(g)=0$; otherwise, $l_T(g) > 0$ is the translation distance of $g$ acting on its axis. When $\mathcal A$ is trivial and $F\backslash T$ is the marked graph $(\Gamma, \alpha)$, then $l_T(\cdot) = \|\cdot \|_\alpha$ as functions $F \to \mathbb R$.
The {\bf minimal subtree} $T(H) \subset T$ of a nontrivial subgroup $H \le F$ is the smallest subtree with a minimal induced $H$-action, i.e., the union of all fixed points and axes of nontrivial elements in $H$.

\begin{rmk} Let $\mathcal A \preceq \mathcal B$ be free factor systems of $F$, $B_i \in \mathcal B$ be a free factor with nontrivial $\mathcal A_i$, and $T$ be an $(F, \mathcal A)$-tree. Choose a connected fundamental domain of $B_i$ acting on $T(B_i)$ and the nontrivial stabilizers in $B_i$ of vertices in the domain form a free factor system $\mathcal A_i'$ of $B_i$. Then $\mathcal A_i'$ and $\mathcal A_i$ are equivalent as free factor systems of $F$ since $\mathcal A_i \preceq \{ B_i \}$.
\end{rmk}

Suppose $\psi:F \to F'$ is an injective homomorphism. Then it determines a contravariant preimage function $\psi^{-1}$ from the poset of free factor systems of $F'$ to the poset of free factor systems of $F$. One way to define this function is through trees. Let $\mathcal A'$ be a free factor system of $F'$, $T$ be any $(F', \mathcal A')$-tree, and $T(\psi(F)) \subset T$ be the {minimal subtree} of $\psi(F) \le F'$. In particular, the quotient $\psi(F)\backslash T(\psi(F))$ is a free splitting of $\psi(F)$ whose nontrivial vertex groups form a nontrivial free factor system $\mathcal A''$ of $\psi(F)$. If $T(\psi(F))$ has a free $\psi(F)$-action, then this condition is independent of the choice of $(F', \mathcal A')$-tree $T$ and we set $\psi^{-1}\cdot \mathcal A'$ to be the trivial system. Otherwise, we have a nontrivial free factor system $\mathcal A''$ of $\psi(F)$ and a corresponding nontrivial free factor system of $F$ via the isomorphism $\psi:F \to \psi(F)$; the equivalence class of the latter free factor system is independent of the $(F', \mathcal A')$-tree $T$ and we will denote it by $\psi^{-1}\cdot \mathcal A'$. %Note that if $F' = F$ and $\mathcal A'$ is $[\psi]$-invariant, then $\psi^{-1}\cdot \mathcal A' \succeq \mathcal A'$.

Let $T$ and $T'$ be $(F, \mathcal A)$- and $(F',\mathcal A')$-trees respectively and $\psi: F \to F'$ be an injective homomorphism such that $\mathcal A'$ carries $\psi(\mathcal A)$ or, equivalently, $\psi^{-1} \cdot \mathcal A' \succeq \mathcal A$. This {\it carrying condition} ensures that elliptic elements in $F$ (with respect to $T$) have elliptic $\psi$-images (with respect to $T'$). A tree map $f: T \to T'$ is {\bf $\boldsymbol{\psi}$-equivariant} if $f(g \cdot x) = \psi(g) \cdot f(x)$ for all $g \in F$ and $x \in T$. If we require $\psi^{-1}\cdot \mathcal A' = \mathcal A$, then loxodromic elements in $F$ have loxodromic $\psi$-images; this is the relative analogue of a $\pi_1$-injective graph map.
\medskip

Suppose $\mathcal A$ is a nontrivial $[\phi]$-invariant free factor system of $F$ for some injective endomorphism $\phi:F \to F$; so there is a function $\sigma:\{1, \ldots, l\} \to \{1, \ldots, l\}$ and inner automorphisms $i_{x_i}: F \to F$ such that the restrictions $\phi_i = \left.\left(i_{x_i} \circ \phi\right)\right|_{A_i}$ are homomorphisms $\phi_i: A_i \to A_{\sigma(i)}$. The collection $ \{ \phi_i \}$ is a {\bf restriction of $\boldsymbol{\phi}$ to $\boldsymbol{\mathcal A}$}, which we denote by $\left.\phi\right|_{\mathcal A}:\mathcal A \to \mathcal A$.
A {\bf weak representative for $\boldsymbol{[\phi_i]}$} is a graph map $f_i : \Gamma_i \to \Gamma_{\sigma(i)}$ from an $A_i\,$-marked graph $(\Gamma_i, \alpha_i)$ to an $A_{\sigma(i)}\,$-marked graph $(\Gamma_{\sigma(i)}, \alpha_{\sigma(i)})$ such that $[\pi_1(f_i)\circ \alpha_i] = [\alpha_{\sigma(i)} \circ \phi_i]$ for all $i \in \{1, \ldots, l\}$. A {\bf weak representative for $\boldsymbol{ \left[\left.\phi\right|_{\mathcal A}\,\right]}$} is a graph map  $f_{*}:\Gamma_{*} \to \Gamma_{*}$ of an $\mathcal A$-marked graph $(\Gamma_{*}, \alpha_{*})$ that is a disjoint union of weak representatives for $[\phi_i]$. A {\bf natural representative} is a weak representative that is also natural. A {\bf topological representative} is a weak representative that has no pretrivial edges and whose underlying graph $\Gamma_{*}$ has no bivalent vertices.
%{\color{red}A topological representative $f_{*}$ is said to be {\bf minimal} if it has no invariant subforests (disjoint union of trees/simply-connected graphs).}

%Choosing a marked graph $(R, \alpha)$ where $R$ is a rose is equivalent to choosing a basis $\mathcal B$ for $F$ determined by the pre-images of the (oriented) petals over the marking $\alpha$.

%When $(R, \alpha)$ is a marked rose corresponding to a basis $\mathcal B$, we shall denote the respective length of $[g]$ by $\lVert g \rVert_{\mathcal B}$.

Recall $\mathcal A \preceq \mathcal B$. Suppose $\phi^{-1}\cdot \mathcal A = \mathcal A$, $\mathcal B$ is $[\phi]$-invariant, and let $\sigma':\{1, \ldots k\} \to \{1, \ldots k\}$ be the function corresponding to the $[\phi]$-invariance of $\mathcal B$ and $\left.\phi\right|_{\mathcal B} = \{ \phi_i:B_i \to B_{\sigma'(i)}\}$ be a restriction of $\phi$. %A {\bf $\boldsymbol{\left.\phi\right|_{\mathcal B}}$-equivariant graph map} $f_{*}: T_{*} \to T_{*}$ is a collection of $\phi_i$-equivariant maps $f_i : T_i \to T_{\sigma(i)}$. A {\bf relative (topological) representative} for $\psi$ is a $\psi$-equivariant graph map $f: T \to T'$ with no pretrivial edges.
An {\bf $\boldsymbol{\mathcal A}$-relative weak representative for $\boldsymbol{\left.\phi\right|_{\mathcal B}}$} is a forest map $f_{*} : T_{*} \to T_{*}$ of a $(\mathcal B, \mathcal A)$-forest $T_*$ that is a disjoint union of $\phi_i$-equivariant tree maps $f_i:T_i \to T_{\sigma'(i)}$. An {\bf $\boldsymbol{\mathcal A}$-relative natural representative} is an ${\mathcal A}$-relative weak representative that is also natural. An {\bf $\boldsymbol{\mathcal A}$-relative representative} is an ${\mathcal A}$-relative weak representative with no pretrivial edges and whose underlying forest has no bivalent vertices. A relative representative is {\bf minimal} if it has no {\it orbit-closed} invariant subforests whose components are bounded. An {\bf (expanding resp.) $\boldsymbol{\mathcal A}$-relative immersion for $\boldsymbol{\left.\phi\right|_{\mathcal B}}$} is an $\mathcal A$-relative representative for $\left.\phi\right|_{\mathcal B}$ that is also an (expanding resp.) immersion. %An $\mathcal A$-relative immersion $f_*$ for $\left.\phi\right|_{\mathcal B}$ is {\bf expanding} if every edge expands under $f_{*}$-iteration.
\medskip

In the first part of the paper, we show that any injective endomorphism of $F$ has a canonical invariant free factor system $\mathcal A$ and an expanding $\mathcal A$-relative immersion $f$ such that $\mathcal A$ is eventually mapped into a canonical fixed free factor system.
In the second part of the paper, we use the canonical systems and the expanding relative immersion to prove the main theorem.

%%%%%%%%%%%%%%%%%%%%%%%%%%%%%%%%%%%%%%%%%%%%%%%%%%%%%%%%%%
\nonumsec{Dynamics of free group endomorphisms}
\subsection{Stallings graphs and bounded cancellation}\label{secStallBCC}
%%%%%%%%%%%%%%%%%%%%%%%%%%%%%%%%%%%%%%%%%%%%%%%%%%%%%%%%%%

Let $(\Gamma, \alpha)$ be a marked graph with no bivalent vertices. For any nontrivial subgroup system $\mathcal H$ of $F$, the {\bf Stallings (subgroup) graph} for $\mathcal H$ with respect to a marked graph $(\Gamma, \alpha)$ with is the smallest $\mathcal H$-marked graph $(S(\mathcal H), \beta_*)$ along with an immersion $\iota_*: S(\mathcal H) \to \Gamma$ such that $[\pi_1(\iota_i) \circ \beta_i] = \left[\left. \alpha\right|_{H_i}\right]$ for every $H_i \in \mathcal H$. 
Alternatively, $S(\mathcal H)$ is the collection of {\bf cores} $S(H_i)$ of the covers $\hat \Gamma_{H_i}$ of $\Gamma$ corresponding to $\alpha(H_i)$, i.e., the smallest deformation retract of $\hat \Gamma_{H_i}$, and $\iota_i$ is the restriction to $S(H_i)$ of the covering map $\hat \Gamma_{H_i} \to \Gamma$. We may sometimes refer to the marked graph $(\Gamma, \alpha)$ as the {\bf ambient graph}. In the notation, the marking and immersion for Stallings graphs will usually be omitted. If $H$ and $H'$ are in the same conjugacy class, $[H]$, then there is a homeomorphism $h: S(H) \to S(H')$ such that $\iota = \iota' \circ h$. The converse holds as well. So the Stallings graph $S[H]$ is uniquely determined by the conjugacy class $[H]$.  Furthermore, the Stallings graph $S[\mathcal H]$ is a finite graph if and only if $\mathcal H$ is finitely generated.
Suppose $\phi:F \to F$ is injective, $\mathcal A$ is a nontrivial free factor system of $F$. We will be studying the {\it (iterated)} Stallings graphs $(S[\phi^k(\mathcal A)], \beta_{*,k})$ for $k \ge 1$. %We define $S[\phi^k(\mathcal A)]$ to be the disjoint union of the Stallings graphs of the free factors of $\phi^k(F)$ in $\phi^k(\mathcal A)$. 
%By definition, $S[\phi^k(\mathcal A)]$ is a $\phi^k(\mathcal A)$-marked graph with the markings
%\(\displaystyle \beta_{*} = \{ \beta_i: A_i \to \pi_1(S_i) ~|~ S_i= S(\phi^k(A_i)) \text{ is a component of } S[\phi^k(\mathcal A)]\, \}. \)
%Identify the fundamental group $\pi_1(S_i)$ of a component $S_i \subset S[\phi^k(\mathcal A)]$ with $\alpha(\phi^k(A_i))$ so that the corresponding immersion $\iota_i: S_i \to \Gamma$ induces the inclusion homomorphism on fundamental groups. By injectivity of $\phi$, $\left(S[\phi^k(\mathcal A)], \left.\alpha \circ \phi^k\right|_{A_{*}}\right)$ is an $\mathcal A$-marked graph.

\begin{rmk}For any free factor system $\mathcal F$ of $F$ and nontrivial subgroup system $\mathcal H$ carried by $\mathcal F$, we can similarly define the Stallings graph for $\mathcal H$ with respect to $\mathcal F$-marked graphs $(\Gamma_*, \alpha_*)$. If $\phi_*:\mathcal F \to \mathcal F$ is an injective endomorphism, i.e., a collection of injective endomorphism $\{\phi_i: F_i \to F_{\sigma(i)}\}$ that need not be a restriction of an endomorphism of $F$, then we can still consider the Stallings graphs $S[\phi^k_*(\mathcal A)]$ for $k \ge 1$ and $\mathcal A \preceq \mathcal F$. The point of this remark is that the results of this paper hold when $\phi:F \to F$ is replaced with $\phi_*:\mathcal F \to \mathcal F$. In other words, the ambient graph $\Gamma$ need not be connected for our results.
\end{rmk}

Unlike the  ambient graph $\Gamma$, we allow Stallings graphs $S =  S[\mathcal H]$ to have bivalent vertices. More precisely, we subdivide $S$ so that the immersion $\iota_* :S \to \Gamma$ is simplicial, i.e., maps edges to edges.
With this subdivision, we get a combinatorial metric on $(S, \beta_*)$ that is compatible with $(\Gamma, \alpha)$, i.e., for any nontrivial element $g$ in $H_i$, $\lVert g \rVert_\alpha = \lVert g \rVert_{\beta_i}$. 

\begin{lem}\label{expand}Let $\phi:F \to F$ be injective and $H$ be a finitely generated nontrivial subgroup of $F$. If $H$ is not eventually $[\phi]$-periodic, then the length of the longest natural edge in $S[\phi^k(H)]$ is unbounded as $k \to \infty$.
\end{lem}
\begin{proof} 
Suppose the length of the longest natural edge in $S[\phi^k(H)]$ with respect to some marked graph $(\Gamma, \alpha)$ was uniformly bounded as $k \to \infty$. We want to show that $H$ is eventually $[\phi]$-periodic. The number of natural edges in $S[\phi^k(H)]$ is bounded above by $3 \cdot \mathrm{rank}(H)-3$. Our assumption implies there is a bound on the volume of (number of edges in) the graphs $S[\phi^k(H)]$ as $k\to \infty$. So the sequence $S[\phi^k(H)]$ is eventually periodic, i.e., there are integers $m > n \ge 1$ and an isometry $h: S[\phi^m(H)] \to S[\phi^n(H)]$ such that $\iota_m  = \iota_n \circ h$. Since a Stallings graph determines the conjugacy class of its defining subgroup, we have $[\phi^m(H)] = [\phi^n(H)]$, i.e., $H$ is eventually $[\phi]$-periodic.
\end{proof}

This lemma will be invoked on invariant free factor systems containing a component that is not eventually $[\phi]$-periodic. Conversely, the next lemma handles the case when an invariant free factor system consists entirely of eventually periodic free factors.

\begin{lem}\label{nonexpand}Let $\phi:F \to F$ be injective and $\mathcal A$ be a nontrivial $[\phi]$-invariant free factor system. If all components in $\mathcal A$ are eventually $[\phi]$-periodic, then some nonempty subset $\mathcal B \subset \mathcal A$ is a $[\phi]$-fixed free factor system and $\phi^k(\mathcal A)$ is carried by $\mathcal B$ for some $k \ge 0$.
\end{lem}
\noindent E.g., if $\phi:F \to F$ is injective and $F$ is eventually $[\phi]$-periodic, then $\phi$ is an automorphism.
\begin{proof}
Let $\sigma:\{1, \ldots, l\} \to \{1, \ldots, l\}$ be the function used to define the $[\phi]$-invariance of $\mathcal A = \{A_1, \ldots, A_l\}$. Then there is a nonempty subset $J \subset \{1, \ldots, l\}$ on which $\sigma$ acts as a bijection and $\sigma^l(\{1, \ldots, l\}) = J$. Let $\mathcal B \preceq \mathcal A$ by the nontrivial $[\phi]$-invariant free factor system corresponding to $J$. Then $\phi^l(\mathcal A)$ is carried by $\mathcal B$ since the image of ${\sigma\,}^l$ is $J$. It remains to show that $\mathcal B$ is $[\phi]$-fixed. Set $j$ to be the order of $\left. \sigma\right|_J$, fix $B \in \mathcal B$, and let $i_x:F \to F$ be the inner automorphism such that $i_x \circ \phi^j(B) \le B$. Define $\psi = i_x \circ \phi^j$. As $B$ is eventually $[\phi]$-periodic and hence eventually $\phi^j$-periodic, there are integers $m > n \ge 1 $ such that $\displaystyle [\psi^m(B)] = [\phi^{jm}(B)] = [\phi^{jn}(B)] = [\psi^n(B)].$ Therefore, there is an element $y \in F$ such that $y \psi^n(B) y^{-1} = \psi^m(B) \le \psi^n(B)$. 
But no finitely generated subgroup of $F$ is conjugate to a proper subgroup of itself (Lemma~\ref{fgconj} below). So $\psi^m(B) = \psi^n(B)$ and, by injectivity of $\phi$, $\psi(B) = B$. Since $B \in \mathcal B$ was arbitrary, $\phi^j$ fixes the free factors of $\mathcal B$ up to conjugation; as $\mathcal B$ is $[\phi]$-invariant, it must be $[\phi]$-fixed.
\end{proof}

The following fact will be used again in the proof of Proposition~\ref{canonical}.

\begin{lem}\label{fgconj} No finitely generated subgroup of $F$ is conjugate to a proper subgroup of itself.
\end{lem}
\begin{proof}By Marshall Hall's theorem, free groups are {\it subgroup separable/locally extended residual finiteness (LERF)}, i.e., for any finitely generated subgroup $H \le F$ and element $g \in F \setminus H$, there is a finite group $G$ and a surjective homomorphism $\varphi: F \to G$ such that $\varphi(g) \notin \varphi(H)$. See \cite{St83} for a proof due to Stallings.

For a contradiction, suppose there is an element $y \in F$ such that $y H y^{-1} \le H$ and $g \in H \setminus y H y^{-1}$. By subgroup separability, there is a finite group $G$ and homomorphism $\varphi: F \to G$ such that $\varphi(g) \notin \varphi(y H y^{-1})$. But $g \in H$ implies $\varphi(g) \in \varphi(H)$ and, by finiteness of $G$, $y H y^{-1} \leq H$ implies $\varphi(y H y^{-1}) = \varphi(H)$ --- a contradiction.
\end{proof}

The next lemma, also known as the {\bf Bounded Cancellation Lemma}, will be used extensively in this \paper. At the risk of overloading notation, for an edge-path $p$ in a graph $\Gamma$, $[p]$ denotes the immersed edge-path that is homotopic to $p$ rel. endpoints; for a loop $\rho$ in $\Gamma$, $[\rho]$ will be the immersed loop that is freely homotopic to $\rho$. 

\begin{lem}[Bounded Cancellation]\label{bcl} Let $g:\Gamma \to \Gamma'$ be a $\pi_1$-injective graph map. Then there is a constant $C(g)$ such that, for any natural edge-path decomposition $p_1 \cdot p_2$ of an immersed path in the universal cover $\tilde \Gamma$, the edge-path $[\tilde g(p_1)] \cdot [\tilde g(p_2)]$ is contained in the $C(g)$-neighborhood of $[\tilde g(p_1) \cdot \tilde g(p_2)]$.
\end{lem}
\noindent The following proof is due to Bestvina-Feighn-Handel \cite[Lemma~3.1]{BFH97}.
\begin{proof} Any graph map $g:\Gamma \to \Gamma'$ factors as a pretrivial edge collapse and edge subdivision $g_0$, a composition of $r \ge 0$ folds $g_r \circ \cdots \circ g_1$, and an simplicial immersion $g_{r+1}$. The collapse, subdivision, and immersion have cancellation constants $0$ while each fold has cancellation constant $1$ by $\pi_1$-injectivity. Thus we may choose $C(g) = r$.
\end{proof}

%If $g: \Gamma \to \Gamma'$ is a $\pi_1$-injective graph map that induces $\psi: \pi_1(\Gamma) \to \pi_1(\Gamma')$, $\mathcal A$ is a nontrivial free factor system of $F$, and $\hat \Gamma_k, \hat \Gamma_k'$ are the disjoint union of covers of $\Gamma, \Gamma'$ corresponding to $\mathcal A, \psi^k(\mathcal A)$ respectively for some $k \ge 1$, then $g$ lifts to a map $\hat g_k: \hat \Gamma_k \to \hat \Gamma_k'$ and the deformation retractions $\hat \Gamma_k \to S(\mathcal A)$ and $\hat \Gamma_k' \to S(\psi^k(\mathcal A))$ induces a map $\bar g_k: S(\mathcal A) \to S(\psi^k(\mathcal A))$ with $K(\bar g_k) = K(g)$ and $C(\bar g_k) = C(g)$. We shall call $\bar g_k$ the {\bf ($\boldsymbol k$-th) homotopy lift of $\boldsymbol g$}.

Let $f:\Gamma \to \Gamma$ be a topological representative for an injective endomorphism $\phi:F \to F$, $\mathcal A$ be a nontrivial $[\phi]$-invariant free factor system of $F$, and $\hat \Gamma_k$ be the disjoint union of covers of $\Gamma$ corresponding to $\phi^k(\mathcal A)$ for some $k \ge 1$. Then $f$ lifts to a map $\hat f_k: \hat \Gamma_k \to \hat \Gamma_k$ and the deformation retraction $\hat \Gamma_k \to S[\phi^k(\mathcal A)]$ induces a map $\bar f_k: S[\phi^k(\mathcal A)] \to S[\phi^k(\mathcal A)]$ with $K(\bar f_k) = K(f)$ and $C(\bar f_k) = C(f)$. We shall call $\bar f_k$ the {\bf ($\boldsymbol k$-th) homotopy lift of $\boldsymbol f$}. 

Stallings graphs $S[\phi^k(\mathcal A)]$ are $\phi^k(\mathcal A)$-marked graphs by definition and the maps $\bar f_k$ are weak representatives for $[\left.\phi\right|_{\phi^k(\mathcal A)}\,]$ that might map branch points to bivalent vertices. Note that injectivity of $\phi$ allows us to also consider the graphs as $\mathcal A$-marked graphs and $\bar f_k$ as weak representatives for $[\left.\phi\right|_{\mathcal A}\,]$. %In the latter view, the maps $\bar f_k$ are weak representatives for $\left[\left.\phi\right|_{\mathcal A}\,\right]$ that might map branch points to bivalent vertices. 
We hope to replace the weak representatives $\bar f_k$ with homotopic natural representatives while maintaining uniform control on the Lipschitz and cancellation constants. 
The next lemma allows us to measure how close the homotopy lift $\bar f_k$ is to mapping branch points to branch points.

\begin{lem}\label{branching} Suppose $g: \Gamma \to \Gamma'$ be a $\pi_1$-injective graph map with cancellation constant $C = C(g)$. Then $g$ maps branch points to the $C$-neighborhood of branch points.
%Let $\phi: F \to F$ be injective, $\mathcal A$ be a nontrivial $[\phi]$-invariant free factor system, and $f: \Gamma \to \Gamma$ be a topological representative for $[\phi]$. If $\bar f_k: S[\phi^k(\mathcal A)] \to S[\phi^k(\mathcal A)]$ is a homotopy lift of $f$ with $C(\bar f_k) = C(f)$, where $k \ge 1$, then $\bar f_k$ maps branch points to the $C(f)$-neighborhood of branch points.
\end{lem}
\begin{proof} Set $C = C(g)$. If $\Gamma'$ is the $C$-neighborhood of its branch points, then there is nothing to prove. Suppose $\nu \in \Gamma'$ is a bivalent vertex whose distance to the nearest branch point is greater than $C$. We need to show that $\nu$ is not the $g$-image of any branch point. Set $\epsilon_1$ and $\epsilon_2$ to be the distinct oriented directions originating from $\nu$ and $\bar \epsilon_1, \bar \epsilon_2$ are the same directions with opposite orientation.

Let $v \in \Gamma$ be a branch point and $g(v) = \nu$. As $v$ is a branch point, there are at least three distinct oriented directions originating from $v$: $e_1$, $e_2$, and $e_3$. Let $p_{12}$ be an immersed path that starts and ends with $e_1$ and $\bar e_2$ respectively and define $p_{23}$ similarly. Set $p_{13} = [p_{12} \cdot p_{23}]$ and $p'_{13} = p_{12} \cdot \bar p_{23}$, where $\bar p_{23}$ is the reversal of the path $p_{23}$. See Figure~\ref{figbcc} for an illustration. Although the paths are loops, we  still treat them as paths, i.e., tightening is done rel. the endpoints. Without loss of generality, assume $[g( p_{12})]$ starts with $\epsilon_1$.

\begin{figure}[ht]
 \centering 
 \includegraphics[scale=1.5]{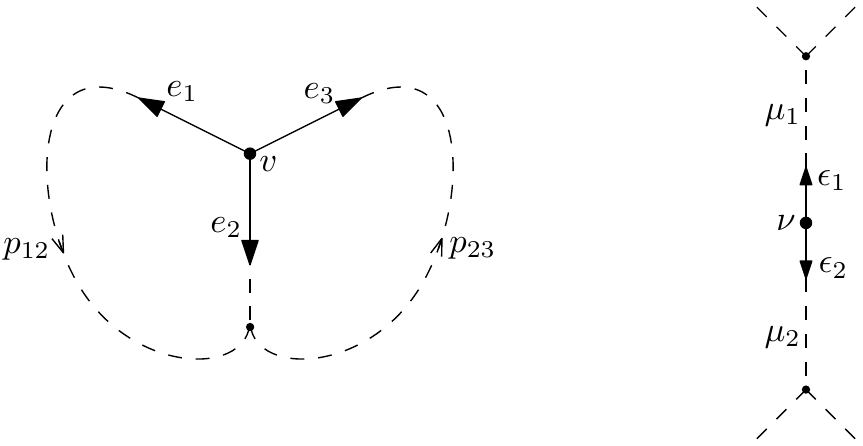}
 \caption{Schematic for paths $p_{12}, p_{23}, p_{13},$ and $p_{13}'$. The path $p_{13}$ starts with $e_1$ follows the dashed path and ends with $\bar e_3$. The path $p_{13}'$ is the ``figure 8'' path traced by $p_{12}$ then $\bar p_{23}$.}
 \label{figbcc}
\end{figure}

If $[g( p_{12})]$ ends with $\bar\epsilon_1$, then $[g( p_{12})] = \mu_1 \cdot \rho \cdot \bar \mu_1$ , where $\mu_1$ is an extension of $\epsilon_1$ to an embedded path from $\nu$ to a branch point and $\rho$ is an immersed nontrivial loop. 
By hypothesis, $\mu_1$ is longer than $C$. Since $p_{12}$ starts and ends with $e_1$ and $\bar e_2$ respectively, the concatenation $p_{12} \cdot p_{12}$ is an immersed path such that $[g(p_{12})] \cdot [g(p_{12})]$ has $\bar \mu_1 \cdot \mu_1$ as a subpath, violating bounded cancellation.
So we may assume $[g( p_{12})]$ starts and ends with $\epsilon_1$ and $\bar \epsilon_2$.

If $[g( p_{23})]$ starts and ends with $\epsilon_2$ and $\bar \epsilon_1$, then $[g( p_{13})] = [g(p_{12}) \cdot g(p_{23})]$ starts and ends with $\epsilon_1$ and $\bar \epsilon_1$ respectively, which violates bounded cancellation for the same reason given in the previous paragraph.
Similarly, if $[g( p_{23})]$ starts and ends with $\epsilon_1$ and $\bar \epsilon_2$, we rule out this possibility by considering $[g(p'_{13})]$. We have ruled out all cases, and therefore, no branch point $v$ of $\Gamma$ is mapped to $\nu$.
\end{proof}

\begin{cor}\label{natural} Let $g: \Gamma \to \Gamma'$ be a $K$-Lipschitz $\pi_1$-injective graph map with cancellation constant $C$. Then $g$ is homotopic to a $(K+C)$-Lipschitz natural graph map with cancellation constant $2C$.
\end{cor}
\begin{proof} By Lemma~\ref{branching}, $g$ maps branch points to the $C$-neighborhood of branch points. So we can find a graph map $g'$ homotopic to $g$ that maps branch points to branch points. Since the homotopy is moving images of branch points a distance at most $C$, we can use $K(g')=K+C$ and $C(g') = 2C$.

The bounded cancellation lemma only considers natural edge-paths, so homotopies that are supported in the interior of natural edges will not affect the cancellation constant. Using a tightening homotopy supported in the interior of natural edges in $\Gamma$, we may assume $g'$ is a natural graph map. The homotopy will not affect the Lipschitz and cancellation constants. 
\end{proof}

So we can replace the homotopy lift $\bar f_k$ with a homotopic natural representative that has Lipschitz and cancellation constants $K(f) + C(f)$ and $2 C(f)$ respectively. One would usually collapse the pretrivial edges and forget bivalent vertices to get a topological representative but we will not since we want to preserve compatibility: $\lVert \cdot \rVert_{\beta_*}$ is the restriction of $\lVert \cdot \rVert_{\alpha}$ to $\phi^k(\mathcal A)$. 
To summarize the properties of homotopy lifts that will be used in the next sections:

\begin{prop}\label{homotopylift}Suppose $\phi:F \to F$ is injective, $f: \Gamma \to \Gamma$ is a topological representative for $[\phi]$, and $\mathcal A$ is a nontrivial $[\phi]$-invariant free factor system. For any $k \ge 1$, there is a natural representative $\bar f_k: S[\phi^k(\mathcal A)] \to S[\phi^k(\mathcal A)]$  for $[\left.\phi\right|_{\phi^k(\mathcal A)}\,]$ with Lipschitz and cancellation constants $K(\bar f_k) = K(f) + C(f)$ and $C(\bar f_k) = 2 C(f)$ respectively.
%\begin{enumerate}
%\item it is a homotopy lift of $f$, i.e., $\iota_k \circ \bar f_k \simeq f \circ \iota_k$;
%\item it maps branch points to  branch points and any natural edge to a branch point or an immersed path; in particular, forgetting bivalent vertices induces a weak representative for $\left[\left.\phi\right|_{\mathcal A}\,\right]$.
%\item $\left.\bar f_k\right|_E$ is either a constant map or an immersion for all natural edges $E$;
%\item $K(\bar f_k) = K(f) + C(f)$ and $C(\bar f_k) = 2 C(f)$.
%\end{enumerate}
\end{prop}
\noindent The crucial point is that the Lipschitz and cancellation constants are independent of $k$. %Condition 2 will be referred to as {\it weak representative with respect to natural edges}.

%%%%%%%%%%%%%%%%%%%%%%%%%%%%%%%%%%%%%%%%%%%%%%%%%%%%%%%%%%
\subsection{Canonical and elliptic free factor systems}\label{secElliptic}
%%%%%%%%%%%%%%%%%%%%%%%%%%%%%%%%%%%%%%%%%%%%%%%%%%%%%%%%%%

In this section, we will construct a canonical invariant free factor system for any given injective endomorphism of $F$. This free factor system, called the {\it elliptic free factor system}, is crucial for the construction of {\it (expanding) relative immersions} later in the \paper.

%{\color{red}\begin{defn} Suppose $\phi:F \to F$ is an endomorphism. We shall say $\phi$ is expansive with respect to a marked graph $(\Gamma, \alpha)$ if for all $L \ge 1$, there exists $k \ge 1$ such that $\lVert \phi^k(g) \rVert_\alpha > L$ for all nontrivial elements $g \in F$. When $\phi$ is expansive with respect to some marked graph, it is expansive with respect all marked graphs; so the marked graph will usually be omitted and we say $\phi$ is {\bf expansive}.
%\end{defn}}

%Automorphisms are not expansive and, roughly speaking, injective nonexpansive endomorphisms are those built from automorphisms of free factor systems. The next definition will be needed for the proof of this characterization.

%\begin{defn}
Suppose $\mathcal A$ is a free factor system and $\phi:F \to F$ is an injective endomorphism. We shall say that a conjugacy class $[g]$ in $F$ {\bf has an infinite $\boldsymbol{[\phi]}$-tail} if for every $n \ge 1$, there is a conjugacy class $[g_n]$ in $F$ such that $[\phi^{n}(g_n)] = [g]$. The system $\mathcal A$ {\bf carries an infinite $\boldsymbol{[\phi]}$-tail of a conjugacy class} $[g]$ in $F$ if for every $n \ge 1$, there is a conjugacy class $[g_n]$ carried by $\mathcal A$ such that $[\phi^{n}(g_n)] = [g]$.
%\end{defn}
We now state and prove the main technical result of this section.

\begin{thm}\label{expequiv1} If $\phi: F \to F$ is injective, then $[\phi]$ has a nontrivial fixed free factor system if and only if some nontrivial conjugacy class has an infinite $[\phi]$-tail.
\end{thm}
\begin{rmk} Reynolds defined {\it expansive} endomorphisms \cite[Definition~3.8]{Rey11} and $[\phi]$ is expansive in this sense exactly when only the trivial conjugacy class has an infinite $[\phi]$-tail. Under this equivalence, Reynolds' Remark 3.12 in \cite{Rey11} is a weaker form of this theorem.
\end{rmk}
\begin{proof}
Let $\phi:F \to F$ be an injective endomorphism. The forward direction is obvious: if $[\phi]$ has a nontrivial fixed free factor system $\mathcal A$, then any nontrivial conjugacy class $[g]$ carried by $\mathcal A$ has an infinite $[\phi]$-tail. The main content of the theorem is in the reverse direction.
Fix a nontrivial conjugacy class $[g]$ in $F$ with an infinite $[\phi]$-tail. %If $\phi$ is surjective, then $\{ F \}$ is a nontrivial $[\phi]$-fixed free factor system and we are done. So we assume $\phi$ is not surjective and 
We proceed by descending down the poset of free factor systems. The following claim is the key idea: 
\begin{claim}[Descent] Let $\mathcal D$ be a $[\phi]$-invariant free factor system that carries an infinite $[\phi]$-tail of $[g]$. If $\mathcal D$ contains a free factor that is not eventually $[\phi]$-periodic, then some $[\phi]$-invariant free factor system $\mathcal D' \prec \mathcal D$ carries an infinite $[\phi]$-tail of $[g]$.\end{claim}
Since there are no infinite chains in the poset of free factor systems, the descent (proven below) starts with the free factor system $\{ F \}$ %and nonsurjectivity of $\phi$
 and then finds a $[\phi]$-invariant free factor sytem $\mathcal A$ that carries an infinite $[\phi]$-tail of $[g]$ and whose free factors are eventually $[\phi]$-periodic; such a free factor system contains a nontrivial $[\phi]$-fixed free factor system by Lemma~\ref{nonexpand} and we are done.\end{proof}

\noindent Note that the assumption ``$\mathcal D$ contains a free factor that is not eventually $[\phi]$-periodic'' in the descent claim is more general than we really need for our conclusion. We could have worked with a more natural assumption ``$\mathcal D$ has no periodic free factors'' and reached the same conclusion. However, we give the general argument instead since we will invoke a variation of it in the next proposition.

\begin{proof}[Proof of descent] Let $(\Gamma, \alpha)$ be a marked graph, $f:\Gamma \to \Gamma$ be a topological representative for $[\phi]$, and set $K = K(f)+C(f)$ and $C = 2C(f)$. Suppose $\mathcal D$ is a $[\phi]$-invariant free factor system that carries an infinite $[\phi]$-tail $[g_{n}]_{n\ge 1}$ of $[g]$ and contains a free factor that is not eventually $[\phi]$-periodic. Then, for all $k \ge 1$ and $n \ge k$, there is an immersed loop $\rho_k(g_n)$ in $\Delta_k = S[\phi^k(\mathcal D)]$ corresponding to $[\phi^{k}(g_{n})]$, where  $\Delta_k = S[\phi^k(\mathcal D)]$ is the Stallings graph with respect to $(\Gamma, \alpha)$. %Note that $\rho_k(g_k)$  has length $\lVert g \rVert_\alpha$ since $[\phi^{k}(g_{k})] = [g]$.
 Set $L = \max\{\lVert g \rVert_\alpha, C\}$.

For all $k \ge 1$, let $\bar f_k:\Delta_k \to \Delta_k$ be the natural representatives for $[\left.\phi\right|_{\phi^k(\mathcal D)}\,]$  given by Proposition~\ref{homotopylift}. In particular, these representatives have Lipschitz and cancellation constants $K$ and $C$ respectively. As $[g_n]_{n\ge 1}$ is an infinite $[\phi]$-tail of $[g]$ carried by $\mathcal D$, we get,  for any fixed $k \ge 1$, an infinite sequence of immersed loops $(\rho_k(g_{n}))_{n \ge k}$ in $\Delta_k$ such that the free homotopy classes of $\bar f_k^{n-k}(\rho_k(g_{n}))$ have length $\lVert g \rVert_\alpha$ for all $n \ge k$. %but the classes need not contain $\rho_k(g_k)$ in $\Delta_k$ as $\phi^{n}(g_{n})$ need not be conjugate to $\phi^k(g_k)$ in $\phi^k(F)$.

Form a directed graph $\mathbb G_k$ whose vertices are the natural edges of $\Delta_k$ and there is a directed edge $E_i \to E_j$ if $\bar f_k$ maps natural edge $E_i$ over $E_j$. Note that the number of natural edges of $\Delta_k$ is at most $N = 3\cdot\mathrm{rank}(F) - 3$ and so $\mathbb G_k$ has at most $N$ vertices.

Since $\mathcal D$ contains a free factor that is not eventually $[\phi]$-periodic, the length of natural edges in $\Delta_k$ is unbounded as $k \to \infty$ by Lemma~\ref{expand}.
Fix $k\gg 0$ such that the longest natural edge in $\Delta_k$ is longer than $L \cdot K^{N-1}$. 
Let $\mathcal L_0$ be the natural edges of $\Delta_k$ longer than $L \cdot K^{N-1}$ and $\mathcal L$ be the union of $\mathcal L_0$ and all natural edges on a directed path to $\mathcal L_0$ in $\mathbb G_k$. Since $\bar f_k$ is $K$-Lipschitz and the shortest directed path in $\mathbb G_k$ from a natural edge in $\mathcal L$ to $\mathcal L_0$ has at most $N$ natural edges on it, every natural edge in $\mathcal L$ is longer than $L$. The natural edges in $\mathcal L$ will be referred to as the {\it long} natural edges and the remaining natural edges as ${\it short}$.

 Set $\Delta' \subset \Delta_k$ to be the union of short natural edges, which will be a proper subgraph since long natural edges exist. The subgraph is automatically $\bar f_k$-invariant by the construction of $\mathcal L$. Since $\rho_k(g_k)$ is an immersed loop in $\Delta_k$ with length $\lVert g \rVert_\alpha \le L$, its natural edges are short and $\Delta'$ is a nonempty, noncontractible proper subgraph of $\Delta_k$.
Therefore, $\Delta'$ determines a nontrivial $[\phi]$-invariant proper free factor system $\mathcal D' \prec \mathcal D$. 
Technically, it determines an invariant free factor system of $\phi^k(F)$ but, as $\phi$ is injective, this corresponds to an invariant free factor system of $F$.
It remains to show that $\mathcal D'$ carries an infinite $[\phi]$-tail of $[g]$.

Let $\mathbb L \subset \mathbb G_k$ be the full subgraph generated by the long natural edges $\mathcal L$. If there are no directed cycles in $\mathbb L$, then $\bar f_k^N(\Delta_k) \subset \Delta'$; in this case, the sequence of nontrivial loops $(\bar f_k^N(\rho_k(g_n)))_{n \ge k+N}$  in $\Delta'$ determines an infinite $[\phi]$-tail of $[g]$ carried by $\mathcal D'$ and we are done.
Now suppose there are directed cycles in $\mathbb L$ and let $\rho$ be an immersed loop in $\Delta_k$ that contains a long natural edge in such a cycle. Then, by bounded cancellation and the fact long natural edges are longer than $L \ge C$, $[\bar f_k^m(\rho)]$ contains a long natural edge in the same directed cycle in $\mathbb L$ for all $m \ge 1$. 
Consequently, none of the immersed loops $\rho_k(g_{n})$ in $\Delta_k$ contain a long natural edge that is in a directed cycle of $\mathbb L$. Therefore, as far as the sequence of loops $(\rho_k(g_n))_{n \ge k}$ is concerned, we may assume there are no directed cycles in $\mathbb L$ and, as before, the sequence $(\bar f_k^N(\rho_k(g_{n})))_{n \ge k+N}$ determines an infinite $[\phi]$-tail of $[g]$ carried by $\mathcal D'$.
\end{proof}

The following dichotomy is (equivalent to) a result in Reynolds' thesis.

\begin{cor}[{\cite[Proposition~3.11]{Rey11}}] If $\phi:F \to F$ is irreducible, then either $\phi$ is an automorphism or only the trivial conjugacy class has an infinite  $[\phi]$-tail.
\end{cor}
\begin{proof} Suppose $\phi$ is irreducible and there is a nontrivial conjugacy class with an infinite $[\phi]$-tail. By Theorem~\ref{expequiv1}, there is a nontrivial $[\phi]$-fixed free factor system $\mathcal A$. Since $\phi$ is irreducible, $\mathcal A = \{ F \}$ and $\phi$ is an automorphism.
\end{proof}

The fixed free factor system given by Theorem~\ref{expequiv1} may depend on the chosen conjugacy class $[g]$ with an infinite tail or the marked graphs $(\Gamma, \alpha)$ chosen in the descent. The next proposition constructs a canonical fixed free factor system for $[\phi]$; this system carries all conjugacy classes with an infinite tail as well as all finitely generated fixed subgroup system. The proof will use both descent and ascent (like a losing game of Tetris) in the poset of free factor systems!

\begin{prop}\label{maxfixed} If $\phi:F \to F$ is injective, then there is a unique maximal $[\phi]$-fixed free factor system $\mathcal A$. Precisely, $\mathcal A$ carries every conjugacy class with an infinite $[\phi]$-tail and every finitely generated $[\phi]$-fixed subgroup system.
\end{prop}
\begin{proof}
Let $\phi:F\to F$ be an injective endomorphism. %If $\phi$ is an automorphism, then let $\mathcal A = \{ F\}$ and we are done. 
If the trivial conjugacy class is the only conjugacy class with an infinite $[\phi]$-tail, then the trivial system is the only $[\phi]$-fixed subgroup system. In this case, the trivial system is the unique maximal $[\phi]$-fixed free factor system and we are done as it vacuously carries all conjugacy classes with infinite $[\phi]$-tails and all $[\phi]$-fixed subgroup systems. We can now assume some nontrivial conjugacy class has an infinite $[\phi]$-tail. By Theorem~\ref{expequiv1}, $\phi$ %is not surjective and 
 has a nontrivial fixed free factor system $\mathcal D_0$. We proceed by ascending up the poset of free factor systems: 

\begin{claim}[Ascent] Let $\mathcal B$ be a finitely generated $[\phi]$-fixed subgroup system, $[g]$ a conjugacy class with an infinite $[\phi]$-tail, and $\mathcal D$ a nontrivial $[\phi]$-fixed free factor system. If $\mathcal D$ does not carry both $\mathcal B$ and $[g]$, then some $[\phi]$-fixed free factor system $\mathcal D' \succ \mathcal D$ carries both $\mathcal B$ and $[g]$.\end{claim}
Once again, as there are no infinite chains in the poset of free factor systems, the ascent (proven below) starts with the nontrivial $[\phi]$-fixed proper free factor $\mathcal D_0$ and stops at a necessarily unique maximal $[\phi]$-fixed free factor system $\mathcal A$ that carries all finitely generated $[\phi]$-fixed subgroup systems and all conjugacy classes with infinite $[\phi]$-tails.\end{proof}

\begin{proof}[Proof of ascent]Let $\mathcal B$ be a finitely generated nontrivial $[\phi]$-fixed subgroup system of $F$, $[g]$ be a nontrivial conjugacy class in $F$ with an infinite $[\phi]$-tail, and $\mathcal D$ be a nontrivial $[\phi]$-fixed free factor system of $F$ that does not carry both $\mathcal B$ and $[g]$. We now describe the descent:
 
\begin{claim}[Descent] Let $\mathcal D''$ be a $[\phi]$-invariant free factor system that carries $\mathcal D$, $\mathcal B$, and an infinite $[\phi]$-tail of $[g]$. If $\mathcal D''$ contains a free factor that is not eventually $[\phi]$-periodic, then some $[\phi]$-invariant free factor system $\mathcal D''' \prec \mathcal D''$ carries $\mathcal D$, $\mathcal B$, and an infinite $[\phi]$-tail of $[g]$.\end{claim}

Starting with $\{ F \}$, the descent (proven below) will find a nontrivial $[\phi]$-invariant free factor system $\mathcal D^*$ that carries $\mathcal D$, $
\mathcal B$, and an infinite $[\phi]$-tail of $[g]$ and whose free factors are eventually $[\phi]$-periodic. By Lemma~\ref{nonexpand}, $\mathcal D^*$ contains a $[\phi]$-fixed free factor subsystem $\mathcal D' \subset \mathcal D^*$ such that $\phi^k(\mathcal D^*)$ is carried by $\mathcal D'$ for some $k \ge 0$. As $\mathcal D$ and $\mathcal B$ are $[\phi]$-fixed, they are carried by $\mathcal D'$. Similarly, $\mathcal D'$ carries $[g]$ since $\mathcal D^*$ carries an infinite $[\phi]$-tail of $[g]$. So $\mathcal D'$ is a $[\phi]$-fixed free factor system that carries $\mathcal D$, $\mathcal B$, and $[g]$ as needed for ascent.\end{proof}

\begin{proof}[Proof of descent] Let $(\Gamma, \alpha)$ be a marked graph, $f:\Gamma \to \Gamma$ be a topological representative for $[\phi]$, and set $K = K(f)+C(f)$ and $C = 2C(f)$. Suppose $\mathcal D''$ is a $[\phi]$-invariant free factor system that carries $\mathcal D$, $\mathcal B$, and an infinite $[\phi]$-tail of $[g]$. Let $S[\phi^k(\mathcal D)]$ and $S[\phi^k(\mathcal B)]$ be the Stallings graphs with respect to $(\Gamma, \alpha)$. Since $\mathcal D$ and $\mathcal B$ are finitely generated and $[\phi]$-fixed, the length of the longest immersed loop in $S[\phi^k(\mathcal D)]$ or $S[\phi^k(\mathcal B)]$ that covers any edge at most twice is uniformly bounded by some $L_0$ for all $k \ge 1$. Set $L = \max\{ L_0, \lVert g \rVert_\alpha, C\}$.
The proof now mimics that of the descent in Theorem~\ref{expequiv1} and we only give a sketch. 

For all $k \ge 1$, let $\Delta_{k}'' = S[\phi^{k}(\mathcal D'')]$ and, by Proposition~\ref{homotopylift}, there is a $K$-Lipschitz natural representative for $\left[\left.\phi\right|_{\mathcal D''}\,\right]$, $\bar f_{k}: \Delta_{k}'' \to \Delta_{k}''$, that has cancellation constant $C$.
As some free factor in $\mathcal D''$ is not eventually $[\phi]$-periodic, we can fix $k \gg 0$ so that the longest natural edge in $\Delta_{k}''$ is longer than $L \cdot K^{N-1}$ by Lemma~\ref{expand}. Define the long and short natural edges as before and deduce long natural edges are longer than $L$. 
Set $\Delta''' \subset \Delta_{k}''$ to be the union of the short natural edges, which is necessarily proper and $\bar f_k$-invariant. Recall that immersed loops of $S[\phi^{k}(\mathcal D)]$ and $S[\phi^{k}(\mathcal B)]$ that cover any edge at most twice have length bounded by $L_0 \le L$ and these Stallings graphs have simplicial immersions into $\Delta_{k}''$. Hence the images of these immersions lie in the subgraph of short natural edges $\Delta'''$. So $\Delta'''$ is neither empty nor contractible. The subgraph $\Delta'''$ determines a $[\phi]$-invariant proper free factor system $\mathcal D''' \prec \mathcal D''$ that carries both $\mathcal D$ and $\mathcal B$ since both $S[\phi^k(\mathcal D)]$ and  $S[\phi^k(\mathcal B)]$ have immersions into $\Delta'''$. From the proof of descent in Theorem~\ref{expequiv1}, $L \ge C$ implies $\mathcal D'''$ carries an infinite $[\phi]$-tail of $[g]$.
\end{proof}

Although this proposition produces a canonical fixed free factor system for an injective endomorphism, we shall enlarge the system again to get a better $[\phi]$-invariant free factor system that gives us some control of the {\it relative dynamics} of $[\phi]$.
We do this by taking iterated preimages of the maximal fixed free factor system. We then show that the resulting invariant free factor system is a disjoint union of the maximal fixed free factor system with a free factor system that eventually gets mapped into the fixed system.

\begin{prop}\label{canonical} If $\phi: F\to F$ is injective and $\mathcal A$ is the maximal $[\phi]$-fixed free factor system, then there is a unique maximal $[\phi]$-invariant free factor system $\mathcal A^* \succeq \mathcal A$ such that $\phi^k(\mathcal A^*)$ is carried by $\mathcal A$ for some $k \ge 0$. After replacing the free factors of $\mathcal A$ with conjugates if necessary, we can assume $\mathcal A \subset \mathcal A^*$.
\end{prop} 
\noindent We shall call the free factor system given by this proposition the {\bf $\boldsymbol{[\phi]}$-elliptic free factor system}. We call it {\it elliptic} since it will be carried by the vertex groups in free splittings of $F$ in the subsequent sections.
\begin{proof} Let $\phi:F \to F$ be an injective endomorphism and $\mathcal A$ be the maximal $[\phi]$-fixed free factor system of $F$ given by Proposition~\ref{maxfixed}. If $\phi$ is surjective or has no nontrivial fixed free factor systems, then $\mathcal A^* = \mathcal A$ is $ \{ F \}$ or trivial respectively and we are done. Thus, we assume that $\mathcal A$ is a nontrivial proper free factor system. For all $k \ge 1$, define $\mathcal A_k = \phi^{-k} \cdot \mathcal A$.

Since $\mathcal A$ is $[\phi]$-invariant, we get that $\mathcal A \preceq \mathcal A_k \preceq \mathcal A_{k+1}$ for all $k \ge 1$ and, consequently, all $\mathcal A_k$ are $[\phi]$-invariant. As there are no infinite chains in the poset of free factor systems of $F$, the chain of $[\phi]$-invariant free factor systems $\mathcal A_k~(k \ge 1)$ stabilizes and we can set $\mathcal A^*$ to be the maximal free factor system in the chain. By construction, $\mathcal A^*$ carries a subgroup system $\mathcal B$ if and only if $\phi^k(\mathcal B)$ is carried by $\mathcal A$ for some $k \ge 0$ and this implies the uniqueness of $\mathcal A^*$ amongst $[\phi]$-invariant free factor system that have $\phi$-iterates carried by $\mathcal A$. It remains to show that $\mathcal A \subset \mathcal A^*$ after replacing the free factors of $\mathcal A$ with conjugates if necessary.

As in the proof of Lemma~\ref{nonexpand}, suppose $\sigma: \{1, \ldots l\} \to \{ 1, \ldots, l\}$ is the function associated to the $[\phi]$-invariance of $\mathcal A^* = \{ A_1, \ldots A_l \}$. Then there is a maximal nonempty subset $J \subset \{ 1, \ldots, l \}$ on which $\sigma$ is a bijection. Let $\mathcal A_J = \{\,A_j \in \mathcal A^*\,:\, j\in J\,\}$. Since $\mathcal A^*$ carries the maximal $[\phi]$-fixed free factor system $\mathcal A$, it follows that $\mathcal A \preceq \mathcal A_J$. Replace components of $\mathcal A$ with conjugates if necessary and assume $\mathcal A$ is a free factor system of $\mathcal A_J$; in particular, each $A \in \mathcal A$ is a subgroup of some $A_j \in \mathcal A_J$. We want to show that  $\mathcal A \subset \mathcal A_J$. Choose a component $A \in \mathcal A$ and let $A_j \in \mathcal A_J$ be the component such that $A \le A_j$. Furthermore, fix an inner automorphism $i_x: F \to F$ such that $i_x \circ \phi^s(A_j) \le A_j$ for some $s \ge 0$. Set $\psi = i_x \circ \phi^s$. %and note that $\psi$ and $\phi$ must have the same maximal fixed and elliptic free factor systems by construction. %We claim that $A_j$ is conjugate to some $A \in \mathcal A$.

By construction, $A_j \in \mathcal A^* = \phi^{-ks}\cdot \mathcal A$ implies $\psi^k(A_j)$ is conjugate to a subgroup of a $[\phi]$-periodic, and hence $[\psi]$-periodic, free factor $A' \in \mathcal A$ for some $k \ge 0$. We must have $ A' \le A_j$ since $A' \le A_{j\,'} \in \mathcal A_J$, $\{ A' \}$ carries $\{\psi^k(A_j)\}$, $\psi(A_j) \le A_j$, and $\left. \sigma \right|_J$ is a bijection.  %Assume $A \le A_j$ after replacing $A$ with a conjugate if necessary. 
So $i_y \circ \psi^k(A_j) \le A' \le A_j$ for some inner automorphism $i_y:F \to F$. The $[\psi]$-periodicity of $A'$ implies $(i_y \circ \psi^k)^m(A') \le A' $ is conjugate to $A'$ for some $m \ge 1$. But Lemma~\ref{fgconj} says no finitely generated subgroup of $F$ is conjugate to a proper subgroup of itself. Therefore, $(i_y \circ \psi^k)^m(A') = A'$ and, by injectivity of $\psi$, $i_y \circ \psi^k(A_j) = A' = A_j$. In particular, $A = A' = A_j$. As this holds for arbitrary free factors $A \in \mathcal A$, we get $\mathcal A \subset \mathcal A_J$.
\end{proof}

 %As a corollary of this lemma, we get that the maximal fixed free factor system carries all conjugacy classes $[g]$ with infnite tails.

%\begin{lem}\label{maxandelliptic} Let $\phi : F \to F$ be an injective endomorphism, $\mathcal A$ be the maximal $[\phi]$-fixed free factor system, and $\mathcal A^*$ be the $[\phi]$-elliptic free factor system. If $A^* \in \mathcal A^*$ and $\phi^k(A^*)$ is conjugate to a subgroup of $A^*$ for some $k \ge 1$, then $A^* \in \mathcal A$. In particular, $\mathcal A \subset \mathcal A^*$.\end{lem}
%\begin{lem}\label{maxandelliptic} If $\phi : F \to F$ is an injective endomorphism, $\mathcal A$ is the maximal $[\phi]$-fixed free factor system, and $\mathcal A^*$ is the $[\phi]$-elliptic free factor system, then $\mathcal A \subset \mathcal A^*$ after replacing the free factors of $\mathcal A$ with conjugates if necessary.\end{lem}
%\begin{proof}
%\end{proof}

It is obvious that the maximal $[\phi]$-fixed free factor system is proper exactly when $\phi$ is not surjective. The same holds for the $[\phi]$-elliptic free factor system:

\begin{obs*} Let $\phi:F \to F$ be an injective endomorphism. The $[\phi]$-elliptic free factor system is \begin{enumerate}
\item proper exactly when $\phi$ is not surjective.
\item trivial exactly when only the trivial conjugacy class has an infinite $[\phi]$-tail;
\end{enumerate}
\end{obs*}
%\begin{proof} Let $\mathcal A^*$ be the $[\phi]$-elliptic free factor system. Then, by construction and injectivity of $\phi$, $\mathcal A^*$ is nontrivial if and only if $\phi$ has a nontrivial fixed free factor system. The first equivalence then follows from Theorem~\ref{expequiv1}.

%The forward direction in the second equivalence is the obvious fact that the elliptic free factor system of an automorphism is $\{ F \}$. For the backward direction, suppose $\mathcal A^* = \{ F \}$. By construction of $\mathcal A^*$ and injectivity of $\phi$, the maximal $[\phi]$-fixed free factor system $\mathcal A$ is nontrivial. So $\mathcal A = \{ F \}$ by Proposition~\ref{canonical} and $\phi$ is an automorphism.\end{proof}

%In the next two sections, we will show that nonsurjective endomorphisms have expansive dynamics relative to their elliptic free factor systems.

%%%%%%%%%%%%%%%%%%%%%%%%%%%%%%%%%%%%%%%%%%%%%%%%%%%%%%%%%%
\subsection{Relative representatives}\label{secRelRep}
%%%%%%%%%%%%%%%%%%%%%%%%%%%%%%%%%%%%%%%%%%%%%%%%%%%%%%%%%%

We will now use relative (weak) representatives as the basis for inductively studying dynamical properties of free group endomorphisms.
For the whole section, we suppose $\phi: F \to F$ is an injective endomorphism and $\phi^{-1} \cdot \mathcal A = \mathcal A$, i.e, $\mathcal A$ is a $[\phi]$-invariant proper free factor system such that there is no free factor system $\mathcal A' \succ \mathcal A$ such that $\mathcal A$ carries $\phi(\mathcal A')$; e.g., consider an injective nonsurjective $\phi$ and its elliptic free factor system.

The first step is to establish the relative version of the bounded cancellation.

\begin{lem}[Bounded Cancellation]\label{relbcl} Let $T$ and $T'$ be $(F, \mathcal A)$- and $(F', \mathcal A')$-trees respectively and $\psi:F \to F'$ be an injective homomorphism such that $\psi^{-1}\cdot \mathcal A' = \mathcal A$. If $g: T \to T'$ is a $\psi$-equivariant tree map, then there is a constant $C(g)$ such that for every natural edge-path decomposition $p_1 \cdot p_2$ of an immersed path in $T$, the edge-path $[g(p_1)] \cdot [g(p_2)]$ is contained in the $C(g)$-neighborhood of $[g(p_1) \cdot g(p_2)]$.
\end{lem}
\begin{proof} The proof is the same as before. Since the trees are simplicial, the map $g$ factors as an equivariant pretrivial edge collapse and subdivision, a composition of $r \ge 0$ equivariant folds, and an equivariant simplicial embedding. As $\psi$ is injective and $\psi^{-1}\cdot \mathcal A' = \mathcal A$, no fold identifies vertices in the same orbit and, hence, each fold has cancellation constant 1. We may choose $C(g) = r$.
\end{proof}

Let $\mathcal A \preceq \mathcal B$ be a chain of $[\phi]$-invariant free factor systems, $T_{*}$ be a $(\mathcal B, \mathcal A)$-forest, and $f_{*}: T_{*} \to T_{*}$ be a $\mathcal A$-relative representative for $\left.\phi\right|_{\mathcal B}$. For all $k \ge 1$, set $T_{*}(\phi^k(\mathcal B)) \subset T_{*}$ to be the minimal subforest for $\phi^k(\mathcal B)$; minimal subforests are the relative analogues of iterated Stallings graphs. We will assume the minimal subforests $T_{*}(\phi^k(\mathcal B))$ inherit their simplicial structure from the {\it ambient forest }$T_{*}$ and so they might have bivalent vertices unlike $T_{*}$. 
For a graph of groups decomposition with bivalent vertices, {\bf branch points} are vertices that are images of branch points of the {\it Bass-Serre} tree and natural edges are images of natural edges of the tree. 
 
For any $k \ge 1$, let $f_{*, k}: T_{*}(\phi^k(\mathcal B)) \to T_{*}$  be the restriction of $f_{*}$ to $T_{*, k}$ and then replace it with an equivariantly homotopic map $T_{*}(\phi^k(\mathcal B)) \to T_{*}(\phi^k(\mathcal B))$ that is induced by the deformation retraction of $f_{*, k}(T_{*}(\phi^k(\mathcal B)))$ to $T_{*}(\phi^k(\mathcal B))$, which we call the {\bf ($\boldsymbol k$-th) homotopy restriction of $\boldsymbol{f_{*}}$}. Note that if $X \subset T_*(\phi^k(\mathcal B))$ is an axis such that $\left.f_*\right|_X$ is an immersion, then $\left.f_{*, k}\right|_X$ is still an immersion.

These homotopy restrictions are the relative analogues of homotopy lifts. 
%As before, using an equivariant (tightening) homotopy supported in the interior of natural edges, we  replace $f_{*, k}: T_{*}(\phi^k(\mathcal B)) \to T_{*}(\phi^k(\mathcal B))$ with a homotopic map that maps any natural edge to either a vertex or an immersed path and has Lipschitz and cancellation constants $K(f_{*})$ and $C(f_{*})$. Once again, if $X \subset T_*(\phi^k(\mathcal B))$ is an axis such that $\left.f_{*,k}\right|_X$ is an immersion before tightening, then it is still an immersion after tightening.
The proof of the following lemma is also almost the same as that of Lemma~\ref{branching}.

\begin{lem}\label{relbranching} Let $T$ and $T'$ be $(F, \mathcal A)$- and $(F', \mathcal A')$-trees respectively and $\psi:F \to F'$ be an injective homomorphism such that $\psi^{-1}\cdot \mathcal A' = \mathcal A$. If $g: T \to T'$ is a $\psi$-equivariant tree map with cancellation constant $C = C(g)$, then $g$ maps branch points to the $C$-neighborhood of branch points.
%Let $f_{*}:T_{*} \to T_{*}$ be an $\mathcal A$-relative representative for $\left.\phi\right|_{\mathcal B}$. Then for any $k \ge 1$, if $f_{*, k}: T_{*}(\phi^k(\mathcal B)) \to T_{*}(\phi^k(\mathcal B))$ is a homotopy restriction of $f_{*}$ with $C(f_{*, k}) = C(f_{*})$, then $f_{*, k}$ maps branch points to the $C(f_{*})$-neighborhood of branch points.
\end{lem}
\begin{proof}
Set $C = C(g)$ and let $\nu$ be a bivalent vertex in $T'$ whose distance to the nearest branch point is greater than $ C$. In particular, $\nu$ has a trivial stabilizer. We denote by $\epsilon_1, \epsilon_2$  the 2 distinct directions at $\nu$ oriented away from the vertex. Suppose $v$ is a branch point of $T$  such that $g(v) = \nu$. As $\psi$ is injective, $v$ has a trivial stabilizer under the action of $F$. Choose 3 distinct directions at $v$: $e_1, e_2,$ and $e_3$. Let $p_{12}$ be an embedded path in $T$ that starts with $e_1$ and ends with a translate of $\bar e_2$. Since $v$ has a trivial stabilizer, the path determines a unique loxodromic element $x_{12}$ in $F$ with axis $a_{12}$ such that $p_{12}$ is a fundamental domain of the axis under the translation action of $x_{12}$. Without loss of generality, $[g(p_{12})]$ starts with $\epsilon_1$.

If $[g(p_{12})]$ ends with the translate $\phi(x_{12})\bar \epsilon_1$, then $[g(p_{12})] = \mu \cdot \rho \cdot (\phi(x_{12})\bar \mu)$, where $\mu$ is an extension of $\epsilon_1$ to an embedded path from $\nu$ to the axis of $\phi(x_{12})$ and $\rho$ is a fundamental domain of the axis of loxodromic element $\phi(x_{12})$. By assumption, $\mu$ is longer than $ C$. Decompose the axis $a_{12} = a_{-} \cdot a_{+}$ at $v$, then $[g(a_{-})] \cdot [g(a_{+})]$ has $\bar \mu \cdot  \mu$  as a subpath, violating bounded cancellation. The remaining cases are handled similarly. Upon ruling out all cases, we conclude that no branch point $v$ of $T$ is mapped to $\nu$.
\end{proof}
As before, we get a corollary whose proof is essentially the same as that of Corollary~\ref{natural}.
\begin{cor}\label{relnatural} Let $T$ and $T'$ be $(F, \mathcal A)$- and $(F', \mathcal A')$-trees respectively and $\psi:F \to F'$ be an injective homomorphism such that $\psi^{-1}\cdot \mathcal A' = \mathcal A$. If $g: T \to T'$ is a $\psi$-equivariant $K$-Lipschitz tree map with cancellation constant $C$, then $g$ is equivariantly homotopic to a $\psi$-equivariant $(K+C)$-Lipschitz natural tree map with cancellation constant $2C$.
\end{cor}

The corollary allows us to replace $f_{*, k}$ with an equivariantly homotopic %({\it bivalent homotopy}) 
$\phi^k(\mathcal A)$-relative natural representative that has Lipschitz and cancellation constants $ K(f_{*, k}) = K(f_{*}) + C(f_{*})$ and $C(f_{*, k}) = 2 C(f_{*})$ respectively. If $X \subset T_{*}(\phi^k(\mathcal B))$ is an axis and $\left.f_{*, k}\right|_X$ is an immersion before the homotopy, then $f_{*, k}(X)$ is an immersed path after the homotopy; however, the restriction $\left.f_{*, k}\right|_X$ may fail to be an immersion due to pretrivial edges.
%To preserve combinatorial lengths, the bivalent homotopy here does not remove bivalent vertices.
%\begin{lem}\label{hrest2} Let $f_{*}: T_{*} \to T_{*}$ be an $\mathcal A$-relative representative for  $\left.\phi\right|_{\mathcal B}$. Then for any $k \ge 1$, there is a homotopy restriction $f_{*, k}: T_{*}(\phi^k(\mathcal B)) \to T_{*}(\phi^k(\mathcal B))$ of $f_{*}$ such that
%\begin{enumerate}\item $f_{*, k}$ maps branch points to branch points and any natural edge to a branch point or an immersed path;
%\item $f_{*, k}$ induced by folding $f_{*}$ and then applying equivariant univalent, tightening, and bivalent homotopies;
%\item $K(f_{*, k}) = K(f_{*})+C(f_{*})$, and $C(f_{*, k}) = 2C(f_{*})$.\end{enumerate}\end{lem}
The following summary is a relative analogue of Proposition~\ref{homotopylift}.
\begin{prop}\label{homotopyrest} Let $f_{*}:T_{*} \to T_{*}$ be an $\mathcal A$-relative representative for $\left.\phi\right|_{\mathcal B}$. For any $k \ge 1$, there is an $\phi^k(\mathcal A)$-relative natural representative $f_{*, k}: T_{*}(\phi^k(\mathcal B)) \to T_{*}(\phi^k(\mathcal B))$  for $\left.\phi\right|_{\phi^k(\mathcal B)}$ such that:
\begin{enumerate}
\item $K(f_{*, k}) = K(f_{*}) + C(f_{*})$ and $C(f_{*, k}) = 2 C(f_{*})$.
\item If $f_{*}$ restricted to the axes of some conjugacy class $[b]$ in $\phi^k(\mathcal B)$ is an immersion, then $f_{*, k}$ restricted to the axes of $[b]$ is also an immersion modulo pretrivial edges.
\end{enumerate}
\end{prop}

Collapsing a maximal (orbit-closed) $f_{*, k}$-invariant subforest of $ T_{*}(\phi^k(\mathcal B))$ with bounded components and forgetting the bivalent vertices induces a minimal $\phi^k(\mathcal A)$-relative representative $g_{*, k} : Y_{*, k} \to Y_{*, k}$ for $\left.\phi\right|_{\phi^k(\mathcal B)}$ defined on a $(\phi^k(\mathcal B), \phi^k(\mathcal A))$-forest $Y_{*, k}$. Note that the collapsed maximal subforest contains the pretrivial edges as $\phi^{-1}\cdot \mathcal A = \mathcal A$. Since $g_{*,k}$ is induced by equivariantly collapsing a forest and forgetting bivalent vertices, we have $K(g_{*, k}) = K(f_{*})+C(f_{*})$, $C(g_{*,k }) = 2C(f_{*})$, and $l_{Y_{*, k}} \le \left.l_{T_{*}}\right|_{\phi^k(\mathcal B)}$. If $X$ is an axis of $b$ in $T_{*}(\phi^k(\mathcal B))$ and $\left.f_{*, k}\right|_X$ is an immersion modulo pretrivial edges, then $\left.g_{*, k}\right|_{X'}$ is an immersion, where $X'$ is the axis of $b$ in $Y_{*,k}$. 

\begin{prop}\label{homotopyrest2} Let $f_{*}:T_{*} \to T_{*}$ be an $\mathcal A$-relative representative for $\left.\phi\right|_{\mathcal B}$. For any $k \ge 1$, there is a $(\phi^k(\mathcal B), \phi^k(\mathcal A))$-forest $Y_{*, k}$ and a minimal $\phi^k(\mathcal A)$-relative representative $g_{*, k} : Y_{*, k} \to Y_{*, k}$ for $\left.\phi\right|_{\phi^k(\mathcal B)}$ such that:
\begin{enumerate}
\item $K(g_{*, k}) = K(f_{*}) + C(f_{*})$ and $C(g_{*, k}) = 2 C(f_{*})$.
\item $l_{Y_{*, k}}:\phi^k(\mathcal B) \to \mathbb R$ is dominated by ($\le$) the restriction $\left.l_{T_{*}}\right|_{\phi^k(\mathcal B)} = l_{ T_{*}(\phi^k(\mathcal B))}$;
\item If $f_{*}$ restricted to the axes of some conjugacy class $[b]$ in $\phi^k(\mathcal B)$ is an immersion, then $g_{*, k}$ restricted to the axes of $[b]$ is also an immersion.
\end{enumerate}
\end{prop}
For an $\mathcal A$-relative weak representative $f_{*}$ for $\left.\phi\right|_{\mathcal B}$, we define the {\bf transition matrix} $A(f_*)$. Let $A(f_{*})$ be a nonnegative integer square matrix whose rows and columns are indexed by the orbits of edges in $T_{*}$; the entry of $A(f_{*})$ in row-$i$ and column-$j$, $A(f_{*})(i,j)$, is given by the number of translates of $e_i$ that are contained in the immersed edge-path $f_{*}(e_j)$, where $e_i$ is a orbit representative for the $i$-th orbit of edges. 
%By local injectivity in the interior of edges, $A(f_{*})$ has at least one positive entry in each column. 
An $\mathcal A$-relative weak representative $f_{*}$ is {\bf irreducible} if the matrix $A(f_{*})$ is irreducible, i.e., for any pair $(i,j)$, there is a positive integer $n_{ij}$ such that $A(f_{*})^{n_{ij}}(i,j) > 0$. In this case, the {\bf stretch factor} of $f_{*}$, $\lambda(f_{*}) \ge 1$, is the {\it Perron-Frobenius eigenvalue} of $A(f_{*})$. An irreducible $\mathcal A$-relative weak representative is {\bf expanding} if $\lambda(f_*) > 1$. Note that irreducible $\mathcal A$-relative representatives are minimal. 

%An $\mathcal A$-relative representative $f_{*}: T_{*} \to T_{*}$ for $\left.\phi\right|_{\mathcal B}$ is {\bf irreducible} if, for every orbit of edges $B_i \cdot e_i$ and $B_j \cdot e_j$, there is an integer $n_{ij}$ such that $B_i \cdot e_i$ has components contained in $f^{n_{ij}}(B_j \cdot e_j)$, where $B_i, B_j \in \mathcal B$ and $e_i, e_j$ are edges in $T_i, T_j \in T_{*}$ respectively. 
%Using {\it Perron-Frobenius theory}, we can define the {\it stretch factor} or {\it growth rate} of an irreducible $\mathcal A$-relative representative $f_{*}$, denoted by $\lambda(f_{*})$ (Details in Appendix~\ref{appreltt}). An irreducible $\mathcal A$-relative representative $f_{*}$ is {\bf expanding} if $\lambda(f_{*}) > 1$. %We shall say that an $\mathcal A$-relative representative $f_{*}: T_{*} \to T_{*}$ of a $(\mathcal B, \mathcal A)$-forest $T_{*}$ is a {\bf $\boldsymbol{\mathcal A}$-relative train track} if, for all edges $e$ in $T_{*}$ and integers $n \ge 1$, the edge-path $f^n_{*}(e)$ is immersed. If $\mathcal A$ is trivial and $\mathcal B = \{ F \}$, then such a map $f_{*} = f:T \to T$ induces a {\it train track} on the quotient graph $F \backslash T$ as defined by Bestvina-Handel \cite{BH92}. 

%We now give the set up for constructing relative train tracks. 
We say $\left.\phi\right|_{\mathcal B}$ is {\bf irreducible relative to $\boldsymbol{\mathcal A}$} if there is no $[\phi]$-invariant free factor system $\mathcal C$ such that $\mathcal A \prec \mathcal C \prec \mathcal B$. %By definition, a restriction $\left.\phi\right|_{\mathcal B}$ being irreducible relative to $\mathcal A$ implies that the associated function $\sigma: \{1, \ldots l\} \to \{ 1, \ldots, l\}$ is a bijection. 
If $\mathcal A$ is trivial and $\mathcal B = \{ F \}$, then we recover the definition of $\phi$'s irreducibility. The next lemma is the most useful property of an irreducible restriction for our purposes.
%Bestvina-Handel's algorithm can used to construct relative train tracks for relatively irreducible endomorphisms; the construction/algorithm is adapted for our relative setting in Appendix~\ref{relalgo}.
\begin{lem}\label{irredequiv} If $\left.\phi\right|_{\mathcal B}$ is irreducible relative to $\mathcal A$, then every minimal $\mathcal A$-relative representative for $\left.\phi\right|_{\mathcal B}$ is irreducible.
\end{lem}
\begin{proof} Suppose some minimal $\mathcal A$-relative representative for $\left.\phi\right|_{\mathcal B}$ has a reducible transition matrix; in particular, it has an invariant $\mathcal B$-equivariant proper subforest (with unbounded components) that determines a $[\phi]$-invariant free factor system $\mathcal C$ such that $\mathcal A \prec \mathcal C \prec \mathcal B$. So $\left.\phi\right|_{\mathcal B}$ is not irreducible relative to $\mathcal A$.
\end{proof}

\begin{rmk}Bestvina-Handel give the absolute version of this property as the definition of irreducibility and then prove that it is equivalent to the definition of irreducibility given in this {\paper} \cite[Lemma~1.16]{BH92}. The relative version of this equivalence holds as well but we will not prove it as it is not needed.
\end{rmk}

Bestvina-Handel used the next proposition to construct {\it train tracks} \cite[Theorem~1.7]{BH92}. 

\begin{prop}\label{reltt} If $\left.\phi\right|_{\mathcal B}$ is irreducible relative to $\mathcal A$, then there is an irreducible $\mathcal A$-relative representative $f_{*}:T_{*} \to T_{*}$ for $\left.\phi\right|_{\mathcal B}$ with the minimal stretch factor, i.e., if $f_{*}':T_{*}' \to T_{*}'$ is an irreducible $\mathcal A$-relative representative for $\left.\phi\right|_{\mathcal B}$, then $\lambda(f_{*}') \ge \lambda(f_{*})$.\end{prop}

\noindent The minimal stretch factor  will be denoted by $\lambda([\phi], \mathcal B, \mathcal A)$.

\begin{proof} Let $g_*: Y_* \to Y_*$ be  a minimal $\mathcal A$-relative representative for $\left.\phi\right|_{\mathcal B}$ and suppose $\left.\phi\right|_{\mathcal B}$ is irreducible relative to $\mathcal A$. Then $g_*$ is an irreducible $\mathcal A$-relative representative by Lemma~\ref{irredequiv} with stretch factor $\lambda(g_*)$.
By the lack on bivalent vertices, any irreducible $\mathcal A$-relative representative has a transition matrix of size $\le N = 3 \cdot \mathrm{rank}(F) - 3$. Suppose $B$ is an irreducible integer square matrix with Perron-Frobenius eigenvalue $\lambda(B) \le \lambda(g_*)$. Then $B$ has a positive right eigenvector $\vec v$ associated with $\lambda(B)$. So for all $k \ge 1$, $B^k$ has right eigenvector $\vec v$ associated with eigenvalue $\lambda(B)^k$. Assuming the smallest entry of $\vec v$ is $1$ (rescale if necessary), we get that the minimum row-sum of $B^k$ is at most $\lambda(B)^k$ for any $k \ge 1$. If $B$ has no more than $N$ rows, then the largest entry of $B$ is at most the minimum row-sum of $B^{N!}$, which we know is at most $\lambda(B)^{N!} \le \lambda(g_*)^{N!}$. So there are finitely many irreducible integer square matrices with size $\le N$ and Perron-Frobenius eigenvalue $\le \lambda(g_*)$. Thus, there is a finite set of stretch factors $\le \lambda(g_*)$ for irreducible $\mathcal A$-relative representatives for $\left.\phi\right|_{\mathcal B}$. In particular, there is an irreducible $\mathcal A$-relative representative $f_*: T_* \to T_*$ for $\left.\phi\right|_{\mathcal B}$ with the minimal stretch factor.
\end{proof}

Bestvina-Handel's work \cite{BH92} can be adapted to show that an irreducible $\mathcal A$-relative representative for $\left.\phi\right|_{\mathcal B}$ with the minimal stretch factor is an {\it $\mathcal A$-relative train track} (Appendix~\ref{relalgo}) and, conversely, bounded cancellation implies all irreducible $\mathcal A$-relative train tracks for $\left.\phi\right|_{\mathcal B}$ have the minimal stretch factor. We do not prove this converse as it is not needed. The next lemma is an application of train track theory that will be invoked once, in the second half of the proof of Proposition~\ref{relexpand}.

\begin{lem}[Train Track Theory]\label{legal} If $\left.\phi\right|_{\mathcal B}$ is irreducible relative to $\mathcal A$ and $f_{*}:T_{*} \to T_{*}$ is an irreducible $\mathcal A$-relative representative for $\left.\phi\right|_{\mathcal B}$ with the minimal stretch factor, then there is an element $g$ in $\mathcal B$ with an axis $a_g$ such that the restriction of $f_*^k$ to $a_g$ is an immersion for all $k \ge 1$.\end{lem}
\noindent Such an axis will be known as an {\bf $\boldsymbol{f_*}$-legal axis}.
\begin{proof} If $\lambda(f_*) = 1$, then $f_*$ is a simplicial embedding and we are done. So we may assume $\lambda(f_*) > 1$. By minimality of its stretch factor, $f_*$ is an expanding irreducible $\mathcal A$-relative {\it train track} for $\left.\phi\right|_{\mathcal B}$ (Theorem~\ref{appreltt}), i.e., for any edge $e$ in $T_*$, $f_*^k(e)$ is an expanding immersed path for all $k \ge 1$. A 2-edge path $e_1 \cdot e_2$ is {\it $f_*$-legal} if it is a translate of a subpath of $f_*^k(e)$ for some edge $e$ and integer $k  \ge 1$. By irreducibility of $f_*$, every edge $e$ is contained in a 3-edge path $e_{-} \cdot e \cdot e_{+}$ whose 2-edge subpaths are both $f_*$-legal. This means we can form an axis $a_g$ whose 2-edge subpaths are all $f_*$-legal. By the train track property, the restriction of $f_*^k$ to $a_g$ is an immersion for all $k \ge 1$.
\end{proof}

The main tools from Section~\ref{secStallBCC} that were used in the previous section were Lemma~\ref{expand}, bounded cancellation, and Proposition~\ref{homotopylift}. The relative analogues of the latter two have already been established in this section. We now state the main technical result of this section, an analogue of Lemma~\ref{expand} --- analogous in the sense that both give sufficient conditions for iterated subgroup graphs to have arbitrarily long natural edges.

\begin{prop}\label{relexpand} Let $\mathcal A \prec \mathcal B$ be a chain of $[\phi]$-invariant free factor systems with $\left.\phi\right|_{\mathcal B}$ irreducible relative to $\mathcal A$ and $\lambda([\phi], \mathcal B, \mathcal A) > 1$. If $\mathcal A$ carries the maximal $[\phi]$-fixed free factor system, then the length of the longest natural edge in $T_{*}(\phi^k(\mathcal B))$ is unbounded as $k \to \infty$.
\end{prop}

Before starting the proof, we will first describe the {\it (absolute) vertex blow-up} construction. Let $\mathcal A \prec \mathcal B$ be a chain of free factor systems and $T_*$ be some $(\mathcal B, \mathcal A)$-forest.
Recall that we assume $\mathcal A_i \subset \mathcal A$ is   also a free factor system of $B_i \in \mathcal B$.
Fix some $\mathcal A$-marked roses $(R_{\mathcal A}, \alpha_{\mathcal A})$ . 
Define $\Gamma_{\mathcal B}$ to be the graph formed by identifying the appropriate vertices of the graph of groups $\mathcal B \backslash T_{*}$ with the basepoints of roses $(R_{\mathcal A}, \alpha_{\mathcal A})$. If $c:R_{\mathcal A} \to \Gamma_{\mathcal B}$ is the inclusion map, then {\it Bass-Serre theory} gives markings $\alpha_{\mathcal B} = \{ \alpha_i: B_i \to \pi_1(\Gamma_i) \}$ such that $[\,\pi_1(c) \circ \alpha_{\mathcal A}\,] = [\,\left. (\alpha_{\mathcal B})\right|_{\mathcal A}\,]$. Thus, $(\Gamma_{\mathcal B}, \alpha_{\mathcal B})$ is a $\mathcal B$-marked graph. This construction and, in general any pair of graphs $\Gamma_{\mathcal A}' \subset \Gamma_{\mathcal B}'$ with collections of markings $\alpha_{\mathcal A}', \alpha_{\mathcal B}'$ such that $\pi_1(c') \circ \alpha_{\mathcal A}' = \left.(\alpha_{\mathcal B}')\right|_{\mathcal A}$ will be referred to as {\bf vertex blow-up}. 

We note that the Stallings graph $S[\phi^k(\mathcal B)]$ with respect to $(\Gamma_{\mathcal B}, \alpha_{\mathcal B})$, as a $\phi^k(\mathcal B)$-marked graph, is a vertex blow-up of $\phi^k(\mathcal B) \backslash T_{*}(\phi^k(\mathcal B))$: let $\iota: S[\phi^k(\mathcal B)] \to \Gamma_{\mathcal B}$ be the Stallings graph's immersion and $S_{\mathcal A} \subset S[\phi^k(\mathcal B)]$ be the core of the  subgraph $\iota_{\mathcal B}^{-1}(R_{\mathcal A})$. Since $\phi^{-1}\cdot \mathcal A = \mathcal A$, $S_{\mathcal A} = S[\phi^k(\mathcal A)]$ is marked by an isomorphism $\alpha_{\mathcal A}':\phi^k(\mathcal A) \to \pi_1(S_{\mathcal A})$ and $\alpha_{\mathcal A}'$ is the restriction of the marking $\alpha_{\mathcal B}':\phi^k(\mathcal B) \to \pi_1(S[\phi^k(\mathcal B)])$ to $\phi^k(\mathcal A)$ with respect to the inclusion $S_{\mathcal A} \subset S[\phi^k(\mathcal B)]$. Therefore, $S[\phi^k(\mathcal B)]$ is also a vertex blow-up of $\phi^k(\mathcal B)\backslash Y_{*, k}$. The noncontractible components of the subgraph $\iota_{\mathcal B}^{-1}(R_{\mathcal A})$ will be known as the {\it lower stratum} and the rest of the graph as the {\it top stratum}.

Now suppose $\mathcal A \prec \mathcal B$ are also $[\phi]$-invariant and let $f_{*}: T_{*} \to T_{*}$ be a $\mathcal A$-relative representative for $\left.\phi\right|_{\mathcal B}$ defined on some $(\mathcal B, \mathcal A)$-forest $T_{*}$ and $f_{\mathcal A}: R_{\mathcal A} \to R_{\mathcal A}$ be a topological representative for $\left.\phi\right|_{\mathcal A}$.
Construct a topological representative $f_{\mathcal B}: \Gamma_{\mathcal B} \to \Gamma_{\mathcal B}$ for $[\left.\phi\right|_{\mathcal B}\,]$  that agrees with $f_{\mathcal A}$ on the $\mathcal A$-marked roses $R_{\mathcal A}$ and induces $f_{*}$ on the Bass-Serre forest $T_{*}$ upon collapsing the roses $R_{\mathcal A}$.

For any $k \ge 1$, we let $g_{*, k}:Y_{*, k} \to Y_{*, k}$ be the minimal $\phi^k(\mathcal A)$-relative representative for $\left.\phi\right|_{\phi^k(\mathcal B)}$ given by Proposition~\ref{homotopyrest2} using $f_*$ and $\bar f_k:S[\phi^k(\mathcal B)] \to S[\phi^k(\mathcal B)]$ be the natural representative for $[\left.\phi\right|_{\phi^k(\mathcal B)}\,]$ given by Proposition~\ref{homotopylift} using $f_{\mathcal B}$.
By Proposition~\ref{homotopyrest2}(3), if an element $b$ in $\mathcal B$ has an $f_*$-legal axis, then $\phi^k(b)$ has a $g_{*, k}$-legal axis. 
It can be arranged for $S[\phi^k(\mathcal A)]\subset S[\phi^k(\mathcal B)]$ to be $\bar f_k$-invariant and $\bar f_k$ to induce $g_{*,k }$ on the $(\phi^k(\mathcal B), \phi^k(\mathcal A))$-forest $Y_{*, k}$ upon collapsing a maximal invariant proper subgraph of $S[\phi^k(\mathcal B)]$ containing $S[\phi^k(\mathcal A)]$ and forgetting bivalent vertices.
\medskip 

Here is the idea behind the proof. By irreducibility of the restriction $\left.\phi\right|_{\mathcal B}$, we may assume the map $g_{*, k}$ is an expanding irreducible representative for $\left.\phi\right|_{\phi^k(\mathcal B)}$. For the contrapositive, suppose the forests $T_{*}(\phi^k(\mathcal B))$ had uniformly bounded natural edges. There is a sequence of loxodromic elements $b_k$ in $\phi^k(\mathcal B)$ with uniformly bounded translation lengths $l_{T_*}(g_k)$. Now suppose that the vertex blow-up $S[\phi^k(\mathcal B)]$ had natural edges with aribtrarily long top stratum subpaths. Bounded cancellation, the fact $\bar f_k$ induces $g_{*, k}$, and the irreducibility of $g_{*, k}$ imply $g_{*, k}$ is an expanding irreducible {\it immersion}. However, this contradicts the first assumption since $l_{Y_{*, k}} \le l_{T_{*}(\phi^k(\mathcal B))}$. So the second supposition is false and the natural edges of $S[\phi^k(\mathcal B)]$ have top stratum subpaths with uniformly bounded length. Using the Lipschitz property, expanding irreducibility of $g_{*, k}$, and existence of a $g_{*, k}$-legal axis (train track theory), we find uniformly bounded lower stratum paths in $S[\phi^k(\mathcal B)]$ connecting the origin of any oriented top stratum subpath of a natural edge to another top stratum subpath of a natural edge. Consequently, we are able to build uniformly bounded immersed loops in $S[\phi^k(\mathcal B)]$ that contain top stratum subpaths. This implies some loxodromic conjugacy class in $\mathcal B$ has an infinite $[\phi]$-tail. By Proposition~\ref{maxfixed}, any conjugacy class in $F$ with an infinite $[\phi]$-tail is carried by the maximal $[\phi]$-fixed free factor system. Thus the maximal $[\phi]$-fixed free factor system cannot be carried by $\mathcal A$ as it carries a loxodromic conjugacy class.

\begin{proof}[Proof of Proposition~\ref{relexpand}]
Suppose $\mathcal A \prec \mathcal B$ are $[\phi]$-invariant free factor systems,  $\left.\phi\right|_{\mathcal B}$ is irreducible relative to $\mathcal A$, $f_{\mathcal A}: R_{\mathcal A} \to R_{\mathcal A}$ is a topological representative for $\left.\phi\right|_{\mathcal A}$ defined on $\mathcal A$-marked roses $(R_{\mathcal A}, \alpha_{\mathcal A})$, and $f_{*}: T_{*} \to T_{*}$ is an expanding irreducible $\mathcal A$-relative representative  for $\left.\phi\right|_{\mathcal B}$ with the minimal stretch factor $\lambda(f_*) > 1$ (Proposition~\ref{reltt}). By Lemma~\ref{legal}, there is an element $b$ in $\mathcal B$ with an $f_*$-legal axis. Set $(\Gamma_{\mathcal B}, \alpha_{\mathcal B})$ to be the vertex blow-up of $\mathcal B \backslash T_{*}$ with respect to the $\mathcal A$-marked roses $(R_{\mathcal A}, \alpha_{\mathcal A})$. 
%Assume $T_{*}$ has the Perron-Frobenius metric with unit volume and the roses have the usual combinatorial metric. 
The discussion preceding the proof gives minimal $\phi^k(\mathcal A)$-relative representatives $g_{*, k} : Y_{*, k} \to Y_{*, k}$ for $\left.\phi\right|_{\phi^k(\mathcal B)}$ and natural representatives $\bar f_k:S[\phi^k(\mathcal B)] \to S[\phi^k(\mathcal B)]$ for $[\left.\phi\right|_{\phi^k(\mathcal B)}\,]$ that have these properties: for all $k \ge 1$,
\begin{enumerate}
\item $\bar f_k$ induces $g_{*,k }$ on $Y_{*, k}$  upon collapsing the $f_{\mathcal B}$-invariant subgraph $R_{\mathcal A} \subset \Gamma_{\mathcal B}$;
\item  $K = K(\bar f_k) = K(f_{\mathcal B})+C(f_{\mathcal B})$  and $C = C(\bar f_k) = 2 C(f_{\mathcal B})$;
\item $l_{Y_{*, k}}:\phi^k(\mathcal B) \to \mathbb R$ is dominated by the restrictions $\left.l_{T_{*}}\right|_{\phi^k(\mathcal B)} = l_{ T_{*}(\phi^k(\mathcal B))}$; and
\item $\phi^k(b)$ has a $g_{*,k}$-legal axis.
\end{enumerate}

The collection $\left.\phi\right|_{\phi^k(\mathcal B)}$ is conjugate to $\left.\phi\right|_{\mathcal B}$ by injectivity of $\phi$. So $\left.\phi\right|_{\phi^k(\mathcal B)}$ is irreducible relative to $\phi^k(\mathcal A)$ and $\lambda(f_{*})$ is the minimal stretch factor for $\left.\phi\right|_{\phi^k(\mathcal B)}$ relative to $\phi^k(\mathcal A)$. Furthermore, the minimal $\phi^k(\mathcal A)$-relative representatives $g_{*, k}$ are irreducible  (Lemma~\ref{irredequiv}) and $\lambda(g_{*, k}) \ge \lambda(f_{*}) > 1$ by the minimality of $\lambda(f_{*})$. 

Suppose for the contrapositive that there is a bound $L \ge 1$ such that all natural edges in $T_{*}(\phi^k(\mathcal B))$ are shorter than $L$ for all $k \ge 1$. Then, for all $k \ge 1$, there is a loxodromic element $b_k$ in $\phi^k(\mathcal B)$ such that $l_{T_{*}}(b_k) \le (3N-3)L$, where $N = 3\cdot \mathrm{rank}(F)-3$. Every edge $E$ in $\Gamma_{*, k} = \phi^k(\mathcal B)\backslash Y_{*, k}$ lifts to a $\phi^k(\mathcal B)$-orbit of a natural edge $E'$ in $\phi^k(\mathcal B)\backslash T_{*}(\phi^k(\mathcal B))$, which corresponds to a top stratum subpath $\bar E$ of a natural edge in $S[\phi^k(\mathcal B)]$.

\begin{claim}The subpath $\bar E$ in $S[\phi^k(\mathcal B)]$ has length $ \le C {\cdot} K^{N-1}$  for all edges $E$ in $\Gamma_{*, k}$ and $k \ge 1$.\end{claim}
\noindent Suppose, the graph $\Gamma_{*, k}$ has an edge $E_0$ whose corresponding subpath $\bar E_0$ in $S[\phi^k(\mathcal B)]$ is longer than $\displaystyle C \cdot K^{N-1}$ for some $k \ge 1$. As we did in the proof of Theorem~\ref{expequiv1}, we construct the set of {\it long} edges $\mathcal E$ by looking at all the edges of $\Gamma_{*, k}$ that are eventually mapped over $E_0$. Here, an edge $E_1$ in $\Gamma_{*, k}$ mapped over $E_0$ if there are lifts $E_1'$ and $E_0'$ in $Y_{*, k}$ such that $g_{*, k}$ maps $E_1'$ over $E_0'$.
Since $\bar f_k: S[\phi^k(\mathcal B)] \to S[\phi^k(\mathcal B)]$ is $K$-Lipschitz and it induces $g_{*, k}$ on $Y_{*, k}$, each long edge in $\Gamma_{*, k}$ corresponds to a top statrum subpath in  $S[\phi^k(\mathcal B)]$ longer than $C$. Since $g_{*, k}$ is an irreducible $\phi^k(\mathcal A)$-relative representative, all edges eventually map over $E_0$ and hence are long. The long natural edges of $S[\phi^k(\mathcal B)]$ will be the natural edges in $S[\phi^k(\mathcal B)]$ containing top stratum subpaths.

Suppose an edge $E$ of $\Gamma_{*, k}$ had a lift $E'$ in $Y_{*, k}$ that is the initial segment of the $g_{*, k}$-image of two edges that share an initial vertex. Then the top stratum subpath $\bar E$ is in a long natural edge of $S[\phi^k(\mathcal B)]$  that is the initial segment of $\bar f_k$-images of natural edges that share an initial vertex; this violates bounded cancellation since long natural edges of $S[\phi^k(\mathcal B)]$ longer than $C = C(\bar f_k)$. 
Hence, there is no folding in $g_{*, k}$, i.e., $g_{*, k}$ is an expanding irreducible $\phi^k(\mathcal A)$-relative {\it immersion}. We may now find an $m \ge 1$ such that all loxodromic elements $b$ in $\phi^k(\mathcal B)$ have $l_{Y_{*, k}}(\phi^m(b)) > (3N-3)L$. Since $l_{Y_{*, k}}$ is dominated by $\left.l_{T_{*}}\right|_{\phi^k(\mathcal B)}$, we get that $l_{T_{*}}(b') > (3N-3)L$ for all loxodromic elements $b'$ in $\phi^{k+m}(\mathcal B)$. This contradicts the assumption that $l_{T_{*}}(b_{k+m}) \le (3N-3)L$ for some loxodromic $b_{k+m}$ in $\phi^{k+m}(\mathcal B)$. So the top stratum subpath $\bar E$ in $S[\phi^k(\mathcal B)]$ has length $\le C \cdot K^{N-1}$ for all natural edges $E$ of $\Gamma_{*, k}$ and $k \ge 1$. This ends the proof of the claim.
\medskip

Next, we prove the existence of paths in the lower stratum of $S[\phi^k(\mathcal B)]$ with uniformly bounded lengths connecting top stratum paths.
Suppose $E_0, E_1,$ and $E_2$ are edges of $\Gamma_{*, k}$ with lifts $E_0', E_1',$ and $E_2'$ in $Y_{*, k}$ such that $E_1' \cdot E_2'$ is a subpath of the immersed path $g_{*, k}(E_0')$. 
Then $\bar E_1 \cdot P_{12} \cdot \bar E_2$ is a subpath of immersed path $\bar f_k(\bar E_0)$ for some lower stratum path $P_{12}$ in $S[\phi^k(\mathcal B)]$. 
Since $\bar E_0$ has length bounded by $C \cdot K^{N-1}$ and $\bar f_k$ is $K$-Lipschitz, the path $P_{12}$ has length bounded by $C \cdot K^N$. 
We say the 2-edge path $E_1 \cdot E_2$ in $\Gamma_{*, k}$ has a {\it nondegenerate turn bounded by $C \cdot K^N$}. 

As $g_{*, k}$ is an expanding irreducible relative representative that has a legal axis (this is where the argument invokes train track theory), every edge $E'$ in $Y_{*, k}$ can be extended to an immersed 3-edge path $E_{-}' \cdot E' \cdot E_{+}'$ that is a translate of a subpath of $g_{*, k}^{n}(E')$ and $n \le 2{\cdot}N!$. In particular, any edge in $\Gamma_{*, k}$ can be extended to a 3-edge path whose 2-edge subpaths both have nondegenerate turns bounded by $C \cdot K^{N-1} \cdot K^{2{\cdot}N!}$, i.e., every top stratum subpath $\bar E$ can be extended to an immersed path $\bar E_{-}{\cdot}P_{-}{\cdot}\bar E{\cdot}P_{+}{\cdot}\bar E_{+}$ with top stratum subpaths $\bar E_{-}$, $\bar E_{+}$ and lower stratum paths $P_{-}$, $P_{+}$ with length bounded by $C \cdot K^{N-1} \cdot K^{2{\cdot}N!}$.

Using this bound on lower stratum paths and the bound on top stratum subpaths given by the claim, we can now form an immersed loop $\rho_k$ in $\Gamma_{*, k}$ with the properties:
 
\begin{enumerate}
\item $\rho_k$ lifts to an axis in $Y_{*, k}$ for some loxodromic conjugacy class $[b_k']$;
\item $\rho_k$ passes any edge of $\Gamma_{*, k}$ at most twice and only takes short turns (including the turn at the endpoint), which implies it has at most $2N$ edges and (short) turns; and
\item $\rho_k$ represents a loop in $S[\phi^k(\mathcal B)]$ with length bounded by $2N{\cdot}C (1+ K^{2{\cdot}N!})K^{N-1}$.
\end{enumerate} 
In summary, for each $k \ge 1$, we construct a loxodromic conjugacy class $[b_k']$ in $\phi^k(\mathcal B)$ whose $\alpha_{\mathcal B}$-length is bounded by a constant independent of $k$. As there are finitely many conjugacy classes with $\alpha_{\mathcal B}$-length bounded by any given constant, the sequence of conjugacy classes $[b_k']_{k=1}^\infty$ has a constant infinite subsequence. Thus, some loxodromic conjugacy class $[b']$ has an infinite $[\phi]$-tail carried by $\mathcal B$. Recall that the maximal $[\phi]$-fixed free factor system carries all conjugacy classes with an infinite $[\phi]$-tail (Proposition~\ref{maxfixed}); on the other hand, $\mathcal A$ does not carry loxodromic conjugacy classes in $\mathcal B$ by definition. Therefore, $\mathcal A$ cannot carry the maximal $[\phi]$-fixed free factor system.
\end{proof}

%%%%%%%%%%%%%%%%%%%%%%%%%%%%%%%%%%%%%%%%%%%%%%%%%%%%%%%%%%
\subsection{Canonical expanding relative immersions}\label{secRelImm}

The main result of this section is the existence of expanding immersions for nonsurjective endomorphisms relative to their elliptic free factor systems.

%\begin{defn}
Let $\phi:F  \to F$ be an injective endomorphism, $\mathcal A \prec \mathcal B$ be a pair of $[\phi]$-invariant free factor systems such that $\phi^{-1}\cdot \mathcal A = \mathcal A$, and $\left.\phi\right|_{\mathcal B}:\mathcal B \to \mathcal B$ be a restriction of $\phi$ to $\mathcal B$. 
Recall that a relative representative is minimal if it has no orbit-closed invariant subforests with bounded components and an expanding {${\mathcal A}$-relative immersion for $\boldsymbol{\left.\phi\right|_{\mathcal B}}$} is an $\mathcal A$-relative representative $f_{*}:T_* \to T_* $  for $\left.\phi\right|_{\mathcal B}$ that is a minimal immersion whose edges expand under $f_{*}$-iteration.
%\end{defn}

%\begin{rmk}Although immersions of trees are also embeddings, the term ``immersion'' is used since these maps can be thought of as inducing immersions on the quotient graph of groups.\end{rmk} 

There will be two possible ways of obtaining a relative immersion from a relatively irreducible restriction with a minimal stretch factor $\lambda$. If $\lambda = 1$, then an irreducible representative with stretch factor $\lambda$ is automatically a simplicial immersion. The next proposition shows how to construct an immersion when $\lambda > 1$. This construction is unique to nonsurjective endomorphisms because we require that the restriction be irreducible relative to a free factor system that carries the $[\phi]$-elliptic free factor system --- when $\phi$ is an automorphism, the $[\phi]$-elliptic free factor system is $\{ F \}$ and no such restriction exists.

\begin{prop}\label{tt2immersion} Let $\phi: F \to F$ be injective and $\mathcal A \prec \mathcal B$ be a chain of $[\phi]$-invariant free factor systems that carry the $[\phi]$-elliptic free factor system. If $\left.\phi\right|_{\mathcal B}$ is irreducible relative to $\mathcal A$ and $\lambda([\phi], \mathcal B, \mathcal A) > 1$, then there is an expanding irreducible $\mathcal A$-relative immersion for $\left.\phi\right|_{\mathcal B}$.
\end{prop}
\begin{proof}
Suppose $\phi:F \to F$ is injective, $\mathcal A \prec \mathcal B$ are $[\phi]$-invariant free factor systems that carry the $[\phi]$-elliptic free factor system,  $\left.\phi\right|_{\mathcal B}$ is irreducible relative to $\mathcal A$, and $f_{*}: T_{*} \to T_{*}$ is an expanding irreducible $(\mathcal B, \mathcal A)$-relative representative  for $\left.\phi\right|_{\mathcal B}$ with minimal stretch factor $\lambda(f_*) > 1$. Set $K = K(f_{*}) + C(f_{*})$ and $C = 2C(f_{*})$. %For any $k \ge 1$, we let $f_{*, k}: T_{*}(\phi^k(\mathcal B)) \to T_{*}(\phi^k(\mathcal B))$ be the homotopy restriction given by Lemma~\ref{hrest2}. In particular,
%\begin{enumerate}\item $f_{*, k}$ maps branch points to branch points and any natural edge to a branch point or immersed path;
%\item $f_{*, k}$ is induced by folding $f_{*}$ and then applying univalent/tightening/bivalent homotopies; and
%\item $K(f_{*, k}) = K$ and $C(f_{*, k}) = C$.\end{enumerate}
By Proposition~\ref{homotopyrest}, there is a natural representative $f_{*, k}:T_*(\phi^k(\mathcal B)) \to T_*(\phi^k(\mathcal B))$ for $\left. \phi\right|_{\phi^k(B)}$ with Lipschitz and cancellation constants $K(f_{*, k}) = K$ and $C(f_{*,k})=C$ respectively for all $k \ge 1$. 

The first part of the proof proceeds as a relativized version of the proof of Theorem~\ref{expequiv1}. By Proposition~\ref{relexpand}, we may fix $k \gg 0$ such that the set of natural edges $\mathcal L_0$ in $T_*(\phi^k(\mathcal B))$ longer than $C \cdot K^{N-1}$ is not empty, where $N = 3 \cdot \mathrm{rank}(F) - 3$. Choose $\mathcal L$ to be the set of all natural edges that eventually get mapped over those in $\mathcal L_0$ by $f_{*, k}$ and call $\mathcal L$ the long natural edges. As $f_{*, k}$ is $K$-Lipschitz and there are at most $N$ orbits of natural edges in $T_*(\phi^k(\mathcal B))$, the long natural edges are longer that $C$. %By irreducibility of the relative representative $g_{*, k}$, the set $\mathcal L$ contains all natural edges of $Y_{*, k}$. 
Injectivity of $\phi$ implies $\left.\phi\right|_{\phi^k(\mathcal B)}$ is conjugate to $\left.\phi\right|_{\mathcal B}$. So $\left.\phi\right|_{\phi^k(\mathcal B)}$ is irreducible relative to $\phi^k(\mathcal A)$, $\lambda(f_{*})$ is the minimal stretch factor for $\left.\phi\right|_{\phi^k(\mathcal B)}$ relative to $\phi^k(\mathcal A)$, 
and the short natural edges of $T_*(\phi^k(\mathcal B))$ form an orbit-closed $f_{*, k}$-invariant subforest with bounded components. 

Collapse a maximal $f_{*,k}$-invariant subforest of $T_*(\phi^k(\mathcal B))$ that has bounded components and contains the short natural edges then forget the bivalent vertices; this induces a minimal $\phi^k(\mathcal A)$-relative representative $g_{*, k}: Y_{*, k} \to Y_{*, k}$ for $\left.\phi\right|_{\phi^k(\mathcal B)}$. The map $g_{*, k}$ is an irreducible $\phi^k(\mathcal A)$-relative representative for $\left.\phi\right|_{\phi^k(\mathcal B)}$ (Lemma~\ref{irredequiv}) and $\lambda(g_{*, k}) \ge \lambda(f_{*})$ by the minimality of $\lambda(f_{*})$. So $g_{*, k}$ is an expanding irreducible $\phi^k(\mathcal A)$-relative representative.

Since the lifts in $T_*(\phi^k(\mathcal B))$ of
all edges in $Y_{*,k}$  are longer than the cancellation constant $C$, there is no folding in $g_{*, k}$ --- otherwise, there would be folding in $f_{*, k}$ identifying paths longer than its cancellation constant, absurd. Thus, $g_{*, k}$ is an immersion. By injectivity of $\phi$, we can view $Y_{*, k}$ as a $(\mathcal B, \mathcal A)$-forest and $g_{*, k}$ as an expanding irreducible $\mathcal A$-relative immersion for $\left.\phi\right|_{\mathcal B}$.
\end{proof}

We are now ready to state and prove our base case for the construction. In light of the previous proposition, the point is that a restriction $\left.\phi\right|_{\mathcal B}$ that is irreducible relative to the $[\phi]$-elliptic free factor system $\mathcal A$ will satisfy $\lambda([\phi], \mathcal B, \mathcal A) > 1$.

\begin{prop}\label{irredrelimm} Let $\phi: F \to F$ be injective and $\mathcal A \prec \mathcal B$ be a chain of $[\phi]$-invariant free factor systems where $\mathcal A$ is the $[\phi]$-elliptic free factor system. If $\left.\phi\right|_{\mathcal B}$ is irreducible relative to $\mathcal A$, then there is an expanding irreducible $\mathcal A$-relative immersion for $\left.\phi\right|_{\mathcal B}$.
\end{prop}

\begin{proof}
Let $\phi: F\to F$ be injective and $\left.\phi\right|_{\mathcal B}$ be irreducible relative to the $[\phi]$-elliptic free factor system $\mathcal A$. Then there is an irreducible $\mathcal A$-relative representative $f_{*}:T_{*} \to T_{*}$  for $\left.\phi\right|_{\mathcal B}$ with stretch factor $\lambda(f_*) = \lambda([\phi], \mathcal B, \mathcal A) \ge 1$ (Proposition~\ref{reltt}). We say $B_i \in \mathcal B$ is loxodromic if $T_i \in T_*$ is not a point, i.e., $B_i$ contains a loxodromic element; similarly, the component $B_i\backslash T_i$ of the graph of groups $\mathcal B \backslash T_*$ is loxodromic if $B_i$ is loxodromic. If $\lambda(f_{*}) = 1$, then the induced map on the loxodromic components of $\mathcal B \backslash T_{*}$ is a graph isomorphism. So for some $k \ge 1$, if $A \in \mathcal A$ is carried by a loxodromic $B \in \mathcal B$, then $A$ is $\phi^k$-invariant. By Proposition~\ref{canonical}, the subset of all $A \in \mathcal A$ carried by the loxodromic component of $\mathcal B$ form a $[\phi]$-fixed free factor subsystem. As $f_{*}$ induces a graph isomorphism on the loxodromic components of $\mathcal B \backslash T_{*}$ and these components' vertex groups form a $[\phi]$-fixed free factor system, we get that $f_{*}$ is surjective when restricted to the unbounded components of the forest $T_{*}$ and the loxodromic components of $\mathcal B$ form a $[\phi]$-fixed free factor system. This is a contradiction since $[\phi]$-periodic free factors are elliptic (Propositions~\ref{maxfixed} and \ref{canonical}).
Therefore, $\lambda(f_{*}) > 1$ and the result follows from Proposition~\ref{tt2immersion}.
\end{proof}

Specializing this proposition to the case where $\phi$ is irreducible and nonsurjective yields an alternate proof to a result due to Reynolds.

\begin{cor}[{\cite[Corollary~3.23]{Rey11}}]\label{irrednonsurj} If $\phi: F \to F$ is irreducible but not surjective, then $\phi$ is induced by an expanding irreducible graph immersion.
\end{cor}

\begin{rmk}This proof of Reynolds' result is a variation of our previous proof \cite[Theorem~4.5]{JPMb} with two crucial differences: 1) it makes no use of {\it limit trees} in the {\it compactification of outer space}; 2) the specialization of Proposition~\ref{tt2immersion} to irreducible nonsurjective endomorphisms need not invoke train track theory since we can use Lemma~\ref{expand} in place of Proposition~\ref{relexpand}.\end{rmk}

The next proposition is the induction step for our construction.

%{\color{blue}\begin{prop}\label{relexpansive2} If $\phi:F \to F$ is irreducible relative to $\mathcal B$ and there is an expanding $\phi$-immersion of $\mathcal B$ rel. $\mathcal A$, the elliptic free factor system for $\phi$, then $\phi$ is expansive relative to $\mathcal A$.\end{prop}}

\begin{prop}\label{steprelimm} Let $\phi:F \to F$ be injective, $\mathcal A$ be the $[\phi]$-elliptic free factor system, and $\mathcal A \prec \mathcal B \prec \mathcal C$ be a chain of $[\phi]$-invariant free factor systems. If there is an expanding $\mathcal A$-relative immersion for $\left.\phi\right|_{\mathcal B}$ and a $\mathcal B$-relative immersion for $\left.\phi\right|_{\mathcal C}$, then there is an expanding $\mathcal A$-relative immersion for $\left.\phi\right|_{\mathcal C}$.
\end{prop} 

Although the proof gets a bit technical, the idea is rather simple: a $\mathcal B$-relative immersion for $\left.\phi\right|_{\mathcal C}$ (top stratum) and an expanding $\mathcal A$-relative immersion for $\left.\phi\right|_{\mathcal B}$ (lower stratum) can be patched together via a {\it (relative) vertex blow-up} to get a minimal $\mathcal A$-relative representative $g_*:Y_* \to Y_*$ whose only possible folds would have to happen between a top and lower stratum edge of $Y_*$. As the restriction of $g_*$ to the lower stratum is an expanding immersion, we may assume the edges in the lower stratum are longer than the cancellation constant. This means no lower stratum edge is identified by a fold and so no folding in $g_*$ is possible. Thus, $g_*$ is a minimal $\mathcal A$-relative immersion for $\left.\phi\right|_{\mathcal C}$, which will be expanding if $\mathcal A$ is the $[\phi]$-elliptic free factor system.
\medskip 

Let $T_{\mathcal C}$ be a $(\mathcal C, \mathcal B)$-forest  and $T_{\mathcal B}$ be a $(\mathcal B, \mathcal A)$-forest. For any free factor $C_i\in \mathcal C$, let $\mathcal B_i$ the maximal subset of $\mathcal B$ that is carried by $C_i$. Replace the free factors of $\mathcal B$ with conjugates if necessary and assume $\mathcal B$ is also a free factor system of $\mathcal C$. In particular, the free factors $B \in \mathcal B_i$ are subgroups of the free factors $C_i \in \mathcal C$. Identifying the appropriate vertices of the graph of groups $\mathcal C \backslash T_{\mathcal C}$ with basepoints on the graph of groups $\mathcal B \backslash T_{\mathcal B}$ results in a graph of groups decomposition for $\mathcal C$ whose Bass-Serre forest $T_{*}$ is a $(\mathcal C, \mathcal A)$-forest that contains $T_{\mathcal B}$.  
We call $T_{*}$ the vertex blow-up of $T_{\mathcal C}$ with respect to $T_{\mathcal B}$. 

\begin{proof}[Proof of Proposition~\ref{steprelimm}] Let $\phi:F \to F$ be injective, $\mathcal B \prec \mathcal C$ be a chain of $[\phi]$-invariant free factor systems that carry the $[\phi]$-elliptic free factor system $\mathcal A$. 
Suppose $f_{\mathcal B}: T_{\mathcal B} \to T_{\mathcal B}$  is an expanding $\mathcal A$-relative immersion for $\left.\phi\right|_{\mathcal B}$
and $f_{\mathcal C}: T_{\mathcal C} \to T_{\mathcal C}$ is a $\mathcal B$-relative immersion  for $\left.\phi\right|_{\mathcal C}$ then define $T_{*}$ to be the vertex blow-up of $T_{\mathcal C}$ with respect to $T_{\mathcal B}$. 
The edges of $T_{*}$ are of two types: the lower stratum, which are edges that are contained in the $\mathcal C$-orbit of $T_{\mathcal B}$, and the top stratum, which are the remaining edges. %By construction, collapsing the lower stratum recovers $T_{\mathcal C}$.

%, {\color{red} a restriction $\left.\phi\right|_{\mathcal C}$ be irreducible relative to $\mathcal B$,} %By Proposition~\ref{reltt}, there is an irreducible $\mathcal B$-relative train track $f_{*}: T_{*} \to T_{*}$ for $\left.\phi\right|_{\mathcal C}$. If $\lambda(f_{*}) = 1$, then $f_\infty = f_{*}$ is a $\mathcal B$-relative immersion for $\left.\phi\right|_{\mathcal C}$. If $\lambda(f_{*}) > 1$, then there is an expanding $\mathcal B$-relative immersion $f_\infty$ for $\left.\phi\right|_{\mathcal C}$ by Proposition~\ref{tt2immersion}. In either case, we get a $\mathcal B$-relative immersion $f_{\infty}$ for $\left.\phi\right|_{\mathcal C}$ defined on a $(\mathcal C,\mathcal B)$-forest $T_\infty$. 

Let $f_{*}: T_{*} \to T_{*}$ be  a minimal $\mathcal A$-relative representative  for $\left.\phi\right|_{\mathcal C}$ such that $T_{\mathcal B}$ is an $f_*$-invariant subforest, the restriction of $f_{*}$ to $T_{\mathcal B}$ agrees with $f_{\mathcal B}$, and $f_{*}$ induces $f_{\mathcal C}$ upon collapsing the lower stratum. For all $k \ge 1$, set $T_{*}(\phi^k(\mathcal C))$ and $T_{*}(\phi^k(\mathcal B))$ to be the minimal subforests of $T_{*}$ for $\phi^k(\mathcal C)$ and $\phi^k(\mathcal B)$ respectively. Similarly, define the minimal subforest $T_{\mathcal B}(\phi^k(\mathcal B)) \subset T_{\mathcal B}$. By the inclusion of $T_{\mathcal B}$ in $T_{*}$, we get a simplicial identification of $T_{*}(\phi^k(\mathcal B))$ with $T_{\mathcal B}(\phi^k(\mathcal B))$. However, we want to consider these two forests differently with respect to their branch points and natural edges. In particular, there may be branch points of $T_{*}(\phi^k(\mathcal C))$ that are bivalent when considered as points on the subforest $T_{*}(\phi^k(\mathcal B))$. So by {\it ``natural edges of $T_{*}(\phi^k(\mathcal B))$''}, we mean those inherited from the parent forest $T_{*}(\phi^k(\mathcal C))$; on the other hand, by {\it ``natural edges of $T_{\mathcal B}(\phi^k(\mathcal B))$''}, we do mean exactly that. Under the identification of the two forests, the natural edges of $T_{*}(\phi^k(\mathcal B))$ partition any natural edge of $T_{\mathcal B}(\phi^k(\mathcal B))$ into at most $2 \mathcal X$ segments, where $\mathcal X = \mathrm{rank}(F)-1$.

Since $f_{\mathcal C}:T_{\mathcal C} \to T_{\mathcal C}$ is a $\mathcal B$-relative immersion for $\left.\phi\right|_{\mathcal C}$, the restrictions of $f_{\mathcal C}$ to $T_{\mathcal C}(\phi^k(\mathcal C))$ are $\phi^k(\mathcal B)$-relative immersions $f_{\mathcal C, k}:T_{\mathcal C}(\phi^k(\mathcal C)) \to T_{\mathcal C}(\phi^k(\mathcal C))$ for $\left.\phi\right|_{\phi^k(\mathcal C)}$ that are conjugate to $f_{\mathcal C}$. 
%As $f_*$ agrees with $f_{\mathcal B}$ on $T_{\mathcal B}$, the restriction of $f_*$ to $T_{*}(\phi^k(\mathcal B))$ is exactly $f_{\mathcal B, k}$ for all $k \ge 1$. Similarly define $\phi^k(\mathcal B)$-relative immersions $f_{\mathcal C, k}:T_{\mathcal C}(\phi^k(\mathcal C)) \to T_{\mathcal C}(\phi^k(\mathcal C))$ for $\left.\phi\right|_{\phi^k(\mathcal C)}$ that are conjugate to $f_{\mathcal C}$.
As $f_*$ induces $f_{\mathcal C}$ upon collapsing the lower stratum, any edges in $f_{*}(T_*(\phi^k(\mathcal C)))$ but not  $T_*(\phi^k(\mathcal C))$ must be in the lower stratum and the restriction of $f_*$ to $T_{*}(\phi^k(\mathcal C))$ induces $f_{\mathcal C, k}$ upon collapsing the lower stratum, i.e., the $\phi^k(\mathcal C)$-orbit of $T_*(\phi^k(\mathcal B))$. By Proposition~\ref{homotopyrest}, there is an $\phi^k(\mathcal A)$-relative natural representative $f_{*,k}:T_*(\phi^k(\mathcal C)) \to T_*(\phi^k(\mathcal C))$ for $\left.\phi\right|_{\phi^k(\mathcal C)}$ with Lipschitz and cancellation constants $K = K(f_*) + C(f_*)$ and $C = 2C(f_*)$ respectively. Furthermore, $T_*(\phi^k(\mathcal B))$ is an $f_{*,k}$-invariant subforest and $f_{*,k}$ still induces $f_{\mathcal C, k}$ upon collapsing the lower stratum. As $f_*$ agrees with the immersion $f_{\mathcal B}$ on $T_{\mathcal B}$,  $f_{*,k}$ differs from $f_{\mathcal B}$ on $T_{*}(\phi^k(\mathcal B))$ by a homotopy supported in the natural edges of $T_{\mathcal B}(\phi^k(\mathcal B))$. 

Since $f_{\mathcal B}:T_{\mathcal B} \to T_{\mathcal B}$ is an expanding $\mathcal A$-relative immersion for $\left.\phi\right|_{\mathcal B}$, the minimal subforest $T_{\mathcal B}(\phi^k(\mathcal B))$ has natural edges whose lengths are all exponential in $k$.
%We pointed out earlier (third paragraph of proof) that the lengths of natural edges in $T_{\mathcal B}(\phi^k(\mathcal B))$ are expanding exponentially. 
Fix $k \gg 0$ such that all natural edges of $T_{\mathcal B}(\phi^k(\mathcal B))$ are longer than $ 2\mathcal X \cdot C \cdot K^{3\mathcal X-1}$. By the pigeonhole principle, each natural edge of $T_{\mathcal B}(\phi^k(\mathcal B))$ contains a natural edge of $T_{*}(\phi^k(\mathcal B))$ longer than $C \cdot K^{3 \mathcal X}$. Let $\mathbb G_k$ be the directed graph of natural edges of the $f_{*,k}$-invariant subforest $T_{*}(\phi^k(\mathcal B))$ where a directed edge $E_i \to E_j$ corresponds to $f_{*, k}$ mapping $E_i$ over $E_j$. Set $\mathcal S_0$ to be those natural edges with length at most $C$ and $\mathcal S$ to be those natural edges with directed path from $\mathcal S_0$ in $\mathbb G_k$ and their $\phi^k(\mathcal C)$-translates; these lower stratum natural edges will be the {\it short} natural edges of $T_*(\phi^k(\mathcal C))$. Since $f_{*, k}$ is $K$-Lipschitz and the shortest path between any two natural edges in $\mathbb G_k$ has $3\mathcal X$ natural edges, all the short natural edges have length at most $ C \cdot K^{3\mathcal X-1}$. So the short natural edges $\mathcal S$ form an orbit-closed $f_{*, k}$-invariant lower stratum subforest of $T_*(\phi^k(\mathcal C))$ with bounded components as $\mathcal S$ does not cover any natural edge of $T_{\mathcal B}(\phi^k(\mathcal B))$.
\medskip

Collapsing the short natural edges of $T_{*}(\phi^k(\mathcal B))$ induces a map $g_{*, k}': Y_{*, k}' \to Y_{*, k}'$ with the same cancellation constant $C(g_{*, k}') = C$. Now iteratively collapse pretrivial edges until the induced map $g_{*, k}:Y_{*, k} \to Y_{*, k}$ has none. As the collapses are supported in the lower stratum, the new map $g_{*, k}$ still induces the immersion $f_{\mathcal C, k}$ upon collapsing the rest of the lower stratum and, as a result, folding in $g_{*, k}$ may only occur between initial segments of natural edges of $Y_{*, k}$ whose $g_{*, k}$-images are a lower stratum natural edge. However, all natural edges in the lower stratum of $Y_{*, k}$ are longer than $C$ by construction and so no folding in $g_{*, k}$ is possible by bounded cancellation, i.e., $g_{*, k}$ is an immersion.

Collapsing a maximal invariant subforest with bounded components and forgetting bivalent vertices if necessary, we may assume $g_{*, k}:Y_{*, k}\to Y_{*, k}$ is a minimal $\phi^k(\mathcal A)$-relative immersion for $\left.\phi\right|_{\phi^k(\mathcal C)}$. By the injectivity of $\phi$, we can view $g_{*, k}$ as a minimal $\mathcal A$-relative immersion for $\left.\phi\right|_{\mathcal C}$. It remains to show that every edge of $Y_{*, k}$ expands under $g_{*, k}$-iteration. 

For a contradiction, suppose there is an edge of $Y_{*, k}$ whose $g_{*, k}$-iterates have uniformly bounded length. Since $g_{*, k}$ is minimal, the non-expanding edges in the graph of groups $\mathcal C \backslash Y_{*, k}$ contain a setwise fixed subgraph $\mathcal F$ that carries a loxodromic element. The subgraph $\mathcal F$ is a free splitting of a $[\phi]$-invariant free factor system. Recall that the $[\phi]$-elliptic free factor system $\mathcal A$  decomposes as a union of the maximal $[\phi]$-fixed free factor system and free factors that eventually get mapped into this fixed system (Proposition~\ref{canonical}). 
By construction, the point stabilizers of $Y_{*, k}$ are conjugates of $\mathcal A$. So any vertex of the fixed graph $\mathcal F$ is labelled by either the trivial group or a free factor of the maximal $[\phi]$-fixed free factor system. Thus $\mathcal F$ is a free splitting of a $[\phi]$-fixed free factor system that carries some loxodromic element. However, Propositions~\ref{maxfixed} and \ref{canonical} imply that all $[\phi]$-fixed free factors systems are elliptic --- a contradiction.
%So $g_{*, k}$ is an expanding $\mathcal A$-relative immersion for $\left.\phi\right|_{\mathcal C}$.
% as an $f_{\mathcal C}$-invariant, induces $f_\infty$ on $T_{\infty}$, and is the restriction of 
%{\color{red}the folding in the induced map $f_{\mathcal C}:T_{\mathcal C} \to T_{\mathcal C}$ only happens between edges of different kind}. \color{blue} [Collapse some maximal invariant forest to get (expanding?) immersion...]
%{\color{red} Define ``unstable'' (partial folds only) and argue that $f_{*}$ is homotopic to immersion?}
%\color{blue} Since $f_{\mathcal B}$ is expanding and the images of the initial segments of edges of $T_\infty$ are predetermined and fixed, the folding in all iterates of $f'_{*}$ is contained a forest properly contained in a fundamental domain of the $(\mathcal B', \mathcal A)$-forest $T_{*}'$; in other words, the folding is contained in a forest of $\mathcal B' \backslash T_{*}'$. Just as in the proof Theorem~\ref{irredrelimm}, $(\left.\phi\right|_{\mathcal B'})^n$ has an expanding immersion defined on the tree $T'_{(\left.\phi\right|_{\mathcal B'})^m}$ for some $m, n \ge 1$.
\end{proof}

We are now ready to inductively construct expanding relative immersions.

\begin{thm}\label{expimmersion} If $\phi:F \to F$ is injective but not surjective, then there is an expanding $\mathcal A$-relative immersion for $\phi$, where $\mathcal A$ is the $[\phi]$-elliptic free factor system.
\end{thm}
\begin{proof}
Suppose $\phi: F \to F$ is injective but not surjective. By Proposition~\ref{canonical}, the $[\phi]$-elliptic free factor system $\mathcal A$ is proper. The naive approach is to assume there exists a chain $\mathcal A = \mathcal B_0 \prec \cdots \prec \mathcal B_n = \{ F \}$ in the poset of $[\phi]$-invariant free factor system such that the restrictions $\left.\phi\right|_{\mathcal B_{m+1}}$ are irreducible relative to $\mathcal B_m$ for all $m \ge 1$. This assumption is typical when working with automorphisms. For each restriction $\left.\phi\right|_{\mathcal B_{m+1}}$, if the minimal stretch factor $\lambda([\phi], \mathcal B_{m+1}, \mathcal B_m) = 1$, then there is automatically a $\mathcal B_m$-relative immersion for $\left.\phi\right|_{\mathcal B_{m+1}}$; and if $\lambda([\phi], \mathcal B_{m+1}, \mathcal B_m) > 1$, then there is an expanding $\mathcal B_m$-relative immersion for $\left.\phi\right|_{\mathcal B_{m+1}}$ by Proposition~\ref{tt2immersion}. In either case, there is a $\mathcal B_m$-relative immersion for $\left.\phi\right|_{\mathcal B_{m+1}}$. By Proposition~\ref{irredrelimm}, $\lambda([\phi], \mathcal B_{1}, \mathcal B_0) > 1$ and there is an expanding $\mathcal B_0$-relative immersion for $\left.\phi\right|_{\mathcal B_{1}}$. By inductively patching these immersions together using Proposition~\ref{steprelimm}, we get an expanding $\mathcal B_0$-relative immersion for $\phi$ and we are done. Unfortunately, since $\phi$ is not surjective, it could be that no chain $\mathcal A = \mathcal B_0 \prec \cdots \prec \mathcal B_n = \{ F \}$ satisfies the naive assumption we made at the start. Recall that our definition of  $\left.\phi\right|_{\mathcal B_{m+1}}$ being irreducible relative to $\mathcal B_m$ presupposed $\phi^{-1}\cdot \mathcal B_m = \mathcal B_m$. Fortunately, this is a minor complication that can be easily addressed. The proof follows the approach described above closely but uses a chain with slightly weaker conditions on it.
\medskip

We first construct a chain $\mathcal A = \mathcal B_0 \prec \cdots \prec \mathcal B_n = \{ F \}$ in the poset of $[\phi]$-invariant free factor system that we will induct on. Let $\mathcal A  \prec \mathcal B_1$ be a chain of $[\phi]$-invariant free factor systems such that $\left.\phi\right|_{\mathcal B_1}$ is irreducible relative to $\mathcal A$. Suppose $\mathcal B_m$ has been constructed for some $m \ge 1$ and let $\mathcal C \succeq \mathcal B_m$ be the maximal free factor system in the chain $\mathcal B_m \preceq \phi^{-1} \cdot \mathcal B_m \preceq \phi^{-2} \cdot \mathcal B_m \preceq \cdots$ of $[\phi]$-invariant free factor systems. If $\phi^{-1} \cdot \mathcal B_m = \mathcal B_m = \mathcal C$, then let $\mathcal B_m \prec \mathcal B_{m+1}$ be a chain of $[\phi]$-invariant free factor systems  such that $\left.\phi\right|_{\mathcal B_{m+1}}$ is irreducible relative to $\mathcal B_m$. If $\mathcal B_m \prec \mathcal C$, then let $\mathcal B_m \prec \cdots \prec \mathcal B_{m+k} = \mathcal C$ be the chain of $[\phi]$-invariant free factor systems such that $\mathcal B_{m+i} = \phi^{-i}\cdot \mathcal B_m$ for $1 \le i \le k$. %The chain is constructed this way since our definition of irreducibility relative to a $[\phi]$-invariant free factor system $\mathcal B$ presupposed $\phi^{-1} \cdot \mathcal B = \mathcal B$.

We proceed by inducting on the resulting chain between $\mathcal A$ and $\{ F \}$. 
For the base case, $\left.\phi\right|_{\mathcal B_1}$ is irreducible relative to $\mathcal A$; therefore, there is an expanding $\mathcal A$-relative immersion for $\left.\phi\right|_{\mathcal B_1}$ by Proposition~\ref{irredrelimm}.
For our induction hypothesis, suppose that there is an expanding $\mathcal A$-relative immersion $f_{\mathcal B_m}: T_{\mathcal B_m} \to T_{\mathcal B_m}$ for $\left.\phi\right|_{\mathcal B_m}$ for some $m \ge 1$. By our construction of the chain, either $\left.\phi\right|_{\mathcal B_{m+1}}$ is irreducible relative to $\mathcal B$ or $\phi(\mathcal B_{m+1})$ is carried by $\mathcal B_m$. We deal with these two cases separately.

{Case 1.} Suppose $\left.\phi\right|_{\mathcal B_{m+1}}$ is irreducible relative to $\mathcal B_m$. 
By Proposition~\ref{reltt}, there is an irreducible $\mathcal B_m$-relative representative $f_{*}: T_{*} \to T_{*}$ for $\left.\phi\right|_{\mathcal B_{m+1}}$ with minimal stretch factor. If $\lambda(f_{*}) = 1$, then $f_{\mathcal B_{m+1}} = f_*$ is a $\mathcal B_m$-relative simplicial immersion for $\left.\phi\right|_{\mathcal B_{m+1}}$. If $\lambda(f_{*}) > 1$, then there is an expanding $\mathcal B$-relative immersion $f_{\mathcal B_{m+1}}$ for $\left.\phi\right|_{\mathcal B_{m+1}}$ by Proposition~\ref{tt2immersion}. In either case, we get a $\mathcal B_m$-relative immersion $f_{\mathcal B_{m+1}}:T_{\mathcal B_{m+1}}\to T_{\mathcal B_{m+1}}$ for $\left.\phi\right|_{\mathcal B_{m+1}}$ defined on a $(\mathcal B_{m+1},\mathcal B_m)$-forest $T_{\mathcal B_{m+1}}$. Thus, there is an expanding $\mathcal A$-relative immersion for $\left.\phi\right|_{\mathcal B_{m+1}}$ by the induction hypothesis and Proposition~\ref{steprelimm}.

{Case 2.} Now suppose $\phi(\mathcal B_{m+1})$ is carried by $\mathcal B_m$. Let $T_{\mathcal B_m}(\phi(\mathcal B_{m+1})) \subset T_{\mathcal B_m}$ be the minimal subforest of $\phi(\mathcal B_{m+1})$ and $T_{\mathcal B_{m+1}} = T_{\mathcal B_m}(\phi(\mathcal B_{m+1}))$ be the same forest after forgetting bivalent vertices. By injectivity of $\phi$, we may consider $T_{\mathcal B_{m+1}}$ as a $(\mathcal B_{m+1}, \mathcal A)$-forest that comes with a natural $\left.\phi\right|_{\mathcal B_{m+1}}$-equivariant immersion $g: T_{\mathcal B_{m+1}} \to T_{\mathcal B_m}$. 
Since  $f_{\mathcal B_m}: T_{\mathcal B_m} \to T_{\mathcal B_m}$ is an immersion, we can identify a subdivision of $T_{\mathcal B_m}$ with the minimal subforest $T_{\mathcal B_{m+1}}(\mathcal B_m) \subset T_{\mathcal B_{m+1}}$ of $\mathcal B_m$. Composing $g$ with the subdivision and inclusion $T_{\mathcal B_{m+1}}(\mathcal B_m) \subset T_{\mathcal B_{m+1}}$ gives an $\mathcal A$-relative immersion $f_{\mathcal B_{m+1}}: T_{\mathcal B_{m+1}} \to T_{\mathcal B_{m+1}}$ for $\left.\phi\right|_{\mathcal B_{m+1}}$, which is expanding since its image lies in $T_{\mathcal B_{m+1}}(\mathcal B_m)$ and its restriction to $T_{\mathcal B_{m+1}}(\mathcal B_m)$ is the expanding $\mathcal A$-relative immersion $f_{\mathcal B_{m}}$ after forgetting bivalent vertices.
\end{proof}

It will follow from bounded cancellation that expanding relative immersion from the theorem is in fact canonical.

\begin{prop}\label{unique} Suppose $\phi:F \to F$ is injective and there is an expanding $\mathcal A$-relative immersion for $\phi$, where $\mathcal A$ is a $[\phi]$-invariant free factor system. Then there is a unique expanding $\mathcal A$-relative immersion for $\phi$.
\end{prop} 
\begin{proof} Let $\phi:F \to F$ be injective, $\mathcal A$ a $[\phi]$-invariant free factor system. Suppose $f:T \to T$ and $f':T' \to T'$ are expanding $\mathcal A$-relative immersions for $\phi$. The goal is to show that $T$ and $T'$ are equivariantly homeomorphic. Let $g:T \to T'$ be an equivariant tree map, i.e., $\psi$-equivariant tree map with $\psi:F \to F$ being the identity automorphism; such a map always exists between $(F, \mathcal A)$-trees. By taking restrictions to the $\phi^k(F)$-minimal subtrees $T(\phi^k(F))$ and $T'(\phi^k(F))$ and applying deformation retractions, we get a $\phi^k(F)$-equivariant tree map $g_k:T(\phi^k(F)) \to T'(\phi^k(F))$ with cancellation constant $C(g)$. By Corollary~\ref{relnatural}, we assume $g_k$ is an equivariant natural tree map with cancellation constant $2C(g)$. 

Recall that $f'$ is an expanding $\mathcal A$-relative immersion for $\phi$, so all the natural edges of $T'(\phi^k(F))$ are longer than $2C(g)$ for large enough $k$. Fix $k\gg 0$, then no folding occurs in $g_k$ by bounded cancellation. So $g_k$ is a forest collapse and $T'(\phi^k(F))$ has at most the same number of orbits of branch points as $T(\phi^k(F))$. By the same argument, $T(\phi^k(F))$ is an equivariant forest collapse of $T'(\phi^k(F))$ and has at most the same number of orbits of branch points as $T'(\phi^k(F))$. This implies the two minimal subtrees have the same number of orbits of branch points and $g_k$ is an equivariant homeomorphism. Since $f$ and $f'$ are immersions, $T$ and $T'$ are $\phi^k$-equivariantly homeomorphic to $T(\phi^k(F))$ and $T'(\phi^k(F))$ respectively. Therefore, $T$ and $T'$ are equivariantly homeomorphic by injectivity of $\phi$. So $f$ and $f'$ are homotopic immersions and hence agree up to isotopy.
\end{proof}

\begin{cor}\label{uniqueexp} If $\phi:F \to F$ is injective but not surjective, then there is a unique expanding $\mathcal A$-relative immersion for $\phi$, where $\mathcal A$ is the $[\phi]$-elliptic free factor system.
\end{cor}

This gives us a complete characterization of when an injective endomorphism is induced by a unique expanding graph immersion.

\begin{cor}\label{expequiv2} Let $\phi:F \to F$ be an injective endomorphism. Then the following conditions are equivalent:\begin{enumerate}
\item $[\phi]$ is induced by a unique expanding graph immersion;
\item $[\phi]$ is induced by an expanding graph immersion;
\item the trivial conjugacy class is the only conjugacy class with an infinite $[\phi]$-tail;
\item the trivial system is the only $[\phi]$-fixed free factor system.
\end{enumerate}\end{cor}
\begin{proof}
We leave the implication $(2) \implies (3)$ as an exercise. By Theorem~\ref{expequiv1}, we have $(3) \iff (4)$. It remains to show $(3) \implies (1)$. Suppose the trivial conjugacy class is the only conjugacy class with an infinite $[\phi]$-tail. So $\phi$ is not surjective and the $[\phi]$-elliptic free factor system is trivial. By Corollary~\ref{uniqueexp}, there is a unique expanding $\phi$-equivariant immersion $f:T \to T$ defined on a free $F$-tree $T$, i.e., $[\phi]$ is induced by a unique expanding immersion on the marked graph $F \backslash T$.
\end{proof}

\nonumsec{Interlude}
%%%%%%%%%%%%%%%%%%%%%%%%%%%%%%%%%%%%%%%%%%%%%%%%%%%%%%%%%%
Let us summarize the main results of the first part of the paper.

\begin{sum*}If $\phi:F \to F$ is injective but not surjective, then there is:
\begin{enumerate}
\item a unique maximal proper $[\phi]$-fixed free factor system $\mathcal A$; (Proposition~\ref{maxfixed})
\item a unique maximal proper $[\phi]$-invariant free factor system $\mathcal A^* \succeq \mathcal A$ such that $\phi^k(\mathcal A^*)$ is carried by $\mathcal A$ for some $k \ge 0$; after replacing the free factors of $\mathcal A$ with conjugates if necessary, we can assume $\mathcal A \subset \mathcal A^*$; (Proposition~\ref{canonical})
\item a unique expanding $\mathcal A^*$-relative immersion for $\phi$. (Corollary~\ref{uniqueexp})
\end{enumerate}
\end{sum*}
We call the free factor system $\mathcal A^*$ the $[\phi]$-elliptic free factor system. By {\it expanding} immersion, we mean %there are no orbit-closed invariant subforest with bounded components and 
every edge of the tree grows under iteration. 
\medskip

As a corollary of this summary, we have the following characterization: $[\phi]$ is induced by an expanding graph immersion if and only if
the trivial system is the only $[\phi]$-fixed free factor system (Corollary~\ref{expequiv2}). We will now contextualize these results and, especially, this corollary.
%\medskip
In our previous work \cite{JPMa}, we determined exactly when the mapping torus of an expanding graph immersion has a word-hyperbolic fundamental group.
\begin{thm*}[{\cite[Theorem~6.3]{JPMa}}]Let $\phi:F \to F$ be induced by an expanding graph immersion. $F*_\phi$ is word-hyperbolic if and only if it has no $BS(1,d)$ subgroups for $d \ge 2$.
\end{thm*}

By Reynolds' result (Corollary~\ref{irrednonsurj}), we knew this theorem applied to all irreducible nonsurjective endomorphisms but, at the time, there was no known algebraic characterization of the general class of endomorphisms induced by expanding graph immersions. It is clear that nontrivial $[\phi]$-periodic subgroups are an obstruction to $\phi$ being induced by an expanding graph immersion. Corollary~\ref{expequiv2} means nontrivial fixed free factor systems are the only (essential) obstruction. So the above theorem can be restated as:

\begin{thm*}Let $\phi:F \to F$ be an injective endomorphism with no nontrivial fixed free factor system. $F*_\phi$ is word-hyperbolic if and only if it has no $BS(1,d)$ subgroups for $d \ge 2$.
\end{thm*}

The interesting thing about the restatement is that it is purely algebraic, i.e., there is no mention of topological maps. Since fixed free factor systems correspond to {\it free-by-cyclic} subgroups $F' \rtimes \mathbb Z$ of $F*_\phi$, the restatement also suggests how to extend our previous work to all injective nonsurjective  endomorphisms.

 In the next sections, we use expanding relative immersions to relativize our previous work in \cite{JPMa}. For instance, the sequence of lemmas/propositions in the next section is essentially identical to the sequence in \cite[Section~3]{JPMa}. In Section~\ref{secHypEnds}, we extend our previous theorem to all injective endomorphisms. However, this will {\it not} constitute an alternate proof of Brinkmann's theorem \cite{Bri00}: $F\rtimes \mathbb Z$ is word-hyperbolic if and only if it has no $\mathbb Z^2$ subgroups. We assume Brinkmann's theorem as the base case of our generalization.
\medskip
 
Before diving into the details, we give an illustration of the consideration in the second part of this \paper. To prove that $BS(1,d)$ subgroups are the only obstruction to word-hyperbolicity, the main tool we use is the Bestvina-Feighn combination theorem \cite{BF92}: as long as the group in question satisfies the {\it annuli flaring condition}, it will be word-hyperbolic. See Section~\ref{secHNNExts} for the exact statement. For non-examples, we now explain why/how $BS(1,d)$ groups fail the annuli flaring condition.
\medskip

Consider two Baumslag-Solitar groups $BS(1,d)$ with $d = 1,2$. Note that $BS(1,1) = \mathbb Z^2$. Fix the presentation $BS(1,d) = \langle a, t~|~t^{-1}at = a^d \rangle$ and notice that the group is a mapping torus of an endomorphism $\psi$ of the cyclic group $\langle a \rangle$. So $BS(1,d)$ acts on a simplicial tree $\mathcal T$ with point and edge stabilizers conjugate to $\langle a \rangle$ and the quotient graph $BS(1,d)\backslash \mathcal T$ has one vertex and edge. The edge has a natural orientation inherited from the nontrivial loop corresponding to $t$ and we can lift the orientation to an equivariant orientation of $\mathcal T$. Under this orientation, $\mathcal T$ is like a rooted tree: every vertex has exactly $d$ incoming edges and one outgoing edge. Choose a vertex $\star$ on the axis of $t$ whose stabilizer by $\langle a \rangle$ and define 
$\mathcal T_{\langle a \rangle}$ to be the family tree of $\star$-descendants, i.e., union of directed edge-paths that terminate on $\star$.

\begin{figure}[ht]
 \centering 
 \includegraphics[scale=1.5]{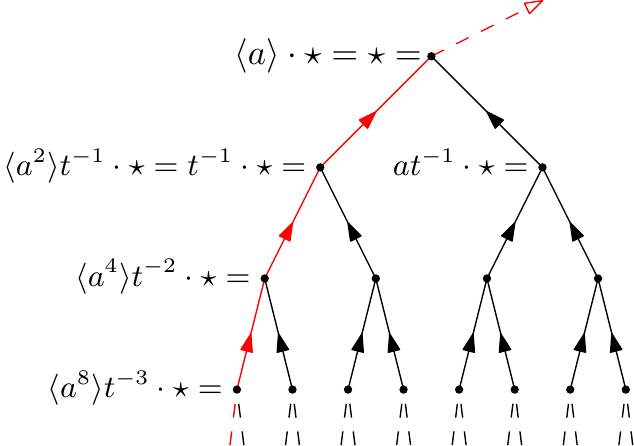}
 \caption{Portion of the trees $\mathcal T_{\langle a \rangle} \subset \mathcal T$ when $d=2$. The line on the left is the axis of $t$.}
 \label{figbstree}
\end{figure}

A {\it class of annuli} of length $n$ in $BS(1,d)$ is an ordered pair $(g, p_g)$ where $g \in BS(1,d)$ is nontrivial and $p_g$ is a geodesic of length $n$ in $\mathcal T$ fixed by $g$. Since $g$ is elliptic, then (up to conjugacy) we may assume $g \in \langle a \rangle$, $\star  \in p_g$, and $p_g \subset \mathcal T_{\langle a \rangle}$. 
The orientation on $\mathcal T$ implies there are two types of annuli: 1) unidirectional --- those whose geodesics are directed paths terminating on $\star$; 2) bidirectional --- those whose geodesics are a wedge of two directed paths terminating on $\star$. When $d=1$, $\mathcal T$ is a line and all annuli are unidirectional. Loosely speaking, flaring means expanding exponentially towards some end. It should be intuitive that annuli will not flare because the associated endomorphism $\psi$ is the identity map. When $d = 2$, $\mathcal T$ is a regular trivalent tree and annuli can be unidirectional or bidirectional. Since the associated endomorphism $\psi$ is represented by a degree 2 map of the circle, the unidirectional annuli will flare. However, due to the branching in $\mathcal T$, it is possible to wedge two flaring unidirectional annuli along their expanding ends to get a bidirectional annulus that shrinks exponentially towards both its ends.

The objective is to rule out both of these behaviors in the setting of a mapping torus $F*_\phi$ that has no $BS(1,d)$ subgroups with $d \ge 1$ and thus conclude that the mapping torus satisfies the annuli flaring condition. Here is an outline of the second part of the paper: \begin{enumerate}
\item In Section~\ref{secPullbacks}, we prove that if $F*_\phi$ has no $BS(1,d)$ subgroups with $d \ge 2$, then there is a uniform bound on one of the two directed pieces of any bidirectional annulus. Thus long bidirectional annuli are unidirectional up to bounded error. However, since annuli are not defined yet in this section, everything is done in terms of {\it pullbacks}, which are special classes of bidirectional annuli. 
\item In Section~\ref{secHypEnds}, we define annuli and give the correspondence with pullbacks. We then use expanding relative immersions and Brinkmann's theorem to prove the section's main result: if $F*_\phi$ has no $\mathbb Z^2$ subgroups, then long unidirectional annuli flare. Thus, together with the previous section, $F*_\phi$ satisfies the annuli flaring condition if it has no $BS(1,d)$ subgroups with $d \ge 1$. 
\item Finally, in Section~\ref{secHNNExts}, we extend the result to HNN extensions of free groups over free factors $F*_A$ by showing that there is a canonical finite collection of mapping tori $F'*_{\phi'} \le F*_A$ that carry all long annuli of $F*_A$. Intuitively speaking, the action of $F*_A$ on its Bass-Serre tree is {\it acylindrical} relative to $\{ F_i'*_{\phi_i'}\}$. So $F*_A$ satisfies the annuli flaring condition if and only if the subgroups $\{ F_i'*_{\phi_i'}\}$ satisfy the condition.
\end{enumerate}

%%%%%%%%%%%%%%%%%%%%%%%%%%%%%%%%%%%%%%%%%%%%%%%%%%%%%%%%%%
\nonumsec{Geometry of ascending HNN extensions}
\subsection{Iterated pullbacks}\label{secPullbacks}
%%%%%%%%%%%%%%%%%%%%%%%%%%%%%%%%%%%%%%%%%%%%%%%%%%%%%%%%%%

In our previous work \cite[Sections~3]{JPMa}, topological pullbacks for a graph immersion $f: \Gamma \to \Gamma$ were used to give sufficient conditions for $\pi_1(f) = \phi:F \to F$ to have an invariant nonfixed {\it cyclic subgroup system}. The goal of this section is to drop the immersion hypothesis and give sufficient conditions that apply to all injective endomorphisms of $F$.

Suppose immersions $f_1: \Gamma_1 \to \Gamma$ and $f_2: \Gamma_2 \to \Gamma$ induce inclusions of free groups $H_1, H_2 \le F$ respectively. Then components of the core {\it (topological) pullback} of $(f_1, f_2)$, also known as {\it fibered products} are in one-to-one correspondence with nontrivial intersection $H_1 \cap g H_2 g^{-1}$ as $g$ ranges over $(H_1, H_2)$-double coset representatives of $H_1 \backslash F / H_2$. For a graph immersion $f: \Gamma \to \Gamma$ that induces an endomorphism $\phi: F \to F$, we get a one-to-one correspondence between components of the core pullback of $(f^k, f^k)$ and nontrivial conjugacy classes $[\phi^k(F) \cap g \phi^k(F) g^{-1}]$ as $[[g]]$ ranges over $\phi^k(F)$-double cosets for all $k \ge 1$. We will not define pullbacks topological maps since we will be working (semi-)algebraically in this section.

%We need to introduce some terminalogy and notation to facilitate this discussion. 
Given subgroups $H_1, H_2 \le F$, we define the {\bf (algebraic) pullback} of $(H_1, H_2)$, denoted by $H_1 \wedge H_2$, to be the set of all nontrivial {\bf components} $[H_1 \cap g H_2 g^{-1}]$ as $[[g]]$ ranges over the $(H_1, H_2)$-double cosets in $F$. Note that pullbacks are {\it not} subgroup systems since they are sets of conjugacy classes of subgroups. When $H_1$ and $H_2$ are finitely generated, then their pullback is a finite set. %The empty set is the {\bf trivial  pullback}.
For an injective endomorphism $\phi: F \to F$ and $k \ge 1$, define the {\bf iterated (algebraic) pullbacks of $\boldsymbol \phi$} to be $\Lambda_k(\phi) = \phi^k(F) \wedge \phi^k(F)$. Furthermore, the iterated pullback depends only on the outer class $[\phi]$ and will be denoted later by $\Lambda_k[\phi]$. In this section, we will simply write $\Lambda_k$ for brevity.

The {\bf (algebraic) mapping torus} of an injective endomorphism $\phi$ is the ascending HNN extension of $F$ with the presentation \( F*_\phi = \langle \, F, t ~|~ t^{-1}xt = \phi(x), \forall x \in F\, \rangle.\)
We now give a third description of $\Lambda_k$ that will actually be the working description. 
The mapping torus $F*_\phi$ is the fundamental group of a circle of groups with one vertex group $F$ and edge group $F$. The edge monomorphisms for this circle of groups are the identity map $id_F:F \to F$ and endomorphism $\phi: F \to F$ and the corresponding Bass-Serre tree $\mathcal T$ has one orbit of edges and vertices. The tree also comes with a natural orientation where each vertex has exactly one outgoing oriented edge and the {\bf stable letter} $t \in F*_\phi$ acts on its directed axis with positive translation. 

By construction, there is a vertex $\star$ of $\mathcal T$ whose stabilizer in $F*_\phi$ is exactly $F$. Let $\mathcal T_F$ be the full subtree of $\mathcal T$ rooted at $\star$. Colloquially, $\mathcal T_F$ is the family tree of $\star$-descendants. We will use the action of $F*_\phi$ on $\mathcal T$ to index the vertices of $\mathcal T_F$ by cosets in $F$. The root $\star$ is indexed $gF$, where $gF=F$ is the coset of $F$ in $F$. The $k^\text{th}$-generation vertices $gt^{-k} \cdot \star$ are indexed by the cosets $g\phi^k(F)$ in $F$. The intersection $g_1 \phi^k(F) g_1^{-1} \cap g_2 \phi^k(F) g_2^{-1}$ corresponds to the simultaneous stabilizer of the ordered pair of vertices indexed by $g_1 \phi^k(F)$ and $g_2 \phi^k(F)$; the intersection is conjugate in $F$ to the stabilizer of the vertices indexed by $\phi^k(F)$ and $g \phi^k(F)$, where $g = g_1^{-1}g_2$. The intersection is also conjugate (in $F$) to the stabilizers of the vertices indexed by $\phi^k(F)$ and $hg \phi^k(F)$, where $h \in \phi^k(F)$. Thus, $F$-orbits of $k^\text{th}$-generation ordered pairs of vertices are indexed by $\phi^k(F)$-double cosets in $F$ and conjugacy classes of their nontrivial stabilizers are the components of $\Lambda_k$.
%Thus if we want to index the (conjugacy classes of) nontrivial stabilizers of ($F$-orbits of) two $k^\text{th}$-generation vertices in $\mathcal T_F$, we get the components of $\Lambda_k$.
 %The components of $\Lambda_k$ are nontrivial stabilizers (up to conjugacy) of ($\phi^k(F)$-orbits of) geodesics in $\mathcal T_F$ whose endpoints are equidistant from $\star$. Colloquially, we can think of $\mathcal T_F$ as a $\star$-descendants family-tree. In this view, the components of $\Lambda_k$ are nontrivial conjugacy classes of simultaneous stabilizers for a pair of $k^\text{th}$-generation vertices in $\mathcal T_F$ whose only common ancestor is $\star$. 

There is a function $\left.\phi\right|_{\Lambda_k}:\Lambda_k \to \Lambda_{k+1}$ induced by $\phi$, given algebraically by \[ [\phi^k(F) \cap g\phi^k(F) g^{-1}] \quad \mapsto \quad [\phi^{k+1}(F) \cap \phi(g)\phi^{k+1}(F) \phi(g)^{-1}]. \]
Graphically, the function %is the conjugation by $t$ in $F*_\phi$ that 
maps (the conjugacy class of) the stabilizer of $( g_1 t^{-k}, g_2t^{-k} )\cdot \star$ to (the conjugacy class of) the stabilizer of $t^{-1}\cdot ( g_1t^{-k}, g_2 t^{-k} ) \cdot \star$. %This graphical description makes the next lemma an obvious observation. %since the action of $t^{-1}$ is injective on the set of $\phi^k(F)$-orbits.
\begin{lem} $\left.\phi\right|_{\Lambda_k}:\Lambda_k \to \Lambda_{k+1}$ as a function on the set of components is injective.
\end{lem}
\begin{proof}[Graphical proof]If $t^{-1}\cdot( g_1 t^{-k}, g_2 t^{-k}) \cdot \star$ and $t^{-1}\cdot( g_1' t^{-k}, g_2' t^{-k} ) \cdot \star$ are in the same $F$-orbit, then they are in fact in the same $\phi(F)$-orbit (descendants of $t^{-1} \cdot \star$). So $(g_1 t^{-k}, g_2 t^{-k} ) \cdot \star$ and $( g_1't^{-k}, g_2' t^{-k} ) \cdot \star$ are in the same $F$-orbit by the action of $t$.
%Fix $k \ge 1$ and let $g_1$ and $g_2$ be $\phi^k(F)$-double coset representatives such that $\phi(g_1)$ and $\phi(g_2)$ were in the same $\phi^{k+1}$-double coset. So there are $x, y \in \phi^{k+1}(F)$ such that $x \phi(g_1) y = \phi(g_2)$. Equivalently, there are $x', y' \in \phi^k(F)$ such that $\phi(x') \phi(g_1) \phi(y') = \phi(g_2)$. Since $\phi$ is injective, we get $x'g_1 y' = g_2$. So $g_1$ and $g_2$ are in the same $\phi^k(F)$-double coset.
\end{proof}

One can give an algebraic proof of the lemma that replaces the action of $t$ with the injectivity of $\phi$. The lemma implies there is a chain of injections $\Lambda_0 \to \Lambda_1 \to \Lambda_2 \to \cdots$. Furthermore, the restriction to each component is an isomorphism since it is a conjugation by $t$ in $F*_\phi$. We will be mainly interested in the set compliment $\hat \Lambda_{k} = \Lambda_{k} \setminus \phi(\Lambda_{k-1})$.
%, i.e., the components of $\Lambda_{k+1}$ that are not in $\phi(\Lambda_{k})$. 
Equivalently, %$\hat \Lambda_{k}$ has the following description:
\( \hat \Lambda_{k} = \{\, [\phi^{k}(F) \cap g \phi^{k}(F) g^{-1}] \in \Lambda_{k} ~:~ g\notin \phi(F) \,\}. \)
Graphically, $\hat \Lambda_{k}$ is the set of components that stabilizes $F$-orbits of $k^\text{th}$-generation pairs of vertices in $\mathcal T_F$ whose only common ancestor is $\star$.
We might say {\bf iterated pullbacks stabilize} if $\hat \Lambda_k = \emptyset$ for some $k$. The image $\phi(F)$ is {\it malnormal}  in $F$ if and only if $\hat \Lambda_1 = \emptyset$. Iterated pullback stability is a sort of generalization of malnormality for $\phi^k(F)$ with respect to $\phi$.

\begin{lem}\label{strictincr} Suppose $\phi:F \to F$ is injective and $k\ge 1$. If $\hat \Lambda_{k}$ is empty, then so is $\hat \Lambda_{k+1}$; and if $\hat \Lambda_{k}$ has only cyclic components, then $\hat \Lambda_{k+1}$ is empty or has only cyclic components.
\end{lem}
\begin{proof} 
There is an ``inclusion'' of components, $\hat \Lambda_{k+1} \preccurlyeq \hat \Lambda_k$, since a stabilizer of a $(k{+}1)^\text{th}$-generation pair of vertices is contained in a stabilizer of a $k^\text{th}$-generation pair of vertices.
%induced by \[ \phi^{k+1}(F) \cap g \phi^{k+1}(F) g^{-1} \le \phi^{k}(F) \cap g \phi^{k}(F) g^{-1}.\] 
So $\hat \Lambda_k = \emptyset$ implies $\hat \Lambda_{k+1} = \emptyset$. Suppose $\hat \Lambda_k$ has cyclic components, then $\hat \Lambda_{k+1}$ is empty or has cyclic components as the subgroups of a cyclic group are trivial or cyclic. 
\end{proof}

The {\bf reduced rank} of a nontrivial finite rank free group $H$ is $\mathrm{rr}(H) = \mathrm{rank}(H)-1$ and the reduced rank of a pullback $H_1 \wedge H_2$, where $H_1, H_2$ are finitely generated free groups, is the sum of the reduced ranks of components in $H_1 \wedge H_2$. The latter is denoted by $\mathrm{rr}(H_1 \wedge H_2)$. 
%We define the reduced rank of the trivial group/pullback to be $0$. 
Since the restriction $\left.\phi\right|_{\Lambda_k}$ gives natural isomorphisms of the components, the chain of injections produces a nondecreasing sequence of positive integers \[\mathrm{rank}(F) - 1 = \mathrm{rr}(F) \le \mathrm{rr}(\Lambda_1) \le \mathrm{rr}(\Lambda_2) \le  \cdots. \]

Observe that $\mathrm{rr}(\Lambda_i)=\mathrm{rr}(\Lambda_{i+1})$ if and only if $\hat \Lambda_{i+1}$ is empty or has cyclic components only. By Lemma~\ref{strictincr}, the sequence becomes constant once two consecutive entries are equal. Walter Neumann used topological pullbacks of Stallings graphs to bound the reduced ranks of algebraic pullbacks and improve Hanna Neumann's bound \cite{HN57}. %This bound is the bedrock of this section.

\begin{thm}[{\cite[Proposition~2.1]{WN90}}]\label{shnc} If $H_1, H_2 \le F$ are nontrivial finitely generated subgroups, then $\mathrm{rr}(H_1 \wedge H_2) \le 2\, \mathrm{rr}(H_1) \,\mathrm{rr}(H_2)$.
\end{thm}

\begin{rmk} Although this bound is weaker than the {\it Strengthened Hanna Neumann Conjecture}, now Friedman-Mineyev's Theorem \cite{Fri15, Min12}, it is sufficient for our purposes. W. Neumann's proof is an elementary Euler characteristic computation for the topological pullback.
\end{rmk}

The uniform bound given by W. Neumann's theorem allows us to conclude that the sequence of reduced ranks given above is eventually constant, and hence, past a certain point we are adding only cyclic components or nothing at all to the iterated pullbacks.

\begin{lem}\label{cyclopullback} If $\phi: F \to F$ is injective, then either $\hat \Lambda_k$ has cyclic components for all $k \ge 2\,\mathrm{rr}(F)^2$ or $\hat \Lambda_k = \emptyset$ for some $k$.
\end{lem}
\begin{proof}
Theorem~\ref{shnc} gives us a uniform bound on the reduced ranks of the iterated pullbacks $\mathrm{rr}(\Lambda_k) \le 2\,\mathrm{rr}(\phi^k(F))^2 = 2 \, \mathrm{rr}(F)^2$ for all $k \ge 1$.
By Lemma~\ref{strictincr} and the uniform bound on the nondecreasing sequence of reduced ranks, $\mathrm{rr}(\Lambda_k)$ are all equal for $k \ge k_0 = 2 \, \mathrm{rr}(F)^2$. Therefore, $\hat \Lambda_{k_0}$ is empty or has only cyclic components. The lemma follows by inductively applying Lemma~\ref{strictincr}.
\end{proof}

We say $\phi:F \to F$ has an {\bf invariant cyclic subgroup system with index $\boldsymbol{d \ge 1}$} if there is an integer $k \ge 1$, element $x\in F$, and nontrivial cyclic subgroup $\langle c \rangle \le F$ such that $\phi^k(\langle c \rangle) \le x \langle c \rangle x^{-1}$ and has index $d$. Note that this is actually a property of the outer endomorphism $[\phi]$. We can now give the main result of this section:

\begin{prop}\label{cycloinv} Let $\phi:F \to F$ be injective. If $\hat \Lambda_k$ is not empty for all $k \ge 1$, then $[\phi]$ has an invariant cyclic subgroup system with index $d \ge 2$.
\end{prop}
\begin{proof}
Let $k_0 = 2\, \mathrm{rr}(F)^2$ and $\phi: F \to F$ be an injective endomorphism such that $\hat \Lambda_k$ is not empty for all $k \ge 1$. By Lemma~\ref{cyclopullback}, $\hat \Lambda_k$ has cyclic components for $k \ge k_0$. 

So far nothing in the section has used relative immersions but our main motivation for constructing them was this proposition. Note that $\hat \Lambda_1 \neq \emptyset$ automatically implies $\phi$ is not surjective. So $\phi$ is injective but not surjective. By Theorem~\ref{expimmersion}, there is an expanding $\mathcal A^*$-relative immersion $f: T \to T$ for $\phi$, where $\mathcal A^*$ is the $[\phi]$-elliptic free factor system and $T$ is an $(F, \mathcal A^*)$-tree. Recall that the equality $\phi^{-1}\cdot \mathcal A^* = \mathcal A^*$ is part of our definition of $\mathcal A^*$-relative representatives. An element of $F$ is {\it elliptic} if it fixes a point on $T$, i.e., its conjugacy class is carried by $\mathcal A^*$. By construction of the $[\phi]$-elliptic free factor system (Proposition~\ref{canonical}), there is an $m \ge 1$ such that $[\phi^{k}(g)]$ is carried by the maximal $[\phi]$-fixed free factor system $\mathcal A$ for all $k \ge m$ and elliptic $g \in F$. %By $[\phi]$-invariance of $\mathcal A$, we get that $[\phi^k(g)]$ is carried by $\mathcal A$ for all $k \ge m$ and elliptic $g\in F$. 
 As $\mathcal A$ is a $[\phi]$-fixed free factor system, the components of $\phi^k(\mathcal A)$ are free factors of $F$ for all $k$.

Suppose $[\phi^k(F) \cap g\phi^k(F) g^{-1}] \in \hat \Lambda_k$ for some $k \ge k_1 = \max(k_0, m+1)$. As $k \ge k_0$, this component is cyclic and we may assume it has a representative $\phi^k(F) \cap g\phi^k(F) g^{-1}$ generated by a nontrivial element $\phi^k(x) \in \phi^k(F)$. In particular, there exists $\phi^k(y) \in \phi^k(F)$ and $g \in F \setminus \phi(F)$ such that $\phi^k(x)  = g \phi^k(y) g^{-1}$. We first show that $x$ is {\it loxodromic}, i.e., not elliptic. Suppose $x \in F$ is elliptic, i.e., $\mathcal A^*$ carries $[x]$. Then some free factor $\phi^k(A) \in \phi^k(\mathcal A)$ carries $[\phi^k(x)] = [\phi^k(y)]$ as $k\ge m+1$. So $y$ must be elliptic too since $\phi^{-1}\cdot \mathcal A^* = \mathcal A^*$ and, as noted in the previous paragraph, $\phi^m(A)$ carries $\phi^m(x)$ and $\phi^m(y)$. Thus, up to conjugation in $F$, we may assume $\phi^m(x), \phi^m(y) \in \phi^m(A)$ and, up to conjugation in $\phi^{k-m}(F)$, we may assume $\phi^k(x), \phi^k(y) \in \phi^k(A)$. {\it Malnormality} of the free factor $\phi^k(A)$ implies $g \in \phi^{k-m}(F)$. But $k-m \ge 1$ leads to the contradiction $g \in \phi(F)$.
%Since $g \notin \phi(F)$, we have $g \notin \phi^k(A)$. But this contradicts the {\it malnormality} of free factor systems: elements in a free factor system, e.g. $x,y$ in $\phi^k(\mathcal A)$, cannot be conjugated by an element not in the free factor system, e.g. $g$ not in $\phi^k(\mathcal A)$. 
Therefore, $x$ is loxodromic. The integer $k \ge k_1$ and component $[\,\langle \phi^k(x) \rangle\,] \in \hat \Lambda_k$ were arbitrary, so all components of $\hat \Lambda_k$ are loxodromic for $k \ge k_1$.
Recall from the proof of Lemma~\ref{strictincr} that there is an infinite descending chain:
\[ \hat \Lambda_1 \succcurlyeq \hat \Lambda_2 \succcurlyeq \hat \Lambda_3 \succcurlyeq \cdots \]
Since $\hat \Lambda_{k_1}$ has finitely many components and the components are all cyclic, there is a cyclic component in $\hat \Lambda_{k_1}$ which ``carries'' some component of $\hat \Lambda_k$ for all $k \ge k_1$. Suppose this component has a representative %$\phi^{k_1}(F) \cap g\phi^{k_1}(F) g^{-1}$ 
generated by $c = \phi^{k_1}(x) \in \phi^{k_1}(F)$. Then for all $k \ge k_1$, there is a cyclic component of $\hat \Lambda_k$ with a representative generated by $\phi^k(x_k) \in \phi^k(F)$ such that $\langle c \rangle$ carries $\langle \phi^k(x_k) \rangle$.
If we let $\alpha \subset T$ be the axis for element $c$, then the previous sentence implies there are sequences of element $(x_k)_{k \ge k_1}$ and $(t_k)_{k \ge k_1}$ such that the (unoriented) axes of $\phi^k(x_k)$ are all translates $t_k \cdot \alpha$ of the (unoriented) axis $\alpha$. 

For any $k \ge k_1$, replace $x_k$ with its inverse if necessary so that the action of $\phi^k(x_k)$ on its axis is coherent (respects orientation) with the action of $c$ on $\alpha$. By passing to a strictly increasing subsequence $(k_i)_{i\ge 1}$, we may assume there is an oriented edge $e$ of $T$ such that the axes $\alpha_{k_i}$ of $(x_{k_i})_{i \ge 1}$ all contain a translate of $e$. We now pass to the graph of groups $\Gamma = F \backslash T$ in order to avoid mentions of translates and orbits. The axis $\alpha$ will project to an immersed loop $\bar \alpha$ in $\Gamma$ representing $c$ and axes $\alpha_{k_i}$ project to immersed loops $\bar \alpha_{k_i}$ that represent $x_{k_i}$ and whose $\bar f^{k_i}$-image is a power $\bar \alpha^{d_i}$ up to rotation/cyclic homotopy, where $d_i \ge 1$ and $\bar f: \Gamma \to \Gamma$ is the immersion induced by $f: T \to T$. The edge $e$ projects to an edge $\bar e$ that is contained in all the loops $\bar \alpha_{k_i}$ for $i \ge 1$

 The proof now mimics the proof of \cite[Proposition~3.11]{JPMa}. Since $f$ is an immersion, it maps axes in $T$ onto axes and $\bar f$ maps immersed loops in $\Gamma$ to immersed loops. So $\bar f^{k_i}(\bar e)$ is a subpath of the immersed loop $\bar f^{k_i}(\bar \alpha_{k_i}) \simeq \bar \alpha^{d_i}$ for all $i$, and since $f$ is expanding, $\bar f^{k_i}(\bar e)$ contains arbitrarily long powers of $\bar \alpha$ as $i \to \infty$. Set $N$ to be the number of subpaths of $\bar \alpha$ (up to rotation) that are also loops. Choose $n \gg 0$ such that $\bar f^{k_n}(\bar e)$ contains the loop $\bar \alpha^{N+1}$ as a subpath. Then $\bar f^{k_{n+1}}(\bar e)$ is a subpath of $\bar \alpha^{d_{n+1}}$ that contains the loop $\bar f^{k_{n+1}-k_n}(\bar \alpha^{N+1})$ as a subpath. In fact, for all positive integers $j \le N+1$, the loop $\bar f^{k_{n+1}-k_n}(\bar \alpha^j)$ is a subpath of $\bar \alpha^{d_{n+1}}$. Thus, there is a sequence of loops $(\epsilon_j)_{j=1}^{N+1}$ that are subpaths of $\bar \alpha$ and strictly increasing positive integers $(s_j)_{j=1}^{N+1}$ such that $\bar f^{k_{n+1}-k_n}(\bar \alpha^j)\cdot \epsilon_j$ is $\bar \alpha^{s_j}$ up to rotation.
 By definition of $N$ and pigeonhole principle, some $\epsilon_t = \epsilon_{t'} = \epsilon$ for some $t < t'$ and $\bar f^{k_{n+1}-k_n}(\bar \alpha^{t'-t})$ is $\bar \alpha^{s_{t'} - s_{t}}$ up to rotation. 
  Lifting back to the $(F, \mathcal A)$-tree $T$, we find that $f^{k_{n+1}-k_n}$ maps the axis $\alpha$ to a translate of itself. So $\phi^{k_{n+1}-k_n}(c)$ is conjugate to a nontrivial power $c^d$ and $d \ge 2$ since $f$ is expanding.
\end{proof}
\begin{rmk} A careful examination of the proof reveals that it can be made ``more effective'' with the pigeonhole principle. For any expanding relative immersion $f$, we can construct a specific number $k = k(f)$ for which 
$ \hat \Lambda_k \neq \emptyset$ implies $\phi$ has an invariant cyclic subgroup system with index $d \ge 2$. Thus, one would not have to check infinitely many iterated pullbacks of $\phi$ to know that an invariant cyclic subgroup system with index $d \ge 2$ exists. Of course, for this to be useful, we need to know whether finding the $[\phi]$-elliptic free factor system and expanding relative immersion for $\phi$ can be made effective. Finally, we remark that the converse of Proposition~\ref{cycloinv} holds but we omit the proof as it is not needed for the rest of the paper.
\end{rmk}

In the next section, we will actually be using the contrapositive of the proposition: if $\phi$ is injective and has no invariant cyclic subgroup system with index $d \ge 2$, then iterated pullbacks of $\phi$ stabilize. In this case, we get control on the types of annuli in the mapping torus $F*_\phi$, which allows us to prove the main theorem: $F*_\phi$ is word-hyperbolic if $\phi$ additionally has no fixed cyclic subgroup system, i.e., invariant cyclic subgroup system with index $d = 1$.

%%%%%%%%%%%%%%%%%%%%%%%%%%%%%%%%%%%%%%%%%%%%%%%%%%%%%%%%%%
\subsection{Hyperbolic endomorphisms}\label{secHypEnds}
%%%%%%%%%%%%%%%%%%%%%%%%%%%%%%%%%%%%%%%%%%%%%%%%%%%%%%%%%%
We are finally ready to put all the major pieces together. The first piece involves understanding the relationship between {\it annuli} in the mapping torus $F*_\phi$ and iterated pullbacks of $\phi$. The second piece involves building on Brinkmann's theorem to show that {\it atoroidal} injective endomorphisms are {\it hyperbolic}.
In our previous work \cite{JPMa}, we used these two pieces to give sufficient conditions for the mapping torus to be word-hyperbolic.

\begin{thm}[{\cite[Theorem~6.4]{JPMa}}]\label{blueprint} If $f: \Gamma \to \Gamma$ is a based-hyperbolic graph map and all strictly bidirectional annuli in its mapping torus $M_f$ are shorter than some integer, then $\pi_1(M_f)$ is word-hyperbolic.
\end{thm}

We will define the new terms in the theorem as we go. Suppose $\phi: F \to F$ is an injective endomorphism and $f:\Gamma \to \Gamma$ is its topological representative. Recall that the (topological) mapping torus of $f$ is the quotient space $M_f = \left(\Gamma \times [0,1] \right)/{\sim_f}$ with the identification $(x, 1) \sim_f (f(x),0)$ for all $x \in \Gamma$ and the algebraic mapping torus $F*_\phi$ is isomorphic to the fundamental group $\pi_1(M_f)$. The {\bf edge-space} of $M_f$ will be the cross-section in $M_f$ represented by $\Gamma \times \{ \frac{1}{2} \}$. 

Strictly bidirectional annuli of the mapping torus $M_f$ with length $2L$ can be thought of as the iterated pullbacks $\hat \Lambda_L$ of $\phi$ but it does take a bit of work to give the correspondence. Fix a basepoint $\nu \in S^1$. For integers $L_1 < L_2$, an {\bf (topological) annulus} in $M_f$ of length $L = L_2 - L_1$ is a homotopy of loops $h:S^1 \times [L_1, L_2] \to M_f$ satisfying the following conditions:
\begin{enumerate}
\item it is transverse to the edge-space of $M_f$;
\item the $h$-preimage of the edge-space is $S^1\times ( [-L_1, L_2] \cap \mathbb Z )$;
\item for integers $i \in [L_1, L_2]$, the {\bf rings} of the annulus $h_i = h(\cdot, i):S^1 \to M_f$ are locally injective every where except possibly at the basepoint $\nu$;
\item and for the {\bf trace} of the basepoint $h^\nu = h(\nu, \cdot ):[L_1, L_2] \to M_f$, no subpath between consecutive integer coordinates $[i, i+1]$ is homotopic rel. endpoints into the edge-space.
\end{enumerate}
%A {\bf homotopy of annuli} $h, h':S^1 \times [L_1, L_2]$ is a homotopy $H:S^1 \times [L_1, L_2] \times [0, 1] \to M_f$ such that $H(\cdot, \cdot, 0)= h$, $H(\cdot, \cdot, 1)= h'$, and $H(\cdot, \cdot, t):S^1\times [L_1, L_2] \to M_f$ are all annuli.  
This is the definition used in \cite{JPMa}. In light of our last description of $\hat \Lambda_k$ in the previous section, we give alternative definition. Let $\mathcal T$ be the Bass-Serre tree for $F*_\phi$ and $\star$ be the point whose stabilizer is $F$. An {\bf (algebraic) annulus} $([\alpha], p_\alpha)$ in $F*_\phi$ of length $L \ge 2$ is a choice a nontrivial conjugacy class $[\alpha]$  in $F*_\phi$ and an orbit of a geodesic path in $p_\alpha \subset \mathcal T$ of length $L$ fixed by $\alpha$. Since elements of $[\alpha]$ act on $\mathcal T$ with fixed points, we can always choose a representative $\alpha \in F$. Technically, we have defined a {\it conjugacy class} of algebraic annuli but the distinction will not be relevant for us.

\begin{lem}\label{algandtop} Let $\phi:F \to F$ be injective and $f:\Gamma \to \Gamma$ be a topological representative for $[\phi]$. There is a one-to-one correspondence between (homotopy classes of) annuli in $M_f$ of length $L \ge 1$ and (conjugacy classes of) annuli in $F*_\phi$ of length $L+1 \ge 2$.
\end{lem}
\begin{proof}
Given a topological annulus $h$ in $M_f$ of length $L \ge 1$, then the generator of the image $\pi_1(h):\mathbb Z \to \pi_1(M_f)$ determines a conjugacy class $[\alpha]$ in $\pi_1(M_f) \cong F*_\phi$. Condition~(3) ensures $\pi_1(h)$ is injective and $\alpha$ is nontrivial. Let $\tilde h:\mathbb R \times [L_1, L_2] \to \tilde M_f$ be the lift of the annulus to the universal cover of $M_f$. Collapsing the $\tilde \Gamma$-direction of $\tilde M_f$ produces the Bass-Serre tree $\mathcal T$ and Condition~(2) ensures the induced map \( \bar h:\mathbb R \times ([L_1, L_2]\cap \mathbb Z) \to \mathcal T \) is constant on the first factor and its image is a collection of edge-midpoint; each ring $h_i$ determines a conjugacy class in the stabilizer of the corresponding edge-midpoint. By Conditions $(1)$ and $(4)$, the midpoints extend to a geodesic edge-path $p_\alpha$ in $\mathcal T$ of length $L+1$ fixed by $\alpha$. 

The other direction works in a similar fashion. For any conjugacy class $[\alpha]$ in $F*_\phi$ and two consecutive edge-midpoints in $\mathcal T_F$ fixed by $\alpha \in F$, we can construct an annulus of length $1$ as follows. Fix a basepoint in the edge-space and assume $\star$ is the vertex between the midpoints. If the midpoints are increasing/decreasing, then $\alpha = \phi(x) \in \phi(F)$ without loss of generality. Let $\sigma, \rho$ be based loop in the edge-space representing $x, \phi(x) \in F$ and $\tau$ a based loop in $M_f$ representing $t \in F*_\phi$. Then the based path $\sigma \cdot \tau \cdot \bar \rho \cdot \bar \tau $ is null-homotopic and can be extended to an annulus with ends $\sigma, \rho$ and trace $\tau$. If the midpoints are at the same height, then \( \alpha = \phi(x) \in \phi(F)\) and \(\alpha = g\phi(y) g^{-1} \in  g \phi(F) g^{-1}\) for some $g \notin \phi(F)$. Let $\sigma, \rho, \gamma$ be based loops in the edge-space representing $x, y, g \in F$ respectively and $\tau$ a based loop in $M_f$ representing $t \in F*_\phi$. So the based path $\sigma \cdot \tau \gamma \bar \tau \cdot \bar \rho \cdot \tau \bar \gamma \bar \tau $ is null-homotopic and can be extended to an homotopy between $\sigma, \rho$ satisfying Conditions~(1)-(3) and having trace $\tau \gamma \bar \tau$. This trace satisfies Condition~(4) because $t g t^{-1} \notin F$ and thus the homotopy is a topological annulus with ends $\sigma, \rho$. Given a geodesic path in $\mathcal T$ of length $L \ge 2$ fixed by $\alpha$, we can replace the path with a translate in $\mathcal T_F$ without affecting the class $[\alpha]$ and then construct a topological annulus in $M_f$ of length $L-1$ by concatenating the length 1 annuli from the preceding discussion. This concludes the correspondence between topological and algebraic annuli. 
\end{proof}

The natural orientation on $\mathcal T$ gives a dichotomy for annuli $([\alpha], p_\alpha)$ in $F*_\phi$:
\begin{enumerate}
\item all edges of $p_\alpha$ have the same orientation--- we say $\alpha$ is {\bf unidirectional}.
\item $p_\alpha$ switches from increasing to decreasing exactly once --- we say $\alpha$ is {\bf bidirectional}.
\end{enumerate}
Each vertex of $\mathcal T$ has exactly one outgoing edge and hence the geodesic $p_\alpha$ cannot switch from decreasing to increasing because $[F:F]=1$. For similar reasons, bidirectional annuli do not exist if and only if $\phi(F)$ is malnormal in $F$. The next proposition  generalizes this equivalence of bidirectional annuli (or lack thereof) and malnormality.

 An annulus $([\alpha], p_\alpha)$ in $F*_\phi$ is {\bf strictly bidirectional} if the switch from increasing to decreasing occurs at the midpoint of $p_\alpha$. %Evidently, every bidirectional annulus has a maximal strictly bidirectional {\it subannulus} that comes from restricting to the appropriate maximal subsegment of the geodesic path for the annulus. 
\begin{lem}\label{biannulibound}Let $\phi:F \to F$ be injective. For any integer $L \ge 1$, the mapping torus $F*_\phi$ has a strictly bidirectional annulus of length $2L$ if and only if $\hat \Lambda_L[\phi]$ is not empty.
\end{lem}
\begin{proof} %Algebraic annuli and the graphical description of $\hat \Lambda_L$ make this a trivial observation. 
If there is a strictly bidirectional annulus $([\alpha], p_\alpha)$ in $F*_\phi$ of length $2L$, then we may assume the midpoint of $p_\alpha$ is $\star$ after replacing $p_\alpha$ with a translate. Then the stabilizer of $p_\alpha$ contains $\alpha$ and so it is nontrivial. Since the stabilizer of $p_\alpha$ is the stabilizer of its (ordered) endpoints, the conjugacy class of the stabilizer is a component of $\hat \Lambda_L[\phi]$.

If $\hat \Lambda_L[\phi]$ is not empty, then some path in $\mathcal T_F$ of length $2L$ with midpoint at $\star$ has a nontrivial stabilizer. Choose a nontrivial element $\alpha$ in this stabilizer and $([\alpha], p_\alpha)$ is a strictly bidirectional annulus in $F*_\phi$ of length $2L$.
\end{proof}

Let $\phi:F \to F$ be injective and $f: \Gamma \to \Gamma$ be a topological representative for $\phi$. If $[\phi]$ has no invariant cyclic subgroup system with index $d \ge 2$, then there is an integer $L \ge 1$ for which $\hat \Lambda_L[\phi]$ is empty (Proposition~\ref{cycloinv}) and all strictly bidirectional annuli in $M_f$ are shorter than $2L$ (Lemma~\ref{biannulibound}). This sets up the second hypothesis of Theorem~\ref{blueprint}.

As for the first hypothesis, we begin by defining {\it (based-)\,hyperbolicity}. For a real number $\lambda  > 1$ and integer $n \ge 1$, we say a graph map $f: \Gamma \to \Gamma$ is {\bf (based-)\,$\boldsymbol{(\lambda, n)}$-hyperbolic} if all (based) loops $\sigma: S^1 \to \Gamma$ (with the basepoint mapped to a vertex) satisfy the inequality \[\lambda\, |f^n(\sigma)| \le \max(\,|f^{2n}(\sigma)|, |\sigma|\,) ~\] 
where $|\cdot|$ is the combinatorial length after tightening; whether tightening respects a basepoint (based homotopy) or not (free homotopy) will be apparent from the context. %These conditions will be referred to as $(\lambda, n)$-Conditions~\ref{cond1}/\ref{cond2} respectively.

When a graph map is (based-)\,$(\lambda', n)$-hyperbolic for some $\lambda'>1, n\ge 1$, then it is (based-)\,$(\lambda^k, nk)$-hyperbolic for all $k \ge 1$ and $\lambda \in (1, \lambda']$. So the constants can be omitted and when we do need them, we can assume $\lambda > 1$ is any preferred integer. As defined, hyperbolicity is a property of the homotopy class $[f]$. Meanwhile, based-hyperbolicity is a property of the map $f$. %In the setting we will be interested in, the former will imply the latter.

A graph map $f: \Gamma \to \Gamma$ is {\bf atoroidal} if it is $\pi_1$-injective and there is no nontrivial loop $\sigma$ in $\Gamma$ and integer $k \ge 1$ such that $f^k(\sigma) \simeq \sigma$. This again is a property of the homotopy class $[f]$. Bestvina-Feighn-Handel showed, as a step in \cite[Theorem~5.1]{BFH97}, that hyperbolic atoroidal homotopy equivalences are based-hyperbolic. Their argument is reproduced here, modified to drop the $\pi_1$-surjectivity assumption. This allows us to consider the growth rate of loops without basepoints for the rest of the section

\begin{lem}\label{tobased} If the graph map $f: \Gamma \to \Gamma$ is atoroidal and $(3, n)$-hyperbolic, then it is based-$(2, n')$-hyperbolic. %Thus hyperbolicity implies based-hyperbolicity for atoroidal maps.
\end{lem}
\noindent To avoid context-ambiguity in the proof, we use $\lVert \cdot \rVert$ for lengths of free homotopy classes of loops and $|\cdot |$ for lengths of loops rel. basepoints. However, the distinction is not needed after the proof as all loops afterwards will be considered up to free homotopy.
\begin{proof} Suppose $f: \Gamma \to \Gamma$ is atoroidal and $(3,n)$-hyperbolic for some integer $n \ge 1$. Set $M$ to be the maximum length of $f^k(s)$ rel. basepoint over all embedded based loops $s$ in $\Gamma$ for $k \in \{0, n, 2n\}$. 

Suppose $|f^n(\sigma)| \ge 4M$ for some immersed based loop $\sigma$ and pick an embedded based loop $s$ with same basepoint as $\sigma$ so that the concatenation $s \cdot \sigma$ is an immersed loop, i.e., $\lVert s\cdot \sigma \rVert = |s| + |\sigma|$. As the graph map $f$ is $(3,n)$-hyperbolic, we get \[ 3 \, \lVert f^n(s \cdot \sigma)\rVert \le \max(\, \lVert f^{2n}(s \cdot \sigma)\rVert, \lVert s \cdot \sigma\rVert \,).\]
For a concatenation of based loops $\rho_1 \cdot \rho_2$, we get $\left|\,|\rho_1| - |\rho_2|\,\right| \le \lVert \rho_1 \cdot \rho_2 \rVert \le |\rho_1| + |\rho_2|$. %Although obvious, we highlight the {\it triangle inequality} because we will implicitly use it a few times in the following calculations. 

Case~1. If $\lVert f^{2n}(s \cdot \sigma)\rVert \ge 3\, \lVert f^n(s \cdot \sigma) \rVert $, then
\[\begin{aligned}|f^{2n}(\sigma)| \ge \lVert f^{2n}(s \cdot \sigma) \rVert - |f^{2n}(s)| &\ge 3\, \lVert f^n(s \cdot \sigma) \rVert - M  \\
&\ge 3\,|f^n(\sigma)| - 3\,|f^n(s)| - M \\
&\ge 3\,|f^n(\sigma)| -4M \\ 
&\ge 2\,|f^n(\sigma)|.\end{aligned}\]

Case~2. If $\lVert s \cdot \sigma\rVert \ge 3\,\lVert f^n(s \cdot \sigma)\rVert$, then
\[|\sigma| = \lVert s \cdot \sigma \rVert - |s| \ge 3\, \lVert f^n(s \cdot \sigma) \rVert - M \ge 3\,|f^n(\sigma)| - 4M \ge 2\,|f^n(\sigma)|\]
Combining both cases: $2\,|f^n(\sigma)| \le \max(\,|f^{2n}(\sigma)|, |\sigma| \,)$. If $|f^{nk}(\sigma)| \ge 4M$ for an immersed based loop $\sigma$ in $\Gamma$ and $k \ge 1$, then by induction \[ 2^k \, |f^{nk}(\sigma)| \le \max(\,|f^{2nk}(\sigma)|, |\sigma| \,).\]
For any bound $B$, there are only finitely many immersed based loops $\sigma'$ in $\Gamma$ with $|\sigma'| \le B$. Since $f$ is atoroidal, there is an integer $k \gg 0$ such that $|f^{nk}(\sigma')| \ge 8M $ for every based loop $\sigma'$ with $|\sigma'| \le 4M$ and we conclude that  $f$ is based-$(2, nk)$-hyperbolic .
\end{proof}

When $p$ is a subpath of an immersed loop $\sigma$ and $n \ge 1$, then $[f^n(p)]_\sigma$ is the subpath of $[f^n(p)]$ that survives in $[f^n(\sigma)]$ and $|f^n(p)|_\sigma$ is the length of $[f^n(p)]_\sigma$. Bounded cancellation implies $|f^n(p)| \le |f^n(p)|_\sigma + 2C(f^n)$.
The next lemma is based on Brinkmann's Lemma 4.2 in \cite{Bri00} with a few changes made to account for graph maps that are not $\pi_1$-surjective.

\begin{lem}\label{critlength} Let $f:\Gamma \to \Gamma$ be a graph map and $R_* \subset \Gamma$ be an $f$-invariant union of roses such that the restriction $\left.f\right|_{R_{*}}:R_* \to R_*$ is $(4, n)$-hyperbolic.
For some constant $L_c$, if $p \subset R_*$ is a subpath (edge-path) of some immersed loop $\sigma$ in $\Gamma$ and $|f^n(p)|_{\sigma} \ge L_c$, then
\[ 3\, |f^n(p)|_{\sigma} \le \max\left(\,|f^{2n}(p)|_{\sigma} , |p|\, \right).\]
\end{lem}
\noindent The number $L_c$ is the {\bf critical length} of the triple $(f, \Gamma, R_*)$.

\begin{proof} Let $f:\Gamma \to \Gamma$ be a graph map and $R_* \subset \Gamma$ be an $f$-invariant union of roses such that the restriction $\left.f\right|_{R_*}$ is $(4,n)$-hyperbolic. Set $M$ to be the maximum length of $f^k(s)$ rel. basepoint over all petals $s$ in $R_{*}$ and $k \in\{n, 2n\}$. Choose $L_c = 2C(f^{2n})+5M$ where $C(f^{2n})$ is the cancellation constant for $f^{2n}$.
Recall the triangle inequality: $\left|\,|p_1| - |p_2|\,\right| \le |p_1 \cdot p_2| \le |p_1| + |p_2|$ for any path decomposition of a loop $p_1 \cdot p_2$. A remark on the context: paths $[p_i]$ are tightened rel. endpoints but the loop $[p_1 \cdot p_2]$ is tightened by free homotopy. 

Given a subpath $p \subset R_*$ of some immersed loop $\sigma$ in $\Gamma$, pick a petal $s$ in $R_*$ such that $s \cdot p$ is an immersed loop in $R_*$.   As the restriction to $R_*$ is $(4,n)$-hyperbolic, we get \[4\,|f^n(s\cdot p)| \le \max(\, |f^{2n}(s\cdot p)|, |s\cdot p| \,).\]

If $ 4\,|f^n(s\cdot p)| \le |f^{2n}(s\cdot p)|$, then 
\[\begin{aligned} 
4\,|f^n(p)|_\sigma \le 4\,|f^n(p)| 
&\le 4\,|f^n(s\cdot p)| + 4\,|f^n(s)| \\
&\le |f^{2n}(s\cdot p)| + 4M \\
&\le |f^{2n}(p)| + 5M \\
&\le |f^{2n}(p)|_\sigma + 2C(f^{2n}) + 5M \quad \text{(bounded cancellation)}\end{aligned}\]

Similarly, if $ 4\,|f^n(s\cdot p)| \le |s\cdot p|$, then 
\[ \begin{aligned} 4\,|f^n(p)|_\sigma &\le  |p| + 1 + 4M \\
&\le |p| + 5M + 2C(f^{2n}) \quad(\text{since } M \ge 1) \end{aligned}\]
Since $L_c = 2C(f^{2n}) + 5M$, we have the desired implication:  
\[ |f^n(p)|_\sigma \ge L_c \implies 3\,|f^n(p)|_{\sigma} \le \max\left(\,|f^{2n}(p)|_{\sigma} , |p|\, \right).\qedhere\]
\end{proof}

An (outer class of an) injective endomorphism $\phi:F \to F$ is {\bf atoroidal} if it has no invariant cyclic subgroup system with index $d = 1$. If $f: \Gamma \to \Gamma$ is a topological representative for $\phi$, then $[f]$ is atoroidal if and only if $[\phi]$ is atoroidal.
The following proposition is an extension of Brinkmann's result \cite[Proposition~7.1]{Bri00} and is the main technical result of the section.
\begin{prop}\label{atoroidal} If $\phi:F \to F$ is injective and atoroidal, then $[\phi]$ has a $(2, n)$-hyperbolic topological representative for some  integer $n \ge 1$.
\end{prop}
\begin{proof} Suppose $\phi:F \to F$ is an injective and atoroidal endomorphism. If $\phi$ is surjective, then the proposition is precisely Brinkmann's result. So we assume $\phi$ is not surjective. By Theorem~\ref{expimmersion}, there is an expanding $\mathcal A^*$-relative immersion $g:T \to T$ for $\phi$, where $\mathcal A^*$ is the $[\phi]$-elliptic free factor system. Fix some $\mathcal A^*$-marked roses $(R_{\mathcal A^*}, \alpha_{\mathcal A^*})$ and set $(\Gamma, \alpha)$ to be the $(R_{\mathcal A^*}, \alpha_{\mathcal A^*})$-vertex blow-up of  the graph of groups $F\backslash T$, i.e., $(\Gamma, \alpha)$ is a marked graph with subgraphs $R_{\mathcal A^*}$ corresponding to the free factor system $\alpha(\mathcal A^*)$. The roses $R_{\mathcal A^*}$ form the {\it lower stratum} of $\Gamma$ and the remaining edges the {\it top stratum}.
\medskip

We outline the proof which follows closely Brinkmann's strategy. Patch together a homotopy equivalence of the lower stratum with the expanding relative immersion to get some topological representative $f$ of $[\phi]$. By Brinkmann's theorem, the restriction of $f$ to the lower stratum is hyperbolic. The expanding relative immersion on the top stratum means loops that are {\it mostly} top stratum will have uniform exponential growth under forward iteration. Lemma~\ref{critlength} implies loops that are {\it mostly} lower stratum will have uniform exponential growth  under forward and/or backward iteration. The heart of the proof lies in quantifying what being mostly top or lower stratum means and showing that all loops are one or the other. Of course, there are a few minor technicalities that need addressing; for instance, the restriction to the lower stratum is almost but not exactly a homotopy equivalence. This concludes the outline.
\medskip

Recall that the maximal $[\phi]$-fixed free factor system $\mathcal A$ is a subset of $\mathcal A^*$ and there is an integer $k_0 \ge 0$ such that $\phi^{k_0}(\mathcal A^*)$ is carried by $\mathcal A$ (Proposition~\ref{canonical}). So we may find a topological representative $f_{\mathcal A^*}: R_{\mathcal A^*} \to R_{\mathcal A^*}$ for $[\left.\phi\right|_{\mathcal A^*}\,]$ whose restriction to the periodic roses $R_{\mathcal A}$, denoted by $f_{\mathcal A}$, is a homotopy equivalence. As $\phi$ is atoroidal, the restriction $f_{\mathcal A}$ is $(4, n_0)$-hyperbolic for some $n_0 \ge 1$ (Brinkmann's theorem).
\medskip

If $\sigma$ is an immersed loop in $R_{\mathcal A^*}$, then $[f^{k}(\sigma)]$ is a loop in the periodic roses $R_{\mathcal A}$ for all $ k\ge k_0$. Since the restriction $f_{\mathcal A}$ is $(4, n_0)$-hyperbolic and $f_{\mathcal A^*}^{n_0 k}(\sigma)$ is a loop in $R_{\mathcal A}$ for any loop $\sigma$ in $R_{\mathcal A^*}$ and $k \ge k_0$, we get the inequality \[ \displaystyle 4 \cdot |f_{\mathcal A}^{n_0}(f_{\mathcal A^*}^{n_0 k}(\sigma))| \le \max(\, |f_{\mathcal A}^{2n_0}(f_{\mathcal A^*}^{n_0 k}(\sigma))|, |f_{\mathcal A^*}^{n_0 k}(\sigma) |\,) \text{ for all loops } \sigma \text{ in } R_{\mathcal A^*} \text{ and } k \ge k_0.\] 
Choose an integer $k_1 \ge 1$ so that $4^{k_1} \ge 4 K^{n_0 k_0}$, where $K = K(f_{\mathcal A^*})$ is the Lipschitz constant for $f_{\mathcal A^*}$. Suppose $\sigma$ is a loop in $R_{\mathcal A^*}$.

If $4 \cdot |f_{\mathcal A}^{n_0}(f_{\mathcal A^*}^{n_0 (k_0 + k_1 - 1)}(\sigma))| \le |f_{\mathcal A}^{2n_0}(f_{\mathcal A^*}^{n_0 (k_0 + k_1 - 1 )}(\sigma))|$, then by induction %2^{k_0+k_1} \cdot |f_{\mathcal A}^{n_0 k_1}(f_{\mathcal A^*}^{n_0 k_0}(\sigma))| &\le |f_{\mathcal A}^{2n_0 k_1}(f_{\mathcal A^*}^{2n_0 k_0}(\sigma))| \text{ and } \\
\[ 4^{k_0+k_1} \cdot |f_{\mathcal A^*}^{n_0(k_0+k_1)}(\sigma)| \le |f_{\mathcal A^*}^{2n_0(k_0+k_1)}(\sigma)|.\]

If $4 \cdot |f_{\mathcal A}^{n_0}(f_{\mathcal A^*}^{n_0 (k_0+k_1-1)}(\sigma))| \le |f_{\mathcal A^*}^{n_0 (k_0+k_1-1)}(\sigma)|$,  then by induction and Lipschitz property
\[ 4^{k_1} \cdot |f_{\mathcal A^*}^{n_0(k_0+k_1)}(\sigma)| \le |f_{\mathcal A^*}^{n_0 k_0}(\sigma)| \le K^{n_0 k_0} \cdot |\sigma|\] and \(4 \cdot |f_{\mathcal A^*}^{n_0(k_0+k_1)}(\sigma)| \le |\sigma| \text{ by choice of } k_1.\)
%\medskip

Therefore, the lower stratum map $f_{\mathcal A^*}$ is $(4, n_1)$-hyperbolic with $n_1 = n_0(k_0+k_1)$.
\medskip

Let $f:\Gamma \to \Gamma$ be a topological representative for $[\phi]$ that extends $f_{\mathcal A^*}$ to the top stratum and induces the expanding $\mathcal A^*$-relative immersion $g: T \to T$ %on the $(F, \mathcal A^*)$-tree $T$ 
upon collapsing the lower stratum in the universal cover $\tilde \Gamma$. For an arbitrary immersed loop $\sigma$ in $\Gamma$, define $\sigma_{top}$ ($\sigma_{low}$ resp.) to be the collection of maximal subpaths of $\sigma$ in the top (lower resp.) stratum. For all $n \ge 1$, define $[f^{n}(\sigma_{top})]_\sigma$ ($[f^{n}(\sigma_{low})]_\sigma$ resp.) to be the collection of paths $[f^{n}(p)]_\sigma$ where $p$ is some path in $\sigma_{top}$ ($\sigma_{low}$~resp.). That $f$ induces an immersion $g$ upon collapsing the lower stratum implies that the top stratum is {\it persistent}: if $\sigma$ is an immersed loop in $\Gamma$, then $f(\sigma)_{top}$ survives in $[f(\sigma)]$.

As the relative immersion $g:T \to T$ is expanding, there is an integer $k_2 \ge 1$, such that $g^{k_2}(e)$ has length $\ge 2$ for all edges $e$ in $T$; and as $f$ induces $g$, for any immersed loop $\sigma$ in $\Gamma$ and path $p$ in $\sigma_{top}$, we get $2|p|_\sigma \le |f^{k_2}(p)|_\sigma $. We may replace $n_1$ and $k_2$ with a common multiple and assume $n_1 = k_2$.
A similar inequality holds in the lower stratum. By the $(4, n_1)$-hyperbolicity of $\left.f\right|_{R_{\mathcal A^*}}$ and Lemma~\ref{critlength}, there is a critical length $L_c = L_c(f, \Gamma, R_{\mathcal A^*})$ such that for any immersed loop $\sigma$ in $\Gamma$ and path $p$ in $\sigma_{low}$, \[ |f^{n_1}(p)|_\sigma \ge L_c \implies 3 |f^{n_1}(p)|_\sigma \le \max(\,|f^{2n_1}(p)|_\sigma, |p|\,).\]

Set $M$ to be the maximal length amongst all paths in $f^{n_1}(e)_{low}$ for all top stratum edges $e$ of $\Gamma$. For any integer $k \ge 1$, we distinguish two cases:

{Case 1.} If $|f^{n_1 k}(\sigma)_{low}| \le (L_c + 6M) \, |f^{n_1 k}(\sigma)_{top}|$, then 
\[\begin{aligned}
|f^{n_1 k}(\sigma)| = |f^{n_1 k}(\sigma)_{low}| + |f^{n_1 k}(\sigma)_{top}| 
&\le (L_c + 6M + 1) \,|f^{n_1 k}(\sigma)_{top}| \text{ and } \\
2^{k} \, |f^{n_1 k }(\sigma)_{top}| \le |f^{n_1 k}(f^{n_1 k }(\sigma)_{top})|_\sigma &\le |f^{2 n_1 k}(\sigma)|.\end{aligned}\]

Additionally, if $2^{k} \ge 2 (L_c + 6M + 1)$, then $2 \, |f^{n_1 k}(\sigma)| \le |f^{2 n_1 k}(\sigma)| $. 
%\medskip
  
{Case 2.} Suppose $|f^{n_1 k}(\sigma)_{low}| \ge (L_c + 6M) \, |f^{n_1 k}(\sigma)_{top}|$. Set $m$ to be the number of paths in $f^{n_1 k}(\sigma)_{low}$. Then $m$ is also the number of paths in $f^{n_1 k}(\sigma)_{top}$ and $m \le |f^{n_1 k}(\sigma)_{top}|$. By the pigeonhole principle, some path $\rho$ in $f^{n_1k }(\sigma)_{low}$ satisfies $|\rho| \ge L_c+6M$. As $|\rho| \ge 6M$, we have $3(|\rho|-2M) \ge 2|\rho|$. Set $\sigma' = f^{n_1(k-1)}(\sigma)$. By definition of $M$ and persistence of $f^{n_1}(\sigma')_{top}$, there must be a path $p'$ in $\sigma'_{low}$ such that $[f^{n_1}(p')]_{\sigma'}$ is a subpath of $\rho$, $|f^{n_1}(p')|_{\sigma'} \ge |\rho| - 2M \ge L_c$, and \( 3 \,|f^{n_1}(p')|_{\sigma'} \le \max(\, |f^{2 n_1}(p')|_{\sigma'}, |p'| \,).\)

If $|f^{2n_1}(p')|_{\sigma'} \ge 3 |f^{n_1}(p')|_{\sigma'}$, then $|f^{2n_1}(p')|_{\sigma'} \ge 2|\rho|$ and $|f^{n_1(k+1)}(p')|_{\sigma'} \ge 3^{k-1} \cdot 2 |\rho|$.
 
If $|p'| \ge 3 |f^{n_1}(p')|_{\sigma'}$, then $|p'| \ge 2 |\rho|$. By inducting on the same argument used at the start of the case, there must be a path $p$ in $\sigma_{low}$ such that $|f^{n_1 k}(p)|_\sigma$ is a subpath of $\rho$ and $|p| \ge 2^k \, |\rho|$. In either case, we get $2^k |\rho| \le \max(\, |f^{n_1(k+1)}(p')|_{\sigma'}, |p| \,) $. Define $f^{n_1 k}(\sigma)_{crit}$ to be the set of paths $\rho$ in $f^{n_1 k}(\sigma)_{low}$ with $|\rho| \ge L_c + 6M$. Altogether, we have shown:
\[ 2^k |f^{n_1 k}(\sigma)_{crit}| \le \max(\, |f^{2n_1 k}(\sigma)|, |\sigma| \,).\]

The following computation is lifted from Brinkmann \cite[Proof of Proposition~7.1]{Bri00}. Set $A=|f^{n_1 k}(\sigma)_{crit}|$ to be the total length of paths in $f^{n_1 k}(\sigma)_{crit}$, $B = |f^{n_1 k}(\sigma)_{low}| - A $ to be the total length of the remaining paths in $f^{n_1 k}(\sigma)_{low}$, and $C = |f^{n_1 k}(\sigma)_{top}|$. We now find a positive lower bound of $\frac{A}{A+B+C}$ that is independent of $\sigma$ and $k$. We assumed $A+ B \ge (L_c + 6M)\,C$ and so $\frac{A}{A+B+C} \ge \frac{A(L_c+ 6M)}{(A+B)(L_c + 6M + 1)}$ and we can focus on the factor $\frac{A}{A+B} = 1 -\frac{B}{A+B}$.
Recall $m \le C$, so $A+B \ge (L_c + 6M)\, m$. Since each path $p$ in $f^{n_1 k}(\sigma)_{low}$ but not in $f^{n_1 k}(\sigma)_{crit}$  satisfies $|p| < L_c + 6M$ and there are at most $m$ of them, $B \le m \, (L_c + 6M - 1)$. Combining the last two inequalities gives the bound \[ 1 - \frac{B}{A+B} \ge 1 - \frac{m \, (L_c + 6M - 1)}{m \, (L_c + 6M)} \ge \frac{1}{L_c + 6M}.\]
Altogether, $\frac{A}{A+B+C} \ge \frac{1}{L_c +6M + 1}$. 

Additionally, if $2^k \ge 2(L_c + 6M + 1)$, then $2 |f^{n_1 k}(\sigma)| \le \max(\,|f^{2 n_1 k}(\sigma)|, |\sigma| \,)$.
%\medskip

Choose $k \ge 1$ so that $2^k \ge 2(L_c +6M+1)$; the two exhaustive cases above imply $f$ is $(2, n_1 k)$-hyperbolic.
\end{proof}

All the heavy lifting is done and we can prove our main theorem
\begin{thm}\label{main} Let $\phi:F \to F$ be an injective endomorphism. Then the following statements are equivalent:
\begin{enumerate}
\item $F*_\phi$ is word-hyperbolic;
\item $F*_\phi$ contains no $BS(1,d)$ subgroups with $d \ge 1$;
\item $[\phi]$ has no invariant cyclic subgroup system with index $d \ge 1$;
%\item $\phi$ has a hyperbolic topological representative and all strictly bidirectional annuli  in its mapping torus are shorter than some positive even integer.
\item all/some topological representatives of $[\phi]$ are based-hyperbolic and all strictly bidirectional annuli in their mapping tori are shorter than some integer.
\end{enumerate}
\end{thm}
\begin{proof} This proof is mostly a matter of bookkeeping.

$(1){\implies}(2)$: $BS(1,d)$ subgroups are well-known obstructions to word-hyperbolicity.

$(2){\implies}(3)$: If $[\phi]$ has an invariant cyclic subgroup system with index $d \ge 1$, then there is a subgroup of $F*_\phi$ isomorphic to a quotient of $BS(1,d)$; use normal forms to show the subgroup is in fact isomorphic to $BS(1,d)$ (See \cite[Lemma~2.3]{Kap00} for details).

$(3){\implies}(4)$: Suppose $[\phi]$ has no invariant cyclic subgroup system with index $d \ge 1$. Then, by Proposition~\ref{atoroidal}, $[\phi]$ has a $(2,n)$-hyperbolic topological representative $f:\Gamma \to \Gamma$ defined on a marked graph $(\Gamma, \alpha)$. That is, for (any conjugacy class) of $x \in F$, we have
\[2\,\|\phi^n(x)\|_\alpha \le \max(\|\phi^{2n}(x)\|_\alpha, \|x\|_\alpha). \]
For any marked graph $(\Gamma', \beta)$, choose difference of markings $g:\Gamma \to \Gamma'$ and $h:\Gamma' \to \Gamma$, i.e., graph maps such that $[\pi_1(g)\circ \alpha] = [\beta] $ and $[\pi_1(h)\circ \beta]=[\alpha]$. Then $K = \max(K(g), K(h))$ satisfies 
$ \frac{1}{K} \|x\|_{\beta} \le \| x\|_\alpha \le K\|x\|_\beta$ for all $x \in F$, where $K(g), K(h)$ are the respective Lipschitz constants. Fix an integer $m \ge 1$ such that $2^m > K^2$, then for every $x\in F$,
\[ \frac{2^m}{K} \| \phi^{mn}(x)\|_\beta \le \max(K\,\|\phi^{2mn}(x)\|_\beta, K\, \|x\|_\beta) \]
and $(\Gamma', \beta, [\phi])$ is $(2^mK^{-2}, mn)$-hyperbolic, i.e., all topological representatives of $[\phi]$ on $(\Gamma', \beta)$ are hyperbolic. Since $[\phi]$ is atoroidal, all topological representatives of $[\phi]$ on $(\Gamma', \beta)$ are based-hyperbolic (Lemma~\ref{tobased}). As $(\Gamma', \beta)$ was arbitrary,  all topological representatives of $[\phi]$ are based-hyperbolic.

Let $f':\Gamma' \to \Gamma'$ be any topological representative for $[\phi]$. Since $[\phi]$ has no invariant cyclic subgroup with index $d \ge 2$, all strictly bidirectional annuli in $M_{f'}$ are shorter than some integer (Proposition~\ref{cycloinv} with Lemmas~\ref{biannulibound} and \ref{algandtop}).

$(4){\implies}(1)$ --- Theorem~\ref{blueprint}. See Theorem~\ref{blueprint2} below for a sketch of the proof.
\end{proof}

%%%%%%%%%%%%%%%%%%%%%%%%%%%%%%%%%%%%%%%%%%%%%%%%%%%%%%%%%%
\subsection{HNN extensions over free factors}\label{secHNNExts}
%%%%%%%%%%%%%%%%%%%%%%%%%%%%%%%%%%%%%%%%%%%%%%%%%%%%%%%%%%
In the last section, we extend the previous characterization to HNN extensions of free groups defined over free factors. Precisely, let $A \le F$ be a (nontrivial) free factor and $\phi:A \to F$ be an injective homomorphism. The HNN extension of $F$ over $A$ is given by:
\[ F*_A = \langle F, t~|~ t^{-1}x t = \phi(x), \forall x \in A \rangle. \]
HNN extensions, just like mapping tori, have associated Bass-Serre trees which we can use to define algebraic annuli as in the previous section. Unlike the Bass-Serre tree of a mapping torus, that of an HNN extension need not behave like a rooted tree: generally, the number of outgoing and incoming edges at a vertex will match the indices of $A$ and $\phi(A)$ in $F$ respectively. However, since free factors are malnormal, these HNN extensions share with mapping tori the dichotomy of annuli: annuli are either unidirectional or bidirectional. %In particular, long annuli will have long unidirectional {\it subannuli}.

We will say a conjugacy class of elements $[g]$ in $F$ has {\bf one forward iterate} if it is contains an element $a \in A$.  By malnormality of $A$, the conjugacy class of $[\phi(a)]$ depends only on the class $[g]$. We set the forward iterate of such a class $[g]$ to be $\phi\cdot[g] = [\phi(a)]$. Inductively, $[g]$ has $n+1$ forward iterates if it has $n$ forward iterates and the iterate $\phi^n \cdot [g]$ has one forward iterate; in this case, set $\phi^{n+1}\cdot [g] = \phi \cdot (\phi^n \cdot [g])$. %Conjugacy classes in $F$ with $n$ forward iterates are the conjugacy classes of unidirectional annuli in $F*_A$ of length $n$.

Similar to our use of infinite tails (backward iteration) to construct a canonical fixed free factor system in Section~\ref{secElliptic}, we shall use forward iteration %long unidirectional annuli 
to construct a canonical invariant free factor system of $A$. The proof will use a simpler variation of descent.

%\begin{prop}\label{maxinvariant} If $A \le F$ is a free factor and $\phi:A \to F$ is an injective homomorphism used to define $F*_A$, then there is a unique $[\phi]$-invariant free factor system of $A$ that carries all unidirectional annuli longer than some constant $L(A)$.\end{prop}
\begin{prop}\label{maxinvariant} If $A \le F$ is a free factor and $\phi:A \to F$ is injective, then there is a unique maximal $[\phi]$-invariant free factor system of $A$. Precisely, $\mathcal F$ carries every conjugacy class with $2\cdot \mathrm{rank}(A)$ forward iterates and every $[\phi]$-invariant subgroup system.
\end{prop}
\begin{proof}Let $A \le F$ be a free factor, $\phi:A \to F$ be injective, and $L(A)-1 = 2\cdot \mathrm{rank}(A)-1$ be the length of the longest chain of free factor systems in $A$. Any conjugacy class in a $[\phi]$-invariant free factor system of $A$ has infinitely many forward iterates. So if only the trivial conjugacy class has $L(A)$ forward iterates, then the trivial system is the unique $[\phi]$-invariant free factor system. Thus, we may assume some fixed nontrivial conjugacy class $[g]$ has $L(A)$ forward iterates. Since $A \le F$ is a free factor and $\phi:A \to F$ is injective, we can iterate $\phi^{-1}$ on the poset of free factor systems of $A$!

%Let us recall the construction of free factor system preimages used in the proof of Proposition~\ref{canonical}. Suppose $\mathcal F$ is a free factor system of $F$ and $\psi:H \to F$ is some injective homomorphism. Take an(y) $(F, \mathcal F)$-tree $T$ and look at the minimal tree $T(\psi(H)) \subset T$ for $\psi(H)$. The quotient graph of groups $\phi(H)\backslash T(\phi(H))$ is a free splitting of $\phi(H)$ and, by injectivity of $\phi$, we get a free factor system of $H$ that we denote by $\phi^{-1} \cdot \mathcal F$.

\begin{claim}[Descent] For some $n \ge 0$, let $\mathcal A_n \prec \cdots \prec \mathcal A_0$ be the chain of nontrivial free factor systems of $A$ such that $\mathcal A_i = \phi^{-i} \cdot \{ A \}$ and $\mathcal A_n$ carries all $[\phi]$-invariant subgroup systems of $A$. If $\mathcal A_n$ is not $[\phi]$-invariant, then $\mathcal A_{n+1} = \phi^{-1} \cdot \mathcal A_n \prec \mathcal A_n$ is a nontrivial free factor system of $A$ that carries all $[\phi]$-invariant subgroup systems. \end{claim}
Starting with $\{ A \}$, the descent will eventually stop at some $n < L(A)-1$ with a nontrivial free factor system $\mathcal A_n = \phi^{-n}\cdot \{ A \}$ that carries all conjugacy classes with $n+1$ forward iterates and all $[\phi]$-invariant subgroup systems of $A$. Since descent stopped, $\mathcal A_n$ must be $[\phi]$-invariant and hence the unique maximal system amongst all $[\phi]$-invariant free factor systems of $A$. %By nontriviality of $\mathcal A_n$, we have $n < L(A)$. So $\mathcal A_n$ carries all conjugacy classes with $L(A)$ forward iterates. It remains to prove the descent claim.
\end{proof}

\begin{proof}[Proof of descent] Let $n\ge 0$ and $\mathcal A_n \prec \cdots \prec \mathcal \mathcal A_0 = \{ A \}$ be the chain of nontrivial free factor systems such that $\mathcal A_i = \phi^{-i} \cdot \{ A \}$ and $\mathcal A_n$ carries all $[\phi]$-invariant subgroup systems of $A$. Nontriviality of $\mathcal A_n$ implies $n<L(A)-1$. Now suppose that $\mathcal A_n$ is not $[\phi]$-invariant. %For simplicity, assume the free factors of $\mathcal A_n$ are subgroups of the free factors of $\mathcal A_{n-1}$.

The existence of a nontrivial conjugacy class $[g]$ with $L(A) \ge n+2$ forward iterates implies $\mathcal A_{n+1} = \phi^{-1} \cdot \mathcal A_n$ is a nontrivial free factor system of $A$. If $n = 0$, then clearly $\mathcal A_{n+1} \preceq \mathcal A_n$. If $n > 1$, then $\mathcal A_n \prec \mathcal A_{n-1}$ by assumption; hence, $\phi^{-1} \cdot \mathcal A_n \preceq \phi^{-1} \cdot \mathcal A_{n-1}$ and, equivalently, $\mathcal A_{n+1} \preceq \mathcal A_n$. If $\mathcal A_n = \mathcal A_{n+1}$, then $\phi(\mathcal A_n)$ is carried by $\mathcal A_n$. But $\mathcal A_n$ is not $[\phi]$-invariant, thus $\mathcal A_{n+1} \prec \mathcal A_n$. 
Let $\mathcal B$ be any $[\phi]$-invariant subgroup system in $A$. By assumption, $\mathcal A_n$ carries  $\mathcal B$ and $\phi(\mathcal B)$; therefore, $\mathcal A_{n+1} = \phi^{-1}\cdot \mathcal A_n$ carries $\mathcal B$.
\end{proof}

The unique $[\phi]$-invariant free factor system of $A$ given by Proposition~\ref{maxinvariant}, denoted by $\mathcal F$, is the {\bf canonical $[\phi]$-invariant free factor system} and it captures the long-term dynamics of $[\phi]$. The system $\mathcal F$ is proper (in $A$) if and only if $A$ is not $[\phi]$-invariant; it is trivial exactly when only the trivial conjugacy class has $L(A)=2\cdot \mathrm{rank}(A)$ forward iterates. The canonical invariant free factor system allows us to naturally extend definitions and results that required iteration of an injective endomorphism. For instance, we can now say $\phi:A \to F$ has {\bf an invariant cyclic subgroup system with index $\boldsymbol d \ge 1$} if a restriction $\left.\phi\right|_{\mathcal F}:\mathcal F \to \mathcal F$ has an invariant cyclic subgroup system with index $d \ge 1$. In particular, $[\phi]$ is {\bf atoroidal} if $\left[ \left.\phi\right|_{\mathcal F}\right]$ is atoroidal. 

For $k \ge L(A)$, we define the {\bf iterated pullbacks $\boldsymbol{\Lambda_k[\phi]}$ of $\boldsymbol{[\phi]}$} to be the iterated pullbacks $\Lambda_k[\left.\phi\right|_{\mathcal F}]$ of the restriction $[\left.\phi\right|_{\mathcal F}]$. Similarly define $\hat \Lambda_k$ for $k \ge L(A)$. 
Since connectedness did not play any role in the proof of Proposition~\ref{cycloinv}, we immediately get the following extension when we replace endomorphisms of $F$ with endomorphisms of $\mathcal F$.

\begin{prop}\label{cycloinv2} If $A \le F$ is a free factor, $\phi:A \to F$ is injective, and $\hat \Lambda_k[\phi]$ is not empty for all $k \ge L(A)$, then $\phi$ has an invariant cyclic subgroup system with index $d \ge 2$.
\end{prop}
\begin{rmk} By the same token, the results in this section do not need $A$ to be ``connected.'' So there is a natural generalization of these results that easily follows if we replace the free factor $A \le F$ with a free factor system $\mathcal A$ of $F$.
\end{rmk} 

For the rest of the section, we will extend the results of Section~\ref{secHypEnds} to the injective homomorphism $\phi:A \to F$. Fix a marked graph $(\Gamma, \alpha)$ such that the free factor systems $\mathcal F$ and $A$ correspond to nested core subgraphs $\Gamma_{\mathcal F} \subset \Gamma_A \subset \Gamma$ respectively, i.e., $\Gamma_{\mathcal F}$ and $\Gamma_A$ have markings $\alpha_{\mathcal F}:\mathcal F \to \pi_1(\Gamma_{\mathcal F})$ and $\alpha_A:A\to \pi_1(\Gamma_A)$ respectively such that the inclusion maps $c_{\mathcal F}:\Gamma_{\mathcal F} \to \Gamma_A$ and $c_A: \Gamma_A \to \Gamma$ satisfy $[\pi_1(c_{\mathcal F}) \circ \alpha_{\mathcal F}] = [\left.\alpha_A\right|_{\mathcal F}]$ and $[\pi_1(c_A) \circ \alpha_A] = [\left.\alpha\right|_A]$. Assume $\mathcal F$ is not trivial so that $\Gamma_{\mathcal F}$ is not degenerate (finite set of points). A {\bf topological representative for $\boldsymbol{[\phi]}$} will be a graph map $f: (\Gamma_A, \Gamma_{\mathcal F}) \to (\Gamma, \Gamma_{\mathcal F})$ with no pretrivial edges such that $[\pi_1(f)\circ \alpha_A] = [\alpha \circ \phi]$. Thus, the {\bf invariant restriction} $\left.f\right|_{\mathcal F}:\Gamma_{\mathcal F} \to \Gamma_{\mathcal F}$ is a topological representative for the restriction $[\left.\phi\right|_{\mathcal F}]$.

The following is the analogue of the mapping torus. Let $f$ be a topological representative for $\phi:A \to F$. We define the {\bf classifying space} to be $M_f = \left( \Gamma \sqcup (\Gamma_A \times [0,1])\right)/\sim$ where we identify $(x,0)\sim x$ and $(x,1)\sim f(x)$ for all $x \in \Gamma_A \subset \Gamma$. The edge-space of $M_f$ will be the cross-section represented by $\Gamma_A \times \{ \frac{1}{2} \}$. By construction, $\pi_1(M_f) \cong F*_A$. Topological annuli in $M_f$ and algebraic annuli in $F*_A$ are defined exactly as before and the correspondence between them is the same. Similarly, the correspondence between strictly bidirectional annuli and the iterated pullbacks remains. We then get this natural extension of Lemma~\ref{biannulibound}.

\begin{lem}\label{biannulibound2} Let $A \le F$ be a free factor and $\phi:A \to F$ be injective. For any integer $L \ge L(A)$, the HNN extension $F*_A$ has a strictly bidirectional annulus of length $2L$ if and only if $\hat \Lambda_L[\phi]$ is not empty.
\end{lem}

The definition of (based-)\,hyperbolicity extends naturally using the invariant restriction. For $\lambda>1$ and $n \ge L(A)$, a topological representative $f:(\Gamma_A, \Gamma_{\mathcal F}) \to (\Gamma, \Gamma_{\mathcal F})$ for $[\phi]$ is {\bf (based-)\,$\boldsymbol{(\lambda, n)}$-hyperbolic} if $\left.f\right|_{\mathcal F}:\Gamma_{\mathcal F} \to \Gamma_{\mathcal F}$ is (based-)\,$ (\lambda, n)$-hyperbolic and $[f]$ is {\bf atoroidal} if $[\left.f\right|_{\mathcal F}]$ is atoroidal. When $\mathcal F$ is trivial, any topological representative of $\phi$ is vacuously (based-)\,hyperbolic. Again, since connectedness played no role in the proofs of Lemma~\ref{critlength} and Proposition~\ref{atoroidal}, we get this extension:

\begin{prop}\label{atoroidal2} Let $A \le F$ be a free factor. If $\phi:A \to F$ is injective and atoroidal, then $[\phi]$ has a $(2,n)$-hyperbolic topological representative for some integer $n \ge L(A)$.
\end{prop}

We are now ready to state and almost prove the main result of the section.

\begin{thm}\label{main2} Let $A \le F$ be a free factor and $\phi:A \to F$ be an injective homomorphism. Then the following are equivalent:
\begin{enumerate}
\item $F*_A$ is word-hyperbolic;
\item $F*_A$ contains no $BS(1,d)$ subgroups with $d \ge 1$;
\item $[\phi]$ has no invariant cyclic subgroup system with index $d \ge 1$;
\item all/some topological representative of $[\phi]$ are based-hyperbolic and all strictly bidirectional annuli in their classifying spaces are shorter than some integer.
\end{enumerate}
\end{thm}

The proof follows the same steps as Theorem~\ref{main} and the only missing ingredient is an extension of Theorem~\ref{blueprint} (Theorem~\ref{blueprint2} below) that proves the implication $(4) \implies (1)$. 

\begin{rmk} Every HNN extension of the form $F*_A$ where $F$ is a finite rank free group and $A \le F$ is a free factor can be written as a mapping torus $\mathbb F*_\psi$ of an injective endomorphism $\psi:\mathbb F \to \mathbb F$ where $\mathbb F$ is a free group with possibly infinite rank. Conversely, every finitely generated subgroup of a mapping torus $\mathbb F*_\psi$ where $\mathbb F$ has possibly infinite rank can be written as an HNN extension $F*_A$ where $F$ has finite rank and $A\le F$ is a free factor \cite{FH99}. This argument is Feighn-Handel's result that $F*_A$ and $\mathbb F*_\psi$ are {\it coherent}. As a corollary, we get the following amusing statement:
\end{rmk}

\begin{cor}Let $\psi:\mathbb F \to \mathbb F$ be an injective endomorphism of an infinite rank free group. $\mathbb F*_\psi$ is locally word-hyperbolic if and only if it contains no $BS(1,d)$ subgroups with $d \ge 1$. 
\end{cor}

\noindent By {\it locally word-hyperbolic}, we mean every finitely generated subgroup is word-hyperbolic.
\medskip

The following theorem will complete our proof of Theorem~\ref{main2}.

\begin{thm}\label{blueprint2} Let $A \le F$ be a free factor and $\phi:A \to F$ be injective. If a topological representative $f$ of $[\phi]$ is based-hyperbolic and all strictly bidirectional annuli in its classifying space $M_f$ are shorter than some integer, then $F*_A$ is word-hyperbolic.
\end{thm}

Roughly speaking, the extension follows from the fact that annuli of $M_f$ longer than $L(A)$ are annuli of the mapping torus of $\left.f\right|_{\mathcal F}$ up to some controlled error. Due to the close similarity with the proof of Theorem~\ref{blueprint} given in \cite{JPMa}, we will only sketch a proof of the extension. 

We start by defining {\it annuli flaring}; let $h:S^1 \times [-M, M] \to M_f$ be a topological annulus in $M_f$ of length $2M \ge 2L(A)$. The {\bf girth of $\boldsymbol h$} is the combinatorial length $|h_0|$ of the middle ring $h_0 = h(\cdot,0):S^1 \to M_f$ in the edge-space. Let $\lambda > 1$ be a real number. The annulus $h$ is {\bf $\boldsymbol \lambda$-hyperbolic} if 
\[ \lambda|h_0| \le \max(|h_{-L}|, |h_{L}|). \]
For integers $i$ between $-M$ and $M-1$ (inclusive), define $\tau_i:(i, i+1) \to \Gamma$ by projecting the trace $\left.h^{\nu}\right|_{(i, i+1)}$ to the vertex-space $\Gamma$ as follows:
\[
\tau_i(t) = 
\begin{cases}
x &\text{if } h^{\nu}(t) = (x,s) \in \Gamma_A \times [0,1] \text{ and } s < \frac{1}{2} \\
f(x) &\text{if } h^{\nu}(t) = (x,s) \in \Gamma_A \times [0,1] \text{ and } s > \frac{1}{2} \\
x &\text{if } h^{\nu}(t) = x \in \Gamma\\
\end{cases}
\]
We say $h$ is {\bf $\boldsymbol \rho$-thin} if $|\tau_i|+1 \le \rho$ for all $i$, where $|\cdot|$ is the length after tightening the path rel. endpoints. We say $M_f$ satisfies the {\bf annuli flaring condition} if there are $\lambda>1, M \ge 2L(A)$, and a function $H:\mathbb R \to \mathbb R$ such that any $\rho$-thin annulus of length $2M$ with girth at least $H(\rho)$ is $\lambda$-hyperbolic. Bestvina-Feighn's combination theorem \cite{BF92} states that $\pi_1(M_f)$ is word-hyperbolic if $M_f$ satisfies the annuli flaring condition.

\begin{proof}[Sketch proof of Theorem~{\ref{blueprint2}}] Suppose $A \le F$ is a free factor, $\phi:A \to F$ is injective, and $\mathcal F$ is the canonical $[\phi]$-invariant free factor system of $A$. Let $f:(\Gamma_A, \Gamma_{\mathcal F}) \to (\Gamma, \Gamma_{\mathcal F})$ be a based-hyperbolic topological representative for $[\phi]$ and $K = K(f) > 1$ be some Lipschitz constant for $f$. 
%If $A$ is $[\phi]$-invariant, then all annuli are contained in the mapping torus of $f:\Gamma_A \to \Gamma_A$ and the result follows from Theorem~\ref{blueprint}. 
If $\mathcal F$ is trivial, then all unidirectional annuli in $M_f$ are shorter than $L(A)$. So $M_f$ vacuously satisfies the annuli flaring condition for some $M=2L(A)$ by the annulus dichotomy (malnormality of $A$). Thus, we may assume $\mathcal F$ is not trivial. The assumptions and tool introduced in the next two paragraphs are needed in case $A$ is not $[\phi]$-invariant.
\medskip

%If $A$ is not $[\phi]$-invariant, we can assume $f:(\Gamma_A, \Gamma_{\mathcal F}) \to (\Gamma, \Gamma_{\mathcal F})$ has the additional structure. 
Let $\mathcal A_n \prec \cdots \prec \mathcal A_0 \prec \mathcal A_{-1}$ be the chain where $\mathcal A_{-1} = \{ F \}, \mathcal A_0 = \{ A \}, \mathcal A_n = \mathcal F,$ and generally $\mathcal A_i = \phi^{-i}\cdot \{ A \}$ for $0 \le i \le n$. By invariance of $\mathcal A_n$, we can extend the chain by setting $\mathcal A_i = \mathcal A_n$ for $i > n$. For $0 \le i < n$, $\mathcal A_{i-1}$ carries $\phi(\mathcal A_i)$ by definition. We assume $\Gamma$ has a filtration of nonempty core subgraphs $\cdots \subset \Gamma_1 \subset \Gamma_0 \subset \Gamma_{-1}$ where $\Gamma_{-1} = \Gamma_{\mathcal F} \vee R$ is a wedge between $\Gamma_{\mathcal F}$ and a rose $R$ and each $\Gamma_i$ is a $\mathcal A_i$-marked graph. If $n >0$, we will also assume the topological representative $f:(\Gamma_0, \Gamma_n) \to (\Gamma_{-1}, \Gamma_n)$ is of the form $(\Gamma_0,\ldots,\Gamma_{n-1}, \Gamma_n) \to (\Gamma_{-1}, \ldots, \Gamma_{n-2}, \Gamma_n)$. 
%Implicit in this assumption is the fact that all topological representative are based-hyperbolic when one of them is and $\phi$ is atoroidal. 
Note that since $\Gamma_n = \Gamma_{\mathcal F}$ was $f$-invariant to begin with, changes can made to the given representative $f$ to allow for these additional assumptions without affecting the invariant restriction $\left.f\right|_{\mathcal F}$. So $f$ is still based-hyperbolic after the changes. %Also recall that the bound on strictly bidirectional annuli depends only on the (outer class of the) injective homomorphism $\phi$.

Using this filtration, we introduce a tool needed to iterate based loops not in $\Gamma_{\mathcal F}$: for any integer $i \ge 0$ and any immersed based loop $s:S^1 \to \Gamma_0$ that is freely homotopic into $\Gamma_{\mathcal F}$, there is a ``closest point'' projection (with respect to Hausdorff distance of lifts to the universal cover) of $s$ to an immersed based loop in $\Gamma_0$ denoted by $\lfloor s \rfloor_i:S^1 \to \Gamma_0$ freely homotopic to $s$ such that $f^i(\lfloor s \rfloor_i)$ is a based loop in $\Gamma_0$. Recall that based loops must send basepoints to vertices. Since we assumed the ``complement'' of $\Gamma_n$ are roses, the projections $\lfloor s \rfloor_i$ are immersed based loops in $\Gamma_i$. So $\lfloor s \rfloor_i = \lfloor s \rfloor_n$ for $i \ge n$. Implicit in this is the fact that, for all $i \ge 0$, any conjugacy class carried by $\mathcal A_i$ but not $\mathcal A_{i+1}$ is mapped by $\phi$ to a conjugacy class carried by $\mathcal A_{i-1}$ but not $\mathcal A_i$. We are ready to start proving the theorem.
\medskip

 Suppose $f$ is based-$(2,m)$-hyperbolic for some $m \ge 1$ and strictly bidirectional annuli in $M_f$ are shorter than $2L \ge 2L(A)$. Fix $r \ge 1$ such that $2^{L+r} \ge 3\cdot K^{Lm}$. Now set $M = (2L+r)m$ and 
\[H(\rho) = 2^M \cdot 4\rho \cdot \frac{K^{2M} - 1}{K-1}.\]
Suppose $h$ is an arbitrary $\rho$-thin annulus of length $2M$ with girth $|h_0| \ge H(\rho)$. We need to show that $h$ is $2$-hyperbolic. 
Without loss of generality, assume the truncation $\left.h\right|_{[-M, 0]}$ is unidirectional and {\it increasing}. 

For $-M \le k \le 0$, consider the rings of $h_k:S^1 \to M_f$ in the edge-space $\Gamma_A \times \{\frac{1}{2}\}$. This can be viewed as an immersed based loop $h_k:S^1 \to \Gamma_0$ by forgetting the $\{ \frac{1}{2} \}$-factor.
Since $M \ge L(A)$ and $h_{-M} \simeq_{M_f} h_0$ (free homotopy in $M_f$) is not trivial, the ring $h_{-M}$ is freely homotopic (in $\Gamma$) into $\Gamma_n = \Gamma_{\mathcal F}$. Since $\Gamma_n$ is $f$-invariant and $h_{-M} \simeq_{M_f} h_k$ for all $k \le 0$, the rings $h_k$ are freely homotopic into $\Gamma_n$. In particular, these rings all have projections to subgraphs $\Gamma_i$ for all $i$.

For the first step, project $h_{-1}$ to $\Gamma_1$ so that the image $f(\lfloor h_{-1}\rfloor_1):S^1 \to \Gamma_0$ is a based loop freely homotopic to $h_0$ in $\Gamma$. The $\rho$-thin assumption (and ``closest point'' projection) means: $|f(\lfloor h_{-1}\rfloor_1)| \ge |h_0| - 2\rho$.
Similarly, we project $h_{-2}$ to $\Gamma_2$ so that $f(\lfloor h_{-2} \rfloor_2):S^1 \to \Gamma_1$ is a based loop freely homotopic to $\lfloor h_1 \rfloor_1$ in $\Gamma$. Then $|f(\lfloor h_{-2}\rfloor_2)| \ge |\lfloor h_1\rfloor_1| - 2\rho$ by $\rho$-thinness and, by the $K$-Lipschitz property of $f$, \[|f^2(\lfloor h_{-2}\rfloor_2)| \ge |f(\lfloor h_1\rfloor_1)| - 2\rho\cdot K \ge |h_0| - 2\rho \cdot (1 + K).\]
By induction, the $f$-invariance of $\Gamma_n$, and the fact $M > L(A) > n$, we get this inequality:
\[ |f^M(\lfloor h_{-M}\rfloor_n)| \ge |h_0| - 2\rho \cdot (1 +  \cdots + K^{M-1}) = |h_0| - 2\rho \cdot \frac{K^M - 1}{K-1}. \]
Since $f$ is based-$(2,m)$-hyperbolic, we know that \[2|f^M(\lfloor h_{-M}\rfloor_n)| \le \max(|f^{M+m}(\lfloor h_{-M}\rfloor_n)|, |f^{M-m}(\lfloor h_{-M}\rfloor_n)|).\]
There are three cases to consider of which we will only prove one.

{\it \noindent Case 1:} Suppose $2|f^M(\lfloor h_{-M}\rfloor_n)| \le |f^{M-m}(\lfloor h_{-M}\rfloor_n)|$. 

Then by induction on based-$(2,m)$-hyperbolicity, we get:
\[\begin{aligned}
|h_{-M}| \ge |\lfloor h_{-M}\rfloor_n| 
&\ge 2^{2L+r}|f^M(\lfloor h_{-M}\rfloor_n)| \\
&\ge 2^{2L+r}|h_0| - 2^{2L+r}\cdot2\rho \cdot \frac{K^M - 1}{K-1} \\
&\ge 3|h_0| - H(\rho) \\
&\ge 2|h_0| &\text{as } |h_0| \ge H(\rho).
\end{aligned}\]
So $h$ is $2$-hyperbolic in this case. The proof so far has been nearly identical to the proof of Theorem~\ref{blueprint} in \cite{JPMa}. But here, we need to be careful when applying iterates of $f$ to the rings of $h$ since $A$ may not be $[\phi]$-invariant. This is why we introduced the projections $\lfloor \cdot \rfloor_i$.

There are two more cases remaining.

{\it \noindent Case 2:} Suppose $2|f^M(\lfloor h_{-M}\rfloor_n)| \le |f^{M+m}(\lfloor h_{-M}\rfloor_n)|$ and $h$ is unidirectional. 

{\it \noindent Case 3:} Suppose $2|f^M(\lfloor h_{-M}\rfloor_n)| \le |f^{M+m}(\lfloor h_{-M}\rfloor_n)|$ and $h$ is bidirectional.

Case~3 is where the bound on strictly bidirectional annuli is needed. We leave the details of these cases to the reader. Alternatively, one could read the proof of Theorem~\ref{blueprint}, compare how Case~1 was handled in the two proofs, and adjust the proofs of the remaining cases accordingly.
We have covered all the cases and $M_f$ satisfies the annuli flaring condition. By the combination theorem, $F*_A \cong \pi_1(M_f)$ is word-hyperbolic.
\end{proof}

%%%%%%%%%%%%%%%%%%%%%%%%%%%%%%%%%%%%%%%%%%%%%%%%%%%%%%%%%%
\nonumsec{Epilogue}
%%%%%%%%%%%%%%%%%%%%%%%%%%%%%%%%%%%%%%%%%%%%%%%%%%%%%%%%%%
We would like to conclude this paper with a discussion of a few questions: the first question is a natural generalization of Section~\ref{secHNNExts}; the rest are problems from Ilya Kapovich's paper \cite[Section 6]{Kap00} that could be answered using expanding relative immersions.

\begin{prob} Suppose $A\le F$ is a vertex group of a cyclic splitting of $F$ and $\phi:A \to F$ is injective.
Is $F*_A$ word-hyperbolic if it contains no $BS(1,d)$ subgroups for $d \ge 1$? Can it have $BS(m,d)$ subgroups with $m, d > 1$?
\end{prob}
A {\it cyclic splitting} of $F$ is an edge of groups decomposition of $F$ with a nontrivial cyclic edge group. Vertex groups of cyclic splittings are generalizations of free factors and it would be interesting to see if the ideas in Section~\ref{secHNNExts} can be adapted to this case. The first obstacle is finding the appropriate generalization of Proposition~\ref{maxinvariant} since $\phi^{-1}$-iteration need not be as well-behaved. Furthermore, the vertex groups $A$ are not always malnormal which means annuli in $F*_A$ could exhibit more complicated behavior.
\medskip

For the remaining problems, assume $\phi:F \to F$ is injective. 

\begin{prob}[{\cite[Problem~6.4]{Kap00}}]\label{isoProb} What kind of isoperimetric functions can $F*_\phi$ have?\end{prob}

\noindent The automorphism case of this problem was answered by Bridson-Groves \cite{BG10}: they showed that free-by-cyclic groups have quadratic isoperimetric functions. 
Implicit in the second part of the paper is the idea that $F*_\phi$ is hyperbolic relative to a canonical finite collection of free-by-cyclic groups when it has no $BS(1,d)$ subgroups for $d \ge 2$. Furthermore, when $A \le F$ is a free factor and $\psi:A \to F$ is injective, then $F*_A$ is hyperbolic  relative to a canonical finite collection of ascending HNN extensions. This would imply that $F*_A$ has a quadratic isoperimetric function when it has no $BS(1,d)$ subgroups for $d \ge 2$. We intend to complete this direction in future work by employing a combination theorem for relatively hyperbolic groups \cite{Gau16, MR08}.
\medskip

When $F*_\phi$ is word-hyperbolic, Mahan Mj proved there is a continuous extension to the Gromov boundary of the inclusion map $F \le F*_\phi$ \cite{Mj98}; this map is known as the {\it Cannon-Thurston map}.

\begin{prob}[{\cite[Problem~6.7]{Kap00}}] If $F*_\phi$ is word-hyperbolic, is the Cannon-Thurston map (uniformly) finite-to-one? Is there a corresponding ending laminations theorem?
\end{prob}

For the first of these questions, it seems that expanding relative immersions reduce the problem to the automorphism case, which has been answered \cite{Gho20, Mj97}. As for the ending laminations, what is missing is the appropriate formulation of the theorem that replaces the usual short exact sequence with a graph of groups decomposition.

\begin{prob}[{\cite[Problem~6.8]{Kap00}}] Let $\partial \phi:\partial F \to \partial F$ be the extension of $\phi$ to the boundary. Can we classify injective endomorphisms $\varphi$ by the dynamics of boundary extensions $\partial \varphi$?\end{prob}

Levitt-Lustig covered the surjective case when they proved that {\it most} automorphisms of non-elementary word-hyperbolic groups have boundary extensions with {\it north-south dynamics} \cite{LL00}: two fixed points --- a repellor and an attractor. Expanding relative immersions should imply that {\it most} injective nonsurjective endomorphisms of $F$ have {\it sink dynamics}: a single fixed point and it is an attractor.
\medskip

There are of course more unresolved questions about injective endomorphisms and their expanding relative immersions. These are left for the reader to ask and answer.

%%%%%%%%%%%%%%%%%%%%%%%%%%%%%%%%%%%%%%%%%%%%%%%%%%%%%%%%%%
\appendix
%%%%%%%%%%%%%%%%%%%%%%%%%%%%%%%%%%%%%%%%%%%%%%%%%%%%%%%%%%
%%%%%%%%%%%%%%%%%%%%%%%%%%%%%%%%%%%%%%%%%%%%%%%%%%%%%%%%%%
\section{Relative train tracks}\label{relalgo}
%%%%%%%%%%%%%%%%%%%%%%%%%%%%%%%%%%%%%%%%%%%%%%%%%%%%%%%%%%
%{\it Bestvina-Handel's paper \cite{BH92} is the reference for the material in this appendix.}
%\medskip

The objective in this appendix is to sketch the proof that irreducible relative representatives with minimal stretch factor are train tracks. Bestvina-Handel's construction of train tracks for irreducible automorphisms \cite{BH92} translates verbatim to the non-free forest setting.

Let $\phi:F \to F$ be an injective endomorphism, $\mathcal A \prec \mathcal B$ be $[\phi]$-invariant free factor systems, and $T_{*}$ be a $(\mathcal B, \mathcal A)$-forest. We allow forests to have bivalent vertices.
Recall, an {\it $\mathcal A$-relative weak representative} for the restriction $\left.\phi\right|_{\mathcal B}$ is a $\left.\phi\right|_{\mathcal B}$-equivariant graph map $f_{*}: T_{*} \to T_{*}$. An {\it $\mathcal A$-relative representative} is an $\mathcal A$-relative weak representative $f_{*}$ with no pretrivial edges and whose underlying forest $T_{*}$ has no bivalent vertices.
Additionally, we say the relative representative is {\it minimal} if it has no orbit-closed invariant subforests with bounded components.

For any $\mathcal A$-relative weak representative $f_{*}$, we get the {\it transition matrix} $A(f_{*})$. %Let $A(f_{*})$ be a square matrix whose rows and columns are indexed by the number of orbits of natural edges in $T_{*}$. If $a_{ij}$ is the entry of $A(f_{*})$ in row-$i$ and column-$j$, then $a_{ij}$ is the number of components the orbit $B_i \cdot e_i$ are contained in the the edge-path $f_{*}(e_j)$, where $e_i$ is a orbit representative for the $i$-th orbit of natural edges. By local injectivity in the interior of edges, $A(f_{*})$ is a nonnegative matrix with at least one positive entry in each column. 
An $\mathcal A$-relative representative $f_{*}$ is {\it irreducible} if the matrix $A(f_{*})$ is irreducible and, in this case, the {\bf stretch factor} of $f_{*}$, denoted by $\lambda(f_{*})$, is the Perron-Frobenius eigenvalue of $A(f_{*})$. An {\bf $\boldsymbol{\mathcal A}$-relative train track} for $\left.\phi\right|_{\mathcal B}$ is an $\mathcal A$-relative representative $f_*$ for $\left.\phi\right|_{\mathcal B}$ that additionally satisfies the property: the edge-paths $f_{*}^n(e)$ are immersed for all edges $e$ in $T_{*}$ and integers $n \ge 1$. We have set the stage for the Bestvina-Handel's result.

\begin{athm}[{\cite[Theorem~1.7]{BH92}}]\label{appreltt} Let $\left.\phi\right|_{\mathcal B}$ be irreducible relative to $\mathcal A$ and $f_{*}:T_{*} \to T_{*}$ be an irreducible $\mathcal A$-relative representative for $\left.\phi\right|_{\mathcal B}$. If $f_{*}$ has minimal stretch factor, then it is an irreducible $\mathcal A$-relative train track.\end{athm}

Minimality is understood to be amongst irreducible representatives rather than weak representatives.
%We now state the lemmas that will be fundamental steps in the proof.
%Following Bestvina-Handel's lead, the proof shall use the following lemma to ensure the moves induce irreducible $\mathcal A$-relative representatives.
The argument relies on understanding how various moves on an irreducible $\mathcal A$-relative representative $f_{*}$ affect $\lambda(f_{*})$ and invoking minimality of $\lambda(f_*)$ to conclude that no such moves are possible. Note that although the moves are described locally, they must be performed equivariantly if we want the resultant forests to be $(\mathcal B, \mathcal A)$-forests. 
The proofs of these moves/lemmas are omitted since they are the same as the proofs in \cite{BH92}.

\begin{rmk} Recently, Bestvina \cite{Bes11} and Francaviglia-Martino \cite{FM15} gave an alternative approach to proving this theorem using the {\it Lipschitz metric} on {\it relative outer space}.
\end{rmk}

The first move is {\it subdivision}, which occurs at an interior point of an edge that is in the preimage of vertices under the representative.

\begin{alem}[{\cite[Lemma~1.10]{BH92}}]If $f_{*}: T_{*} \to T_{*}$ is an irreducible $\mathcal A$-relative weak representative for $\left.\phi\right|_{\mathcal B}$ and $f_{*}': T_{*}' \to T_{*}'$ is induced by a subdivision, then $f_{*}'$ is an irreducible $\mathcal A$-relative weak representative and $\lambda(f_{*}') = \lambda(f_{*})$.\end{alem} 

The next move is {\it bivalent homotopy}, which occurs at a bivalent vertex and decreases the number of edges. %The bivalent homotopies discussed in the main text of this {\paper} differ from the ones in this appendix in that they do not remove the bivalent vertex --- the homotopy is applied on a neighborhood of a branch point whose image is bivalent until its image is a branch point.

\begin{alem}[{\cite[Lemma~1.13]{BH92}}]\label{movebivalent} If $f_{*}: T_{*} \to T_{*}$ is an irreducible $\mathcal A$-relative weak representative for $\left.\phi\right|_{\mathcal B}$ and $f_{*}'': T_{*}'' \to T_{*}''$ is an irreducible $\mathcal A$-relative weak representative induced by a bivalent homotopy followed by collapse of a maximal invariant subforest with bounded components, then $\lambda(f_{*}'') \le \lambda(f_{*})$.\end{alem}

The last move we need is {\it folding}, which occurs between a pair of oriented edges originating from the same vertex that have the same image under the representative. 

\begin{alem}[{\cite[Lemma~1.15]{BH92}}]\label{movefold} Suppose $f_{*}: T_{*} \to T_{*}$ is an irreducible $\mathcal A$-relative weak representative for $\left.\phi\right|_{\mathcal B}$ and $f_{*}': T_{*}' \to T_{*}'$ is induced by a fold.
If $f_{*}'$ is an $\mathcal A$-relative weak representative, then it is irreducible and $\lambda(f_{*}') = \lambda(f_{*})$. Otherwise, if $f_{*}'': T_{*}'' \to T_{*}''$ is an irreducible $\mathcal A$-relative weak representative induced by a homotopy of $f_{*}'$ that makes the final map locally injective on the interior of edges, followed by collapse of a maximal invariant subforest with bounded components, then $\lambda(f_{*}'') < \lambda(f_{*})$.
\end{alem}

\begin{proof}[Sketch proof of Theorem~\ref{appreltt}] 

Let $\left.\phi\right|_{\mathcal B}$ be irreducible relative to $\mathcal A$ and $f_{*} : T_{*} \to T_{*}$ be an irreducible $\mathcal A$-relative representative. If $\lambda(f_{*})=1$, then $f_*$ is a simplicial embedding with minimal stretch factor and we are done. So we may assume $\lambda(f_{*}) > 1$. 

Suppose for the contrapositive that $f_{*}$ is not an $\mathcal A$-relative train track, then the edge-path $f_{*}^n(e)$ is not immersed for some edge $e$ in $T_{*}$ and integer $n \ge 1$. Let $n$ be the smallest such integer and assume $\star$ is an interior point of an edge $e$ at which $f_{*}^n$ fails to be locally injective. 
We appropriately subdivide $T_{*}$ so that a neighborhood $U$ of $\star$ and its iterates $f_{*}^k(U)$ $(1 \le k \le n)$ satisfy nice properties:
1) $U$ is an interval whose boundary consists of distinct vertices; 2) $f_*^k$ is locally injective on $U$ for $1 \le k < n$; 3) $f_*^n$ folds $U$ at $\star$ to an edge; and 4) $\star \notin f_*^k(U)$ for $1 \le k \le n$.
We can then iteratively fold $f_{*}^{n-1}(U)$, \ldots, $f_{*}^2(U)$, and $f_{*}(U)$. By minimality of $n$, all the folds except the last one induce an irreducible $\mathcal A$-relative weak representative. By the first case of Lemma~\ref{movefold}, this irreducible $\mathcal A$-relative weak representative has the same stretch factor as $f_{*}$. By construction, the last fold induces a map $f_{*}'$ that fails to be an $\mathcal A$-relative weak representative as it fails to be locally injective at $\star$. We can apply a tightening homotopy on $f_{*}'$ to make it locally injective at $\star$, then collapse a maximal invariant subforest with bounded components to get $f_{*}'':T_{*}'' \to T_{*}''$, a minimal $\mathcal A$-relative weak representative for $\left.\phi\right|_{\mathcal B}$. By Lemma~\ref{irredequiv}, the map $f_{*}''$ is irreducible. By the second case of Lemma~\ref{movefold}, the stretch factor is strictly smaller: $\lambda(f_{*}'') < \lambda(f_{*})$. 
We then sequentially apply bivalent homotopies and collapse maximal invariant subforests with bounded components until we get an irreducible $\mathcal A$-relative representative $f_{*}'''$. The stretch factor satisfies $\lambda(f_{*}''') \le \lambda(f_{*}'') < \lambda(f_*)$ by Lemma~\ref{movebivalent}. So $f_{*}$ did not have minimal stretch factor.
\end{proof}

%\nocite{*} % Insert publications even if they are not cited in the poster
\bibliography{zrefs}
\bibliographystyle{plain}

\end{document}